\documentclass[a4paper,11pt]{article}
\usepackage{amsmath,amssymb,amsthm,graphicx}
\usepackage[top=80pt, bottom=90pt, left=70pt, right=70pt]{geometry}
\usepackage{comment}
\usepackage{color}
\usepackage{multirow}
\usepackage[colorlinks=true,linkcolor=blue,citecolor=blue]{hyperref}
\usepackage{cleveref}
\usepackage{kbordermatrix}
\usepackage{blkarray}
\usepackage{cases}
\usepackage{enumitem}
\usepackage{algorithm}
\usepackage[noend]{algpseudocode}

\usepackage{tikz}
\usetikzlibrary{positioning}
\usetikzlibrary{intersections}
\usetikzlibrary{calc}
\usetikzlibrary{arrows.meta}
\usetikzlibrary{shapes.geometric}
\usetikzlibrary{shapes.misc}
\usetikzlibrary {decorations.pathmorphing}

\newcounter{nodecount}
\providecommand{\nodecounter}[1]{
	\setcounter{nodecount}{-1}
	\foreach \iter in #1{
		\stepcounter{nodecount}
	}
}

\DeclareMathOperator{\rank}{rank}
\DeclareMathOperator{\ncrank}{nc-rank}
\DeclareMathOperator{\kerL}{\ker_{\rm L}}
\DeclareMathOperator{\kerR}{\ker_{\rm R}}
\DeclareMathOperator{\proj}{proj}
\newcommand{\F}{\mathbf{F}}
\newcommand{\Q}{\mathbf{Q}}

\newcommand{\A}{A_{\alpha \beta}}
\newcommand{\x}{x_{\alpha \beta}}

\newcommand{\labeltt}{\texttt{label}}

\newcommand{\perpab}{ {\perp_{\alpha \beta}} }
\newcommand{\ainm}{\alpha \in [\mu]}
\newcommand{\binn}{\beta \in [\nu]}
\newcommand{\bit}{\texttt{bit}}
\newcommand{\CL}{C_{\rm L}}
\newcommand{\CR}{C_{\rm R}}

\newcommand{\calP}{\mathcal{P}}
\newcommand{\calQ}{\mathcal{Q}}
\newcommand{\calR}{\mathcal{R}}
\newcommand{\calT}{\mathcal{T}}
\newcommand{\Ni}{{\rm N}_{\rm inner}}
\newcommand{\No}{{\rm N}_{\rm outer}}

\newcommand{\Ai}{{\rm A}_{\rm init}}
\newcommand{\Al}{{\rm A}_{\rm last}}

\newtheorem{thm}{\bfseries Theorem}[section] 

\newtheorem{lem}[thm]{\bfseries Lemma} 
\newtheorem{prop}[thm]{\bfseries Proposition} 
\newtheorem{cor}[thm]{\bfseries Corollary}

\newtheorem*{cl*}{\bfseries Claim}

\theoremstyle{definition}

\crefname{thm}{Theorem}{Theorems}
\crefname{prop}{Proposition}{Propositions}
\crefname{lem}{Lemma}{Lemmas}
\crefname{exmp}{Example}{Examples}
\crefname{cor}{Corollary}{Corollarys}
\crefname{cl}{Claim}{Claims}
\crefname{remark}{Remark}{Remarks}
\crefname{section}{Section}{Sections}

\setdescription{font=\normalfont}

\makeatletter

\@addtoreset{equation}{section}
\makeatother

\usepackage{bm}
\allowdisplaybreaks

\begin{document}
	\title{A combinatorial algorithm for computing the rank of a generic partitioned matrix with $2 \times 2$ submatrices\thanks{A preliminary version of this paper has appeared in the proceedings of the 21st Conference on Integer Programming and Combinatorial Optimization (IPCO 2020).
	This work was done while Yuni Iwamasa was at National Institute of Informatics.}}
	\author{Hiroshi Hirai\thanks{Department of Mathematical Informatics,
			Graduate School of Information Science and Technology,
			The University of Tokyo, Tokyo 113-8656, Japan.
			Email: \texttt{hirai@mist.i.u-tokyo.ac.jp}} \and 
			Yuni Iwamasa\thanks{Department of Communications and Computer Engineering, Graduate School of Informatics, Kyoto University, Kyoto 606-8501, Japan.
			Email: \texttt{iwamasa@i.kyoto-u.ac.jp}}}
	\date{\today}
	\maketitle
	
\begin{abstract}
	In this paper, we consider the problem of computing the rank of a block-structured symbolic matrix (a generic partitioned matrix)
	$A = (\A x_{\alpha \beta})$,
	where $\A$ is a $2 \times 2$ matrix over a field $\F$ and $x_{\alpha \beta}$ is an indeterminate for $\alpha = 1,2,\dots, \mu$ and $\beta = 1,2, \dots, \nu$.
	This problem
	can be viewed as an algebraic generalization of the bipartite matching problem
	and
	was considered by Iwata and Murota (1995).
	Recent interests in this problem lie in the connection with non-commutative Edmonds' problem by Ivanyos, Qiao, and Subrahamanyam (2018)
	and Garg, Gurvits, Oliveiva, and Wigderson (2019),
	where a result by Iwata and Murota implicitly states that the rank and
	non-commutative rank (nc-rank) are the same for this class of symbolic matrices.
	
	The main result of this paper is a simple and combinatorial $O((\mu \nu)^2 \min \{ \mu, \nu \})$-time algorithm for computing the symbolic rank of a $(2 \times 2)$-type generic partitioned matrix of size $2\mu \times 2\nu$.
	Our algorithm is inspired by the Wong sequence algorithm by Ivanyos, Qiao, and Subrahamanyam for the nc-rank of a general symbolic matrix,
	and requires no blow-up operation, 
	no field extension,
	and no additional care for bounding the bit-size.
	Moreover it naturally provides a maximum rank completion of $A$ for an arbitrary field $\F$.
\end{abstract}
\begin{quote}
\end{quote}
	
\section{Introduction}\label{sec:intro}
The maximum matching problem in a bipartite graph $G$ has 
a natural algebraic interpretation;
it amounts to the symbolic rank computation of the matrix $A$
defined by $(A)_{ij} := x_{ij}$ if $ij \in E(G)$
and zero otherwise,
where $x_{ij}$ is a variable for each edge $ij$
and the row and column sets of $A$
are identified with the color classes of $G$.
Such an algebraic interpretation is also known for other matching-type combinatorial optimization problems
such as linear matroid intersection, nonbipartite matching, 
and their generalizations (e.g., linear matroid matching); see~\cite{BSBM/L89,JLMS/T47}.
{\em Edmonds' problem}~\cite{JRNBSS/E67} is 
a general algebraic formulation, which asks 
to compute the rank of a symbolic matrix $A$
represented by
\begin{align}\label{eq:linear A}
	A = A_1 x_1 + A_2 x_2 + \cdots + A_k x_k.
\end{align} 
Here $A_i$ is a matrix over a field $\F$ and $x_i$ is a variable for $i = 1,2,\dots,k$.
Although a randomized polynomial-time algorithm for Edmonds' problem
is known (if $|\F|$ is large)~\cite{FCT/L79,JACM/S80},
a deterministic polynomial-time algorithm is not known,
which is one of the prominent open problems in theoretical computer science (see e.g.,~\cite{CC/KI04}).
Known polynomial-time algorithms for the above-mentioned matching-type problems
can be viewed as solutions for special Edmonds' problems.

The present article addresses the rank computation (Edmonds' problem) of 
a matrix of the following block-matrix structure
\begin{align}\label{eq:A}
	A =
	\begin{pmatrix}
		A_{11} x_{11} & A_{12} x_{12} & \cdots & A_{1 \nu} x_{1 \nu}\\
		A_{21} x_{21} & A_{22} x_{22} & \cdots & A_{2\nu} x_{2 \nu}\\
		\vdots & \vdots & \ddots & \vdots \\
		A_{\mu 1} x_{\mu 1} & A_{\mu 2} x_{\mu 2} & \cdots & A_{\mu \nu} x_{\mu \nu}
	\end{pmatrix},
\end{align}
where $\A$ is a $2 \times 2$ matrix over a field $\F$
and
$x_{\alpha \beta}$ is a variable for $\alpha = 1,2,\dots, \mu$ and $\beta = 1,2,\dots, \nu$.
Recall that bipartite matching is precisely the case where each $\A$ is a $1 \times  1$ matrix.
Such matrices, which we call {\em $(2 \times 2)$-type generic partitioned matrices}, 
were considered in detail by Iwata and Murota~\cite{SIMAA/IM95}, 
subsequent to the study on partitioned matrices of general type~\cite{SIMAA/IIM94}.
They established a min-max formula (i.e., good characterization)
for the rank of this class of matrices,
which involves
the minimization of a submodular function 
on the lattice of vector subspaces. 
A combinatorial polynomial-time rank computation is not known and desired.
Our main result solves this issue.
\begin{thm}\label{thm:main}
	There exists a combinatorial $O((\mu\nu)^2 \min\{ \mu, \nu \})$-time
	algorithm for Edmonds' problem for a $(2 \times 2)$-type generic partitioned matrix of the form~\eqref{eq:A}.
\end{thm}

This result links the recent development of Edmonds' problem in a noncommutative setting.
The {\em noncommutative Edmonds' problem}~\cite{CC/IQS17} asks to compute 
the rank of a matrix of the form~\eqref{eq:linear A},
where 
$x_i$ and $x_j$ are supposed to be noncommutative, 
i.e., $x_ix_j \neq x_jx_i$.
In this setting, the rank concept can be defined 
(via the free skew field or the inner rank of a matrix over a ring)
and is called {\em noncommutative rank} or {\em nc-rank}.
Nc-rank is an upper bound of (usual) rank. 
Surprisingly, the nc-rank can be computed in deterministic polynomial time. 
Algorithms to do this were proposed by Garg, Gurvits, Oliveira, and Wigderson~\cite{FCT/GGOW20}
for the case of $\F = \Q$
and by Ivanyos, Qiao, and Subrahmanyam~\cite{CC/IQS18} for an arbitrary field.
The former algorithm ({\em operator scaling}) 
is an analytical algorithm motivated by quantum information theory.
The latter, which we call the {\em IQS-algorithm}, 
is an augmenting-path type algorithm.
It 
utilizes a {\em Wong sequence}~\cite{JCSS/IKQS15}
---a vector-space analogue of 
an alternating walk--- 
and the formula of nc-rank earlier proved by Fortin and Reutenauer~\cite{SLC/FR04}.
In fact, for $(2 \times 2)$-type generic partitioned matrices, 
the rank formula proved by Iwata and Murota 
is essentially the same as the nc-rank formula by Fortin and Reutenauer.
This means that the rank and nc-rank are the same for this class of matrices and that the polynomial solvability follows from these results.
Thus, the real contribution of this paper is in the
term ``combinatorial'' in the theorem, which is explained as follows.

Our proposed algorithm is viewed as a combinatorial enhancement of the IQS-algorithm 
for $(2 \times 2)$-type generic partitioned matrices.
As mentioned, the IQS-algorithm is 
an augmenting-path type algorithm: 
Given a {\em substitution} $\tilde A$ obtained from $A$ 
by substituting a value in $\F$ to each $x_i$, 
construct the Wong sequence for $(A,\tilde A)$, 
which is an analogue of augmenting path search in the auxiliary graph.
If an augmenting path exists,
then one can find another substitution $\tilde{A}'$ 
with $\rank \tilde{A}' > \rank \tilde{A}$ and repeat it with updating $\tilde{A} \leftarrow \tilde{A}'$.
If an augmenting path does not exist,
then one obtains a certificate of optimality of the nc-rank formula. 
Here, for reaching $\rank \tilde{A} = \ncrank A$, 
the algorithm conducts the {\em blow-up} operation,
which
replaces $A$ with a larger matrix $A^{(d)} := \sum_{i=1}^m A_i \otimes X_i$ for $d \times d$ matrices $X_1,X_2,\ldots,X_k$ of variable entries.
It is known that $\ncrank A$ is equal to 
$1/d$ times the rank of a substitution $\tilde A^{(d)}$ for some $d$ (if $|\F|$ is large).
The blow-up steps (and field extensions of $\F$) make the algorithm considerably complicated and slow.
For our special case, the rank and nc-rank are equal, 
and therefore a {\em blow-up-free} algorithm is expected and desirable; 
a naive application of the IQS-algorithm cannot avoid the blow-up.

Our algorithm is the first blow-up-free algorithm that can solve Edmonds' problem
for the class of $(2 \times 2)$-type generic partitioned matrices.
The key concept that we introduce in this paper is a {\em matching}
in this setting.
It is actually a $2$-matching in the graph consisting of edges $\alpha \beta$
with nonzero $A_{\alpha \beta}$, 
which inherits the $2 \times 2$ structure of our matrix $A$ and 
provides a canonical substitution of $A$.
Incorporating the idea of Wong sequence, we introduce
an {\it augmenting space-walk} for a matching. 
Then our algorithm continues analogously to the augmenting path algorithm, as expected.
It requires no blow-up operation, 
no field extension,
and no additional care for bounding the bit size.
Moreover it naturally provides a maximum rank completion (substitution) $\tilde A$ of $A$ for an arbitrary field $\F$.
This reveals that
the maximum rank completion problem for a $(2 \times 2)$-type generic partitioned matrix is polynomially solvable for arbitrary $\F$,
while this problem is known to be NP-hard in general~\cite{JCSS/BFS99}.

\paragraph{Related work.}
It is an interesting research direction to construct a polynomial-time
blow-up-free algorithm
for general matrices $A$ with $\rank A = \ncrank A$.
Such an algorithm can decide whether rank and nc-rank are equal,
which leads to
a solution of (commutative) Edmonds' problem.
Indeed,
such an algorithm
naturally provides polynomial-time solvability of the full-rank decision version of Edmonds' problem.
Note that the family of linear independent column sets of $A$ forms a matroid, and the full-rank decision version of Edmonds' problem can play a role as an independence oracle of the matroid.
Thus, by a greedy algorithm for the matroid, we can obtain a basis in polynomial time; the size of the basis is equal to rank $A$.

A representative example of a matrix $A$ with $\rank A = \ncrank A$ is a matrix
such that each $A_i$ in~\eqref{eq:linear A} is a rank-1 matrix.
In this case,
the rank computation is equivalent 
to the linear matroid intersection problem~\cite{BSBM/L89}.
Edmonds' matroid intersection algorithm becomes obviously blow-up-free.
In fact, it can naturally be interpreted as the Wong sequence~\cite{bachelor/I18}.
Ivanyos, Karpinski, Qiao, and Santha~\cite{JCSS/IKQS15}
gave a Wong-sequence-based blow-up-free algorithm for 
matrices $A$ having an ``implicit'' rank-1 expression, 
that is,  $A$ becomes rank-1 summands
by some (unknown) linear transformation of variables.

Computation of nc-rank is formulated 
as submodular function minimization on 
the modular lattice of vector subspaces. 
Based on this, Hamada and Hirai~\cite{arxiv/HH17,SIAGA/HH21} 
developed a conceptually different algorithm from~\cite{FCT/GGOW20} and~\cite{CC/IQS18}.  
Via an analogue of the Lov\'asz extension, 
they solved the problem as a geodesically-convex optimization on a CAT(0)-space.
For the case of a $(2 \times 2)$-type generic partitioned matrix, 
the submodular function is defined on the direct product of modular lattices of rank-2 with infinite size. 
Following a pioneering work by Kuivinen~\cite{DO/K11},
Fujishige, Kir\'aly, Makino, Takazawa, and Tanigawa~\cite{EGRES/FKMTT14} 
demonstrated the oracle tractability of submodular function minimization on {\em diamond}, which is 
a direct product of modular lattices of rank-2 with ``finite'' size. 

A weighted analogue of Edmonds' problem is computation of 
the degree of the determinant of a matrix of the form~\eqref{eq:linear A}
such that 
each $A_i = A_i(t)$ is a polynomial matrix over $t$.
This algebraically abstracts the weighted versions
of combinatorial optimization problems, 
such as the maximum-weight bipartite matching problem. 
Its noncommutative extension was studied in~\cite{SIAGA/H19} (see also~\cite{ICALP/O20}).
It may also be interesting to extend our results
to such a weighted version.

\paragraph{Organization.}
The remainder of this paper is organized as follows.
In \cref{sec:matching},
we introduce the concept of matching for a $(2 \times 2)$-type generic partitioned matrix,
which is an algebraic generalization of bipartite matching.
We also provide a combinatorial and algebraic characterization of a matching
and its useful properties.
The formal description of our proposed algorithm is given in \cref{sec:algorithm,sec:augmentation}.
We introduce an {\it augmenting space-walk} for a matching (\cref{subsec:augmenting space-walk}),
and develop an algorithm for finding an augmenting space-walk (\cref{subsec:finding})
and an augmentation algorithm (\cref{sec:augmentation}).
In \cref{sec:discussion},
we present concluding remarks.

\section{Matching}\label{sec:matching}
In this section,
we introduce the concept of {\it matching},
which plays a central role in devising our algorithm.
We give an algebraic and combinatorial characterization of a matching.
This provides several nontrivial properties of matchings,
which will be used in our algorithm.

\subsection{Notations}
For a positive integer $k$,
we denote $\{1,2,\dots,k\}$ by $[k]$.
A $p \times q$ matrix $B$ over a field $\F$
is regarded
as the bilinear map defined by
$B(u, v) := u^\top Bv$
for $u \in \F^p$ and $v \in \F^q$.
We denote by $\kerL(B)$ and $\kerR(B)$
the left and right kernels of $B$,
respectively.

Let $A$ be
a $(2 \times 2)$-type generic partitioned matrix $A$ of the form~\eqref{eq:A}.
The matrix $A$ is regarded as
a matrix over the field $\F(x)$ of rational functions with variables $\x$ for $\ainm$ and $\binn$.
Symbols $\alpha$, $\beta$, and $\gamma$ 
are used to represent elements of $[\mu]$, $[\nu]$, and $[\mu] \sqcup [\nu]$, respectively,
where $\sqcup$ denotes the direct sum.
We often drop ``$\in [\mu]$'' from the notation of ``$\alpha \in [\mu]$'' if it is clear from the context.
For each $\alpha$ and $\beta$, 
consider 2-dimensional $\F$-vector spaces $U_{\alpha} =\F^2$ and $V_{\beta} = \F^2$. 
Each submatrix $A_{\alpha \beta}$ is considered 
as a bilinear form $U_{\alpha} \times V_{\beta} \to \F$.
Let $U := \F^{2\mu} = \bigoplus_{\alpha} U_\alpha$ 
and $V := \F^{2\nu} = \bigoplus_{\beta} V_\beta$.

We define the (undirected) bipartite graph $G := ([\mu], [\nu]; E)$
by $E := \{ \alpha \beta \mid \A \neq O \}$.
For $I \subseteq E$,
let $A_I$ denote
the matrix obtained from $A$ by replacing  
each submatrix $\A$ with $\alpha \beta \not\in I$ by the $2 \times 2$ zero matrix.
An edge $\alpha \beta \in E$ is said to be {\it rank-$k$} $(k=1,2)$
if $\rank \A = k$.
For notational simplicity, 
the subgraph $([\mu], [\nu]; I)$ for $I \subseteq E$
is also denoted by~$I$. 
For a node $\gamma$, let $\deg_I(\gamma)$ denote the degree of $\gamma$ in $I$, i.e.,
the number of edges in $I$ incident to $\gamma$. 
A connected component of $I$ is said to be {\it rank-1}
if it contains a rank-1 edge.
An edge $\alpha \beta \in I$ is said to be {\it isolated}
if $\deg_I(\alpha) = \deg_I(\beta) = 1$.

\subsection{Definition and characterization}\label{subsec:def and chara}
An edge subset $I \subseteq E$
is called a {\it matching}
if
\begin{align}\label{eq:def matching}
	\rank A_I > \rank A_{I \setminus \{ \alpha \beta \}}
\end{align}
holds
for any $\alpha \beta \in I$.
A matching $I$ is said to be {\it maximum}
if $\rank A_I \geq \rank A_{I'}$ for any matching $I'$,
or equivalently, if $\rank A = \rank A_I$ holds.
This matching concept generalizes bipartite matching;
for a matrix $A$ of the form~\eqref{eq:A} with $1 \times 1$ blocks $\A$
and the corresponding bipartite graph $G$ to $A$,
an edge subset satisfies~\eqref{eq:def matching} if and only if it is a bipartite matching of $G$.

\cref{thm:chara} below provides a characterization of a matching $I$
and a simple combinatorial rank formula of $A_I$.
The characterization consists of the following four conditions
for an edge subset $I \subseteq E$:
\begin{description}
	\item[(Deg)] $\deg_I(\gamma) \leq 2$ for each node $\gamma$ of $G$.
\end{description}
Suppose that $I$ satisfies this condition. Then each connected component of $I$ forms a path or a cycle.
Thus $I$ is 2-edge-colorable, namely,
there are two edge classes such that any two incident edges are in different classes.
An edge in one color class is called a {\it $+$-edge},
and an edge in the other color class is called a {\it $-$-edge}.
For a path component $C$ of $I$,
an {\it end edge}
is an edge $\alpha \beta \in C$ with $\deg_I(\alpha) = 1$ or $\deg_I(\beta) = 1$.
\begin{description}
	\item[(Path)]
	For each non-isolated path component of $I$,
	both the end edges are rank-1.
	\item[(Cycle)]
	Each cycle component of $I$
	has both a rank-1 $+$-edge and rank-1 $-$-edge.
\end{description}

A {\em valid labeling} for $I$ is a node-labeling that assigns
two distinct 1-dimensional subspaces to each node,
$U_{\alpha}^+, U_{\alpha}^- \subseteq U_{\alpha}$
for $\alpha$
and $V_{\beta}^+, V_{\beta}^- \subseteq V_{\beta}$ for $\beta$,
such that
for each edge $\alpha \beta \in I$, it holds that
\begin{align}
	&A_{\alpha \beta}(U_{\alpha}^+, V_{\beta}^-) = A_{\alpha
		\beta}(U_{\alpha}^-, V_{\beta}^+) = \{0\},\label{eq:+-}\\
	&(\kerL(A_{\alpha \beta}), \kerR(A_{\alpha \beta})) =
	\begin{cases}
		(U_{\alpha}^+, V_{\beta}^+) & \text{if $\alpha \beta$ is a rank-1 $+$-edge},\\
		(U_{\alpha}^-, V_{\beta}^-) & \text{if $\alpha \beta$ is a rank-1 $-$-edge}.
	\end{cases}\label{eq:++ --}
\end{align}
\begin{description}
	\item[(VL)]
	$I$ has a valid labeling.
\end{description}
By~\eqref{eq:+-} and~\eqref{eq:++ --},
one of the labels $U_\alpha^+$ and $U_\alpha^-$ ($V_\beta^+$ and $V_\beta^-$)
is uniquely determined from the label of a rank-1 end edge
along the path to it.
Therefore,
if $I$ satisfies (Deg), (Path), and (Cycle),
then two labels on any vertex with degree 2
are uniquely determined.
(VL) requires that they are different.

For the rank formula,
we define
\begin{align*}
	r(I) := |I| + \textrm{the number of isolated rank-2 edges in $I$}.
\end{align*}
We are now ready to describe the characterization and rank formula;
the proof is given in \cref{subsubsec:thm:chara}.
\begin{thm}\label{thm:chara}
	An edge subset $I$ is a matching
	if and only if
	$I$ satisfies
	{\rm (Deg)}, {\rm (Path)}, {\rm (Cycle)}, and {\rm (VL)}.
	In addition,
	for a matching $I$, it holds that $\rank A_I = r(I)$.
\end{thm}
Now Edmonds' problem for $A$
is equivalent to the problem of obtaining a maximum matching for $A$,
since, for a matching $I$,
it holds that $\rank A_I = \rank A = r(I)$
by \cref{thm:chara}.
Our proposed algorithm (described in \cref{sec:algorithm,sec:augmentation}) finds a maximum matching of $A$ in $O((\mu\nu)^2 \min\{ \mu, \nu \})$ time.

\cref{thm:chara} has several consequences.
It can be immediately seen from \cref{thm:chara} that
a matching has a good characterization.
That is,
one can decide if a given edge subset $I$ is a matching in polynomial time.

By combining the rank formula in \cref{thm:chara} with Iwata--Murota's minimax formula~\cite{SIMAA/IM95} 
(or Fortin--Reutenauer's one~\cite{SLC/FR04} with $\rank A = \ncrank A$), 
we obtain the following combinatorial and algebraic minimax theorem between matchings and vector spaces,
which is used for the validity of the optimality of a matching.
\begin{cor}\label{cor:min-max}
	\begin{align*}
		\displaystyle
		\max \{ r(I) \mid I : \text{matching} \} = \min \left\{ 2\mu + 2\nu - \sum_{\alpha} \dim X_\alpha - \sum_{\beta} \dim Y_\beta \right\},
	\end{align*}
	where the minimum is taken over
	all vector spaces $X_\alpha \subseteq U_\alpha$ and $Y_\beta \subseteq V_\beta$
	such that $\A(X_\alpha, Y_\beta) = \{0\}$ for $\alpha$ and $\beta$.
\end{cor}
An {\it optimality witness} of $I$
is a tuple of vector spaces $X_\alpha$ and $Y_\beta$ for $\alpha , \beta$
satisfying $\A(X_\alpha, Y_\beta) = \{0\}$ and $r(I) = 2\mu + 2\nu - \sum_{\alpha} \dim X_\alpha - \sum_{\beta} \dim Y_\beta$.
\cref{thm:chara} and \cref{cor:min-max} verify that
$I$ is maximum if
there exists an optimality witness for $I$.

For a matching $I$,
the {\it canonical substitution} $\tilde{A}_I$ of $A_I$
is the matrix over $\F$ obtained from assigning $1$ to $\x$ for $\alpha\beta \in I$
and $0$ to $\x$ for $\alpha\beta \not\in I$.
Then the following holds; the proof is given in \cref{subsubsec:prop:substitution}.
\begin{prop}\label{prop:substitution}
	For a matching $I$,
	it holds that $\rank A_I = \rank \tilde{A}_I$.
\end{prop}
By \cref{prop:substitution},
for an arbitrary ground field $\F$,
$\rank \tilde{A}_I = \rank A$ holds for a maximum matching $I$.
Since
our algorithm
outputs a maximum matching in $O((\mu\nu)^2 \min\{ \mu, \nu \})$ time,
we obtain a maximum rank substitution in the same time.
While the {\it maximum rank completion problem},
the problem of computing the maximum rank substitution of a
symbolic matrix of the form~\eqref{eq:linear A},
is NP-hard in general if $|\F|$ is small~\cite{JCSS/BFS99},
the class of $(2 \times 2)$-type generic partitioned matrices
constitutes a new tractable class of the maximum rank completion problem for arbitrary $\F$.
\begin{thm}
	The maximum rank completion problem for a $(2 \times 2)$-type generic partitioned matrix of the form~\eqref{eq:A}
	can be solved in $O((\mu\nu)^2 \min\{ \mu, \nu \})$ time.
\end{thm}

From the rank formula
in \cref{thm:chara},
we obtain an explicit expression of
the left/right kernel of the canonical substitution $\tilde{A}_I$.
For a matching $I$,
define
\begin{align}\label{eq:kera}
	\ker_I(\alpha) :=
	\begin{cases}
		U_\alpha & \text{if $\deg_I(\alpha) = 0$},\\
		\kerL(\A) & \text{if $\alpha$ is incident only to one rank-1 edge $\alpha \beta$ in $I$},\\
		\{ 0 \} & \text{otherwise}.
	\end{cases}
\end{align}
Also let $\ker_I(\beta)$ be the vector subspace of $V_\beta$ 
such that ``L'', $\alpha$, and $U$ in~\eqref{eq:kera} are replaced with ``R'', $\beta$, and $V$,
respectively.
\begin{cor}\label{cor:ker}
	If $I$ is a matching,
	then $\kerL(\tilde{A}_I) = \bigoplus_\alpha \ker_I(\alpha)$ and
	$\kerR(\tilde{A}_I) = \bigoplus_\beta \ker_I(\beta)$.
\end{cor}
\begin{proof}
	We only show $\kerL(\tilde{A}_I) = \bigoplus_\alpha \ker_I(\alpha)$.
	It is clear that $\kerL(\tilde{A}_I) \supseteq \bigoplus_{\alpha} \ker_I(\alpha)$.
	Hence it suffices to prove $\dim(\kerL(\tilde{A}_I)) = \sum_{\alpha}\dim(\ker_I(\alpha))$.
	Let $\kappa$ be the number of isolated rank-2 edges in $I$.
	By the definition of $\ker_I(\alpha)$,
	we have $\sum_{\alpha}\dim(\ker_I(\alpha)) = 2\mu - 2\kappa - (|I| - \kappa) = 2\mu - r(I)$.
	On the other hand,
	it holds that $\dim(\kerL(\tilde{A}_I)) = 2\mu - \rank \tilde{A}_I = 2\mu - r(I)$ by \cref{thm:chara} and \cref{prop:substitution}.
	Thus we have $\dim(\kerL(\tilde{A}_I)) = \sum_{\alpha}\dim(\ker_I(\alpha))$.
\end{proof}

\subsection{Proofs}\label{subsec:proof}
For a row subset $M$ and column subset $N$ of $A$,
we denote by $A[M, N]$ the submatrix of $A$ with row set $M$ and column set $N$.
Let $\mathcal{C}$ be the set of non-isolated connected components.
Note that any component in $\mathcal{C}$ is rank-1.
For $C \in \mathcal{C}$,
we denote by $\CL$ (resp. $\CR$) the set of $\alpha$ (resp. $\beta$) belonging to $C$.
In this proof,
$A_C$ is also regarded as $A_C[L, R]$ for simplicity,
where $L$ (resp. $R$) is the row subset (resp. column subset)
corresponding to $\CL$ (resp. $\CR$).
That is, $A_C$ is a $2|\CL| \times 2|\CR|$ matrix obtained by the deletion of all zero rows (resp. columns) corresponding to $\alpha \not\in \CL$ (resp. $\beta \not\in \CR$)
from $A$.

\subsubsection{Proof of \cref{thm:chara}}\label{subsubsec:thm:chara}
(If part).
Suppose that $I$ satisfies (Deg), (Path), (Cycle), and (VL).
We show $\rank A_I > \rank A_{I \setminus \{\alpha \beta\}}$ for every $\alpha \beta \in I$.
By (Deg),
each connected component of $I$ forms a path or a cycle.
It is clear that
\begin{align}\label{eq:rank}
\rank A_I = \sum \{ \rank \A \mid \alpha \beta \in I : \text{$\alpha \beta$ is isolated} \} + \sum_{C \in \mathcal{C}} \rank A_C.
\end{align}
If $\alpha \beta$ is isolated in $I$,
then $\rank A_{I \setminus \{\alpha \beta\}} = \rank A_I - \rank \A < \rank A_I$.
Thus,
in the following,
we prove $\rank A_C > \rank A_{C \setminus \{\alpha \beta\}}$ for each $C \in \mathcal{C}$ and $\alpha \beta \in C$.
We only consider the case where $C$ is a cycle component in $I$;
the argument for a path component in $I$ is similar.
Suppose that $C$ consists of $+$-edges $\alpha_1 \beta_1, \alpha_2 \beta_2, \dots, \alpha_k \beta_k$
and $-$-edges $\beta_1 \alpha_2, \beta_2 \alpha_3, \dots, \beta_k \alpha_1$.
Choose a valid labeling $U_\alpha^+, U_\alpha^-, V_\beta^+, V_\beta^-$ for $I$.
Take nonzero vectors $u_\alpha^+ \in U_\alpha^+$, $u_\alpha^- \in U_\alpha^-$, $v_\beta^+ \in V_\beta^+$, and $v_\beta^- \in V_\beta^-$ for each $\alpha$ and $\beta$.
By $U_\alpha^+ \neq U_\alpha^-$ and $V_\beta^+ \neq V_\beta^-$,
the $2 \times 2$ matrices
$S_\alpha :=
\left[
\begin{array}{c}
u_\alpha^+\\
u_\alpha^-
\end{array}
\right]$
and
$T_\beta :=
\begin{bmatrix}
v_\beta^+ &
v_\beta^-
\end{bmatrix}$
are both nonsingular.
By the conditions~\eqref{eq:+-} and~\eqref{eq:++ --},
it holds that
\begin{numcases}
{S_\alpha \A T_\beta =}
\kbordermatrix{
	& v_\beta^+ & v_\beta^- \\
	u_\alpha^+ & \bullet & 0 \\
	u_\alpha^- & 0 & \bullet
}
& if $\alpha\beta$ is rank-2,\label{eq:rank 2}\\
\kbordermatrix{
	& v_\beta^+ & v_\beta^- \\
	u_\alpha^+ & 0 & 0 \\
	u_\alpha^- & 0 & \bullet
}
& if $\alpha\beta$ is a rank-1 $+$-edge\label{eq:rank 1 +},\\
\kbordermatrix{
	& v_\beta^+ & v_\beta^- \\
	u_\alpha^+ & \bullet & 0 \\
	u_\alpha^- & 0 & 0
}
& if $\alpha\beta$ is a rank-1 $-$-edge\label{eq:rank 1 -}
\end{numcases}
for each $\alpha\beta \in C$,
where $\bullet$ represents some nonzero element in $\F$.
Let $S$ and $T$ be the block-diagonal matrices with diagonal blocks $S_{\alpha_1}, S_{\alpha_2}, \dots, S_{\alpha_k}$
and $T_{\beta_1}, T_{\beta_2}, \dots, T_{\beta_k}$,
respectively.
By~\eqref{eq:rank 2}--\eqref{eq:rank 1 -},
we obtain
\begin{align}\label{eq:EAF}
SA_CT =
\begin{blockarray}{ccccccccccc}
& v_{\beta_1}^+ & v_{\beta_2}^+ & \cdots & v_{\beta_{k-1}}^+ & v_{\beta_k}^+  & v_{\beta_1}^- & v_{\beta_2}^- & \cdots & v_{\beta_{k-1}}^- & v_{\beta_k}^- \\
\begin{block}{c[ccccc|ccccc]}
u_{\alpha_1}^+ & \ast & & & & \bullet & & & & & \\
u_{\alpha_2}^+ & \bullet & \ast & & & & & & & &\\
\vdots &  & \ddots & \ddots & & & & & & &\\
u_{\alpha_{k-1}}^+ & & & \bullet & \ast & & & & & &\\
u_{\alpha_k}^+ & & & & \bullet & \ast & & & & &\\
\cline{2-11}
u_{\alpha_1}^- & & & & & & \bullet & & & & \ast\\
u_{\alpha_2}^- & & & & & & \ast & \bullet & & &\\
\vdots & & & & & & & \ddots & \ddots & &\\
u_{\alpha_{k-1}}^- & & & & & & & & \ast & \bullet & \\
u_{\alpha_k}^- & & & & & & & & & \ast & \bullet \\
\end{block}
\end{blockarray},
\end{align}
where $\bullet$ represents some nonzero element in $\F(x)$
and $\ast$ can be a zero/nonzero element;
$\ast$ is nonzero if and only if the corresponding edge is rank-2.
Let $A_C^+$ and $A_C^-$ denote the submatrices of $SA_CT$
with the $+$ indices and $-$ indices,
respectively.
Then it holds that $\rank A_C = \rank SA_CT = \rank A_C^+ + \rank A_C^-$.

If the $u_{\alpha_i}^+ v_{\beta_j}^+$-th entry of $A_C^+$ is nonzero,
then it is of the form $a x_{\alpha_i \beta_j}$ with some nonzero $a \in \F$.
This implies that all nonzero entries in $A_C^+$ have different variables.
Since all entries represented by bullets in~\eqref{eq:EAF} are nonzero,
it holds $\rank A_C^+ = k$.
By a similar argument,
it also holds $\rank A_C^- = k$.
Thus we obtain $\rank A_C = 2k = |C|$.

On the one hand,
by (Cycle),
one of the $+$-edges
is rank-1,
which implies that one of the asterisks in $A_C^+$ is zero.
Hence
we have $\rank A_C^+ > \rank A_{C \setminus \{\alpha \beta\}}^+$ for any $-$-edge $\alpha \beta$ in $C$.
Similarly $\rank A_C^- > \rank A_{C \setminus \{\alpha \beta\}}^-$ holds for any $+$-edge $\alpha \beta$ in $C$.
Thus we obtain $\rank A_C > \rank A_{C \setminus \{\alpha\beta\}}$ for all $\alpha\beta \in C$.

(Only-if part and the rank formula).
Suppose that $I$ is a matching.
We first show that $I$ satisfies (Deg), (Path), and (Cycle) by contraposition.

(Deg).
We only consider the case of $\deg_I(\alpha) \geq 3$
for some $\alpha$;
the case of $\deg_I(\beta) \geq 3$ is similar.
Let $\beta_1, \beta_2, \beta_3$ be distinct nodes with $\alpha \beta_1, \alpha \beta_2, \alpha \beta_3 \in I$.
There are a row subset $M$ and a column subset $N$ such that $|M| = |N| = \rank A_I$ and $\det A_I[M,N] \neq 0$.
Take any monomial of the expansion of $\det A_I[M,N]$.
Since the row size of $\A$ is 2,
at least one of the variables $x_{\alpha \beta_1}$, $x_{\alpha \beta_2}$, and $x_{\alpha \beta_3}$, say $x_{\alpha \beta_1}$, does not appear in the monomial.
Hence it holds that $\det A_{I \setminus \{\alpha \beta_1\}}[M,N] \neq 0$.
This means $\rank A_{I \setminus \{\alpha \beta_1\}} = \rank A_I$.
Thus $I$ is not a matching.

(Path).
Let $C \subseteq I$ be a non-isolated path component.
If an end edge $e$ of $C$ is rank-2,
then $\rank A_C = \rank A_{C \setminus \{\alpha \beta\}}$,
where $\alpha \beta$ is the edge incident to $e$ in $C$.
Indeed, $A_{C \setminus \{\alpha \beta\}}$ can be obtained from $A_C$ via elementary row or column operations using $A_e$.
This implies that $I$ is not a matching.

(Cycle).
Let $C \subseteq I$ be a cycle component.
Suppose that all $+$-edges in $C$ are rank-2.
Let $C^-$ be the subset of $C$ consisting of $-$-edges.
Then it holds that $\rank A_C = \rank A_{C \setminus C^-}$.
Indeed, $\rank A_{C \setminus C^-} \leq \rank A_C$ immediately follows.
On the other hand,
we have $\rank A_C \leq |C| = \rank A_{C \setminus C^-}$
since $A_C$ is a $|C| \times |C|$ matrix and $A_{C \setminus C^-}$ is a block-diagonal matrix with rank-2 diagonal blocks $\A$.
This implies that $I$ is not a matching.

(The rank formula: $\rank A_I = r(I)$).
Recall that $\CL$ (resp. $\CR$) is the set of $\alpha$ (resp. $\beta$) belonging to $C$.
It suffices to show $\rank A_C = |C|$ for each $C \in \mathcal{C}$ by~\eqref{eq:rank}.
If $C$ is a cycle component or a path component with even length,
the inequality $\rank A_C \leq  \min \{ 2|\CL|, 2|\CR| \} = |C|$ immediately follows.
For a path component $C$ with odd length,
the square matrix $A_C$ has nontrivial kernel by (Path).
Hence $\rank A_C < 2|\CL| (= 2|\CR|)$.
Thus, by $|C| = 2|\CL| - 1$, we obtain $\rank A_C \leq |C|$.
On the other hand,
let $M$ and $N$ be index sets such that
$|M| = |N| = \rank A_C$ and $\rank A_C[M, N] = \rank A_C$.
Since $I$ is a matching,
$C$ is also a matching.
Hence,
if we substitute $0$ for $x_{\alpha \beta}$ in $\det A_C[M, N]$,
then it becomes $0$.
By the factor theorem,
the polynomial $\det A_C[M, N]$ can be divided by $x_{\alpha \beta}$ for all $\alpha \beta \in C$,
that is,
$\det A_C[M, N] = c \cdot \prod_{\alpha \beta \in C} x_{\alpha \beta}$ for some nonzero polynomial $c \in \F[x]$.
This implies $\rank A_C \geq |C|$.

(VL).
We can assume that $I$ satisfies (Deg), (Path), and (Cycle).
By the argument for the rank formula,
it suffices to show that, if $I$ violates (VL),
then $\rank A_C < |C|$ for some $C \in \mathcal{C}$.
In the proof for (VL),
a vector space $Z$ over $\F$ is also regarded as one over $\F(x)$
by the scalar extension $\F(x) \otimes Z$.
Note that $\A(X, Y) = \{0\}$ if and only if $\A \x (X, Y) = \{0\}$,
in which the former $X,Y$ are vector spaces over $\F$
and the latter $X,Y$ are over $\F(x)$.

For each $C \in \mathcal{C}$ and nodes $\alpha, \beta$ in $C$,
let $U_\alpha^+, U_\alpha^- \subseteq U_\alpha$
and $V_\beta^+, V_\beta^- \subseteq V_\beta$ be 1-dimensional vector spaces
satisfying~\eqref{eq:+-} and~\eqref{eq:++ --}.
Such spaces for $\gamma$ with $\deg_I(\gamma) = 2$
are uniquely determined.
For a node $\alpha$ with $\deg_I(\alpha) = 1$ that is incident to a rank-1 $+$-edge $\alpha \beta$,
$U_\alpha^+$ is uniquely determined as $\kerL(\A)$ by~\eqref{eq:+-};
set $U_\alpha^- \neq U_\alpha^+$.
Similar arguments hold for the other cases.

Suppose that $I$ violates (VL).
Then there is a vertex $\gamma$ with $\deg_I(\gamma) = 2$
such that its $+$-space and $-$-space coincide;
we call such $\gamma$ {\it degenerate}.
Let $C \in \mathcal{C}$ be a connected component
containing a degenerate node.
In the following,
we construct vector spaces $X \subseteq \bigoplus_{\alpha \in \CL} U_\alpha$ and $Y \subseteq \bigoplus_{\beta \in \CR} V_\beta$
such that $A_C(X, Y) = \{0\}$ and $(2|\CL| + 2|\CR|) - (\dim X + \dim Y) < |C|$,
implying $\rank A_C \leq (2|\CL| + 2|\CR|) - (\dim X + \dim Y) < |C|$.

We first consider the case where $C$ is a path component with odd length;
the case where $C$ is a path component with even length is similar.
Suppose $C = (\alpha_1 \beta_1, \beta_1 \alpha_2, \dots, \alpha_k \beta_k)$.
Without loss of generality,
assume that $\alpha_1 \beta_1$ and $\alpha_k \beta_k$ are $+$-edges
and $\alpha_p$ with $1 < p \leq k$
is degenerate.
Let $X_{\alpha_1} := U_{\alpha_1}$,
$X_{\alpha_i} := U_{\alpha_i}^-$ for $1 < i \leq p$,
and $X_{\alpha_j} := U_{\alpha_j}^+$ for $p < j \leq k$.
Similarly,
let $Y_{\beta_i} := V_{\beta_i}^+$ for $1 \leq i < p$,
$Y_{\beta_j} := V_{\beta_j}^-$ for $p \leq j < k$,
and $Y_{\beta_k} := V_{\beta_k}$.
Namely,
the end vertices are associated with 2-dimensional vector spaces,
and the other vertices have 1-dimensional vector spaces.
Define $X := \bigoplus_{\alpha \in \CL} X_\alpha$ and $Y := \bigoplus_{\beta \in \CR} Y_\beta$.
Then one can see $A_C(X, Y) = \{0\}$ and $(2|\CL| + 2|\CR|) - (\dim X + \dim Y) < |C|$.
Indeed, the latter is clear.
The former follows from
the conditions~\eqref{eq:+-} and~\eqref{eq:++ --} of $U_\alpha^+, U_\alpha^-, V_\beta^+, V_\beta^-$ and $U_{\alpha_p}^+ = U_{\alpha_p}^-$.

We next consider the case where $C$
is a cycle component.
By (Cycle),
there are a rank-1 $+$-edge $\alpha_1 \beta_1$,
and a rank-1 $-$-edge $\beta_k \alpha_{k+1}$.
Suppose that
$C$ is a cycle of $\alpha_1 \beta_1$, a $\beta_1$-$\beta_k$ path,
$\beta_k \alpha_{k+1}$,
and an $\alpha_{k+1}$-$\alpha_1$ path.
We may
assume that a degenerate node $\alpha_p$ belongs to the $\beta_1$-$\beta_k$ path $(\beta_1 \alpha_2, \alpha_2 \beta_2, \dots, \alpha_k \beta_k)$.
Let
$X_{\alpha_i} := U_{\alpha_i}^-$ for $1 < i \leq p$,
$X_{\alpha_j} := U_{\alpha_j}^+$ for $p < j \leq k$,
and $X_{\alpha} := U_\alpha$ for other $\alpha$.
Similarly,
let $Y_{\beta_i} := V_{\beta_i}^+$ for $1 \leq i < p$,
$Y_{\beta_j} := V_{\beta_j}^-$ for $p \leq j \leq k$,
and $Y_{\beta} := \{0\}$ for other $\beta$.
Namely,
the vertices belonging to the $\beta_1$-$\beta_k$ path are associated with 1-dimensional vector spaces,
and the other vertices $\alpha$ (resp. $\beta$) have 2-dimensional (resp. 0-dimensional) vector spaces.
Then one can see $A_C(X, Y) = \{0\}$ and $(2|\CL| + 2|\CR|) - (\dim X + \dim Y) < |C|$
by a similar argument as for the above path case.

\subsubsection{Proof of \cref{prop:substitution}}\label{subsubsec:prop:substitution}
Let $I$ be a matching.
Also let $M$ and $N$ be a row subset and a column subset of $A_I$,
respectively,
with
$|M| = |N| = \rank A_I$ and $\rank A_I[M, N] = \rank A_I$.
We denote by $\mathcal{C}'$ the set of rank-1 connected components.
By the proof of \cref{thm:chara} on the rank formula with $\rank A_C = |C|$ for $C \in \mathcal{C}'$,
we have
\begin{align*}
\det A_I[M, N] = a \cdot \prod \{\x^2 \mid \text{$\alpha \beta$ is an isolated rank-2 edge in $I$} \} \cdot \prod \{x_{\alpha \beta} \mid \alpha \beta \in C : C \in \mathcal{C}' \}
\end{align*}
for some nonzero $a \in \F$.
By assigning $1$ to all $\x$ with $\alpha \beta \in I$,
we obtain $\det \tilde{A}_I[M, N] = a \neq 0$,
implying $\rank \tilde{A}_I = \rank A_I$.

\section{Augmenting space-walk}\label{sec:algorithm}
Our proposed algorithm
is an augmenting-path type algorithm,
where we use the concept of {\it augmenting space-walk}, introduced in this section.
An outline of our algorithm is as follows.
Let $I$ be a matching of $A$.
We first search an optimality witness for $I$ or an augmenting space-walk for $I$.
In the former case,
we have $\rank A = \rank A_I$,
and hence output $r(I)$ by \cref{thm:chara}.
In the latter case,
we update a matching $I$ to another matching $I^*$ with $\rank A_{I^*} > \rank A_I$
via the augmenting space-walk.

For a vector space $X \subseteq U_\alpha$,
let $X^{\perp_{\alpha \beta}}$ (or $X^{\perp_{\beta \alpha}}$) denote
the orthogonal vector space with respect to $\A$:
\begin{align*}
X^{\perp_{\alpha \beta}} (= X^{\perp_{\beta \alpha}}) &:= \{ y \in V_\beta \mid \A(x, y) = 0 \text{ for all } x \in X \}.
\end{align*}
For a vector space $Y \subseteq V_\beta$,
$Y^{\perp_{\alpha \beta}}$ (or $Y^{\perp_{\beta \alpha}}$) is defined analogously.

\subsection{Definition}\label{subsec:augmenting space-walk}
Our augmenting path concept, named {\it augmenting space-walk},
simulates a part of the Wong sequence
needed for the augmentation;
see \cref{subsec:Wong sequence} for details.
We first define the components of augmenting space-walks.
Let $I$ be a matching of $A$.
An {\it outer walk} for $I$ is a walk $(\beta_1 \alpha_1, \alpha_1 \beta_2, \dots, \beta_k \alpha_k)$ in $G$
such that $\beta_i \alpha_i \in E \setminus I$ and $\alpha_i \beta_{i+1}$ is an isolated rank-2 edge in $I$ for each $i$.
An {\it outer space-walk} for $I$
is a sequence
$\calP = (Y_1, \beta_1 \alpha_1, X_1, \alpha_1 \beta_2, \dots, \beta_k \alpha_k, X_k)$
such that
\begin{itemize}
	\item
	$(\beta_1 \alpha_1, \alpha_1 \beta_2, \dots, \beta_k \alpha_k)$ is an outer walk for $I$,
	and
	\item
	for each $i$,
	vector subspaces
	$Y_i \subseteq V_{\beta_i}$ and $X_i \subseteq U_{\alpha_i}$
	satisfy
	$Y_i \not\subseteq \kerR(A_{\alpha_i \beta_i})$,
	$X_i = {Y_i}^{\perp_{\alpha_i \beta_i}}$,
	and $Y_{i+1} = {X_i}^{\perp_{\alpha_i \beta_{i+1}}}$.
\end{itemize}
Note $\dim X_i \leq 1$ and $\dim Y_i \geq 1$ for each $i$.
The initial vertex $\beta_1$ and last vertex $\alpha_k$ of $\calP$
are denoted by $\beta(\calP)$ and $\alpha(\calP)$,
respectively.
Also
the initial space $Y_1$ and last space $X_k$ of $\calP$
are denoted by $Y(\calP)$ and $X(\calP)$,
respectively.

An {\it inner walk} for $I$ is a walk $(\alpha_1 \beta_1, \beta_1 \alpha_2, \dots, \alpha_k \beta_k)$ in $G$
such that it belongs to a rank-1 connected component of $I$ and $\beta_i \alpha_{i+1}$ is rank-2
for each $i$.
We here fix signs $+/-$ of edges and a valid labeling.
An {\it inner space-walk} for $I$
is a sequence
$\calQ = (X_1, \alpha_1 \beta_1, Y_1, \beta_1 \alpha_2, \dots, \alpha_k \beta_k, Y_k)$
such that
\begin{itemize}
	\item
	$(\alpha_1 \beta_1, \beta_1 \alpha_2, \dots, \alpha_k \beta_k)$
	is an inner walk for $I$,
	and
	\item
	for each $i$,
	$(X_i, Y_i)$ is equal to
	$(U_{\alpha_i}^+, V_{\beta_i}^-)$ if $\alpha_i \beta_i$ is a $+$-edge,
	and is equal to $(U_{\alpha_i}^-, V_{\beta_i}^+)$ if $\alpha_i \beta_i$ is a $-$-edge.
\end{itemize}
The initial vertex $\alpha_1$ and last vertex $\beta_k$ of $\calQ$
are denoted by
$\alpha(\calQ)$ and $\beta(\calQ)$,
respectively.
Also
the initial space $X_1$ and last space $Y_k$ of $\calQ$
are denoted by $X(\calQ)$ and $Y(\calQ)$,
respectively.

We denote by $(a_1, a_2, \dots, a_k) \circ (b_1, b_2, \dots, b_\ell)$
the concatenation of sequences $(a_1, a_2, \dots, a_k)$ and $(b_1, b_2, \dots, b_\ell)$,
i.e.,
\begin{align*}
(a_1, a_2, \dots, a_k) \circ (b_1, b_2, \dots, b_\ell)
:=
\begin{cases}
(a_1, \dots, a_k, b_1, \dots, b_\ell) & \text{if $a_k \neq b_1$},\\
(a_1, \dots, a_k, b_2, \dots, b_\ell) & \text{if $a_k = b_1$}.
\end{cases}
\end{align*}

We consider an alternating concatenation
\begin{align*}
\calT := \calP_0 \circ \calQ_1 \circ \calP_1 \circ \cdots \circ \calQ_m \circ \calP_m
\end{align*}
of outer space-walks $\calP_0, \calP_1, \dots, \calP_m$
and inner space-walks $\calQ_1, \dots, \calQ_m$,
in which the following connection condition, called {\it compatibility},
is required for each $i$:
\begin{itemize}
	\item
	$\beta(\calQ_i) = \beta(\calP_i)$ and $Y(\calQ_i) = Y(\calP_i)$.
	\item
	$\alpha(\calP_i) = \alpha(\calQ_{i+1})$,
	and
	$X(\calP_i) \neq U_{\alpha(\calQ_{i+1})}^-$ if $X(\calQ_{i+1}) = U_{\alpha(\calQ_{i+1})}^+$,
	and $X(\calP_i) \neq U_{\alpha(\calQ_{i+1})}^+$ if $X(\calQ_{i+1}) = U_{\alpha(\calQ_{i+1})}^-$.
\end{itemize}

An {\it augmenting space-walk} $\calT$
is a compatible concatenation $\calP_0 \circ \calQ_1 \circ \calP_1 \circ \cdots \circ \calQ_m \circ \calP_m$ for $I$ such that
\begin{description}
	\item[($\Ai$)]
	$Y(\calP_0) = \ker_I(\beta(\calP_0)) \neq \{0\}$ and
	\item[($\Al$)]
	$X(\calP_m) \neq \ker_I(\alpha(\calP_m)) \neq \{0\}$.
\end{description}
Note that
the union of the underlying walks of $\calP_0, \calQ_1, \calP_1, \dots, \calQ_m, \calP_m$
also forms a walk in $G$.
By $X(\calP_m) \leq 1$,
($\Al$) is equivalent to $X(\calP_m) \not\supseteq \ker_I(\alpha(\calP_m)) \neq \{0\}$.

An augmenting space-walk is said to be {\it irredundant}
if every vertex appears at most twice in $\calT$ and,
in the case where a vertex $\alpha$ (resp. $\beta$) appears twice in $\calT$ as $(\dots, \beta \alpha, X, \dots, \beta' \alpha, X' \dots)$ (resp. $(\dots, \alpha \beta, Y, \dots, \alpha' \beta, Y' \dots)$),
it holds $X \not\subseteq X'$ (resp. $Y \not\supseteq Y'$).
In the following,
we assume that an augmenting space-walk is always irredundant.
Indeed, our algorithm in \cref{subsec:augmenting space-walk} outputs an irredundant augmenting space-walk.
Also, any redundant augmenting space-walk can be converted to an
irredundant one in $O(|E|)$ time by
shortening procedures in \cref{sec:augmentation};
in fact, the irredundancy is maintained during the entire augmentation procedure.

An augmenting space-walk actually augments a matching.
The following provides the validity of our augmenting procedure.
\begin{thm}\label{thm:augmentation}
	For a matching $I$ and an augmenting space-walk for $I$,
	we can obtain a matching $I^*$ with $\rank A_{I^*} > \rank A_I$
	in $O(|E|^2)$ time.
\end{thm}
The augmentation procedure
is the most difficult but intriguing part of this paper,
and the proof of \cref{thm:augmentation} is given in \cref{sec:augmentation}.

\subsection{Finding an augmenting space-walk}\label{subsec:finding}
In this subsection,
we present an algorithm for finding either an optimality witness
or an augmenting space-walk,
which is inspired by the computation of the Wong sequence; see \cref{subsec:Wong sequence} for details.
\begin{thm}\label{thm:labeling}
	For a matching $I$,
	we can find either an optimality witness
	or an augmenting space-walk for $I$
	in $O(|E|)$ time.
\end{thm}
\cref{thm:augmentation,thm:labeling} implies \cref{thm:main},
since at most $\min \{ \mu, \nu \}$ augmentations occur in the algorithm.
The proof of \cref{thm:labeling} is given in \cref{subsubsec:labeling,subsubsec:back tracking} below.

In order to find an optimality witness
or an augmenting space-walk,
we introduce a {\it label}
for
each vertex $\gamma$,
which is a vector subspace of $U_\alpha$ if $\gamma = \alpha$, and that of $V_\beta$ if $\gamma = \beta$.
Each pair $(\gamma, Z)$ of a vertex $\gamma$ and a label $Z$ at $\gamma$
has a back pointer to another pair $(\gamma', Z')$,
which represents that ``$Z$ is added to the label set $\labeltt(\gamma)$ through $Z'$ and $\gamma' \gamma$.''
At the end of the algorithm,
we obtain either an optimality witness for $I$ by composing labels
or an augmenting space-walk by tracking back pointers and concatenating them.

\subsubsection{Labeling procedure}\label{subsubsec:labeling}
For each vertex $\gamma$,
let $\labeltt(\gamma)$ be a label set (or a set of vector spaces) of $\gamma$.
While updating,
let
\begin{align*}
X_\alpha^* := \bigcap_{X \in \labeltt(\alpha)} X, \qquad Y_\beta^* := \sum_{Y \in \labeltt(\beta)} Y
\end{align*}
for each $\alpha$ and $\beta$,
where $X_\alpha^* := U_\alpha$ if $\labeltt(\alpha) = \emptyset$,
and $Y_\beta^* := \{0\}$ if $\labeltt(\beta) = \emptyset$.
As initialization,
let $\labeltt(\beta) := \{ \ker_I(\beta) \}$ for $\beta$ with $\ker_I(\beta) \neq \{0\}$,
and $\labeltt(\gamma) := \emptyset$ for other $\gamma$.

The initial stage of the labeling procedure is to
check whether there is a triple $(\alpha, \beta, Y)$
such that $\alpha \beta \in E \setminus I$,
$Y \in \labeltt(\beta)$,
and $Y^{\perp_{\alpha \beta}} \not\supseteq X_\alpha^*$.
If there is no such triple,
then output $X_\alpha^*$ and $Y_\beta^*$ for $\alpha, \beta$
as an optimality witness for $I$,
and
stop the procedure.
Otherwise choose such a triple $(\alpha, \beta, Y)$.
There are three cases:
\begin{itemize}
	\item[(A)]
	$\deg_I(\alpha) = 0$, or $\alpha$ belongs to a rank-1 connected component $C$
	such that $\deg_I(\alpha) = 1$ and $Y^{\perp_{\alpha \beta}} \neq \ker_I(\alpha) \in \{ U_\alpha^+, U_\alpha^- \}$.
	\item[(B)]
	$\alpha$ is incident to the isolated rank-2 edge $\alpha \beta'$ in $I$.
	\item[(C)]
	$\alpha$ belongs to a rank-1 connected component $C$ such that $\deg_I(\alpha) = 2$, or $\deg_I(\alpha) = 1$ and $Y^{\perp_{\alpha \beta}} = \ker_I(\alpha) \in \{ U_\alpha^+, U_\alpha^- \}$.
\end{itemize}

(A).
In this case, we have $Y^{\perp_{\alpha \beta}} \neq \ker_I(\alpha) \neq \{0\}$ (cf. ($\Al$) in \cref{subsec:augmenting space-walk}).
Define the back pointer from $(\alpha, Y^{\perp_{\alpha \beta}})$ to $(\beta, Y)$,
and go to \cref{subsubsec:back tracking}
to
obtain an augmenting space-walk.

(B).
This case will be an expansion of an outer space-walk.
Update
\begin{align}
\labeltt(\alpha) &\leftarrow \labeltt(\alpha) \cup \{ Y^{\perp_{\alpha \beta}} \},\label{eq:update outer a}\\
\labeltt(\beta') &\leftarrow \labeltt(\beta') \cup \{ (Y^{\perp_{\alpha \beta}})^{\perp_{\alpha \beta'}} \}\label{eq:update outer b}
\end{align}
and define the back pointers from $(\alpha, Y^{\perp_{\alpha \beta}})$ to $(\beta, Y)$
and from $(\beta', (Y^{\perp_{\alpha \beta}})^{\perp_{\alpha \beta'}})$ to $(\alpha, Y^{\perp_{\alpha \beta}})$.
Return to the initial stage.

(C).
This case will be an addition of an inner space-walk.
Update
\begin{align}\label{eq:update inner a}
\labeltt(\alpha) \leftarrow
\begin{cases}
\labeltt(\alpha) \cup \{ U_\alpha^+\} & \text{if $Y^{\perp_{\alpha \beta}} = U_\alpha^+$},\\
\labeltt(\alpha) \cup \{ U_\alpha^-\} & \text{if $Y^{\perp_{\alpha \beta}} = U_\alpha^-$},\\
\labeltt(\alpha) \cup \{ U_\alpha^+, U_\alpha^-\} & \text{if $U_\alpha^+ \neq Y^{\perp_{\alpha \beta}} \neq U_\alpha^-$};
\end{cases}
\end{align}
see the compatibility condition in \cref{subsec:augmenting space-walk}.
For each $X$ newly added to $\labeltt(\alpha)$,
do the following:
Define the back pointer from $(\alpha, X)$ to $(\beta, Y)$.
Choose the longest inner space-walk $\calQ$ in $C$ such that
$\alpha = \alpha(\calQ)$ and $X = X(\calQ)$,
where
we can assume $X = X(\calQ) = U_\alpha^+$
(if $X = X(\calQ) = U_\alpha^-$,
then we change all signs in the argument below).
For each vertex $\gamma$ in $\calQ$,
add to $\labeltt(\gamma)$ the subspace of $\calQ$ at $\gamma$
if $\labeltt(\gamma)$ does not have it:
\begin{align}\label{eq:update inner a b}
	\labeltt(\gamma) \leftarrow
	\begin{cases}
	\labeltt(\gamma) \cup \{ U_{\gamma}^+ \} & \text{if $\gamma = \alpha$},\\
	\labeltt(\gamma) \cup \{ V_{\gamma}^- \} & \text{if $\gamma = \beta$}.
	\end{cases}	
\end{align}
Define the back pointers of newly added labels along $\calQ$.
\begin{lem}\label{lem:labeling}
	If the algorithm outputs $X_\alpha^*$ and $Y_\beta^*$ for each $\alpha, \beta$,
	then it is an optimality witness for $I$.
	In addition, the running-time is $O(|E|)$.
\end{lem}
\begin{proof}
	In the initial phase,
	it holds
	$r(I) = 2\mu + 2\nu - \sum_{\alpha} \dim X_\alpha^* - \sum_{\beta} \dim Y_\beta^*$.
	Since the sum of $\dim X_\alpha^*$ and $\dim Y_\beta^*$ for $\alpha, \beta$ does not change during the algorithm,
	it suffices to show $\A(X_\alpha^*, Y_\beta^*) = \{0\}$ for all $\alpha \beta \in E$ on the outputs $X_\alpha^*$ and $Y_\beta^*$.

	For $\alpha \beta \in E \setminus I$,
	we have $\A (X_\alpha^*, Y_\beta^*) = \{ 0 \}$ by the termination condition.
	For a rank-2 edge $\alpha \beta \in I$,
	we have $Y^\perpab \in \labeltt(\alpha)$ if $Y \in \labeltt(\beta)$.
	Hence it holds that $\A (X_\alpha^*, Y_\beta^*) \subseteq \A(\bigcap_{Y \in \labeltt(\beta)} Y^\perpab, \sum_{Y \in \labeltt(\beta)} Y) = \{0\}$.
	For a rank-1 edge $\alpha\beta \in I$,
	we can assume that $\alpha\beta$ is a $+$-edge.
	Then $U_\alpha^+ = \kerL(\A)$ and $V_\beta^+ = \kerR(\A)$.
	If $V_\beta^- \in \labeltt(\beta)$ then $U_\alpha^+ \in \labeltt(\alpha)$,
	since an inner space-walk $\calQ$ reaching $(\beta, V_\beta^-)$
	comes from $(\alpha, U_\alpha^+)$.
	Thus, if $V_\beta^- \not\in \labeltt(\beta)$
	then $\A(X_\alpha^*, Y_\beta^*) \subseteq \A(X_\alpha^*, \kerR(\A)) = \{0\}$,
	and if $V_\beta^- \in \labeltt(\beta)$
	then $\A(X_\alpha^*, Y_\beta^*) \subseteq \A(\kerL(\A), Y_\beta^*) = \{0\}$.

	By $|\labeltt(\gamma)| \leq 2$ for each $\gamma$,
	every edge in $E$ is checked at most two times during the labeling procedure;
	the labeling procedure can be done in $O(|E|)$ time.
\end{proof}

\subsubsection{Back tracking}\label{subsubsec:back tracking}
Suppose that
we have $(\alpha, X)$
with $X \in \labeltt(\alpha)$ and $X \neq \ker_I(\alpha) \neq \{0\}$.
By tracking the back pointer,
we obtain an augmenting space-walk for $I$.
As initialization,
define $\calT := (Y, \beta \alpha, X)$,
where $(\alpha, X)$ has the back pointer to $(\beta, Y)$.
Observe $\beta \alpha \not\in I$ and $Y \not\subseteq \kerR(\A)$,
which implies that $(Y, \beta \alpha, X)$ is an outer space-walk of $I$.

We update $\calT$ as follows.
Let $Y$ and $\beta \alpha$ be the initial space and initial edge of $\calT$,
respectively.
Suppose that $(\beta, Y)$ has the back pointer to $(\alpha', X)$
and $(\alpha', X)$ has the back pointer to $(\beta', Y')$.
Then update
\begin{align*}
	\calT \leftarrow (Y', \beta' \alpha', (Y')^{\perp_{\beta' \alpha'}}) \circ (X, \alpha' \beta, Y) \circ \calT.
\end{align*}
If $(\beta', Y')$ has no back pointer,
i.e., $Y' = \ker_{\beta'}(I)$,
then output the resulting $\calT$ as an augmenting space-walk.
Otherwise, consider $(\beta', Y')$ of the resulting $\calT$,
and repeat the above update.
Clearly, the running-time of the back tracking procedure is $O(|E|)$.

Then the following holds.
\begin{lem}\label{lem:constructing}
	The output $\calT$ is an augmenting space-walk for $I$.
\end{lem}
\begin{proof}
	We consider when the back pointers from $(\beta, Y)$ to $(\alpha', X)$ and from $(\alpha', X)$ to $(\beta', Y')$
	are defined.
	If the former is defined in the update~\eqref{eq:update outer b}
	and the latter in~\eqref{eq:update outer a},
	then
	$(Y')^{\perp_{\beta' \alpha}} = X$
	and
	$(Y', \beta' \alpha, (Y')^{\perp_{\beta' \alpha}}) \circ (X, \alpha \beta, Y)$ constitutes a part of an outer space-walk of $I$.
	If the former is defined in the update~\eqref{eq:update inner a b} and the latter
	in~\eqref{eq:update inner a},
	then $(Y', \beta' \alpha, (Y')^{\perp_{\beta' \alpha}})$
	and $(X, \alpha \beta, Y)$
	constitute a part of an outer and inner space-walk of $I$,
	respectively.
	If both the former and latter are defined in the update~\eqref{eq:update inner a b},
	then
	$(Y')^{\perp_{\beta' \alpha}} = X$
	and
	$(Y', \beta' \alpha, (Y')^{\perp_{\beta' \alpha}}) \circ (X, \alpha \beta, Y)$ constitutes a part of an inner space-walk of $I$.
	Thus the output $\calT$ forms an augmenting space-walk for $I$.
	In particular,
	$\calT$ is irredundant,
	since no label appears twice in $\calT$ by the description of the labeling procedure.
\end{proof}

By \cref{lem:labeling,lem:constructing},
we obtain \cref{thm:labeling}.

\subsection{Augmenting space-walk and Wong sequence}\label{subsec:Wong sequence}
This subsection is devoted to clarifying the relationship between augmenting space-walk and Wong sequence,
with highlighting the motivation behind our algorithm.
The contents of this subsection are irrelevant to \cref{sec:augmentation} (the validity of the algorithm).

Let $A$ be an $m \times n$ symbolic matrix of the form~\eqref{eq:linear A}, not necessarily partitioned.
Let $\tilde{A}$ be a substitution of $A$.
For a vector space $X \subseteq \F^m$,
we denote by
\begin{align*}
X^{\perp_{\tilde{A}}} := \{ y \in \F^{n} \mid \tilde{A}(x, y) = 0 \text{ for all } x \in X \},
\end{align*}
the orthogonal vector space with respect to $\tilde{A}$.
For a vector space $Y \subseteq \F^{n}$,
$Y^{\perp_{A_i}}$ is defined analogously.
In addition,
let
\begin{align*}
Y^{\perp_A} := \bigcap_{i = 1}^k Y^{\perp_{A_i}}.
\end{align*}
The (orthogonal version of) {\it Wong sequence}~\cite{JCSS/IKQS15} for $(A, \tilde{A})$
is a sequence $X^0, Y^1, X^1, \dots$ of vector spaces determined by
$X^0 := \F^m$, $Y^i := (X^{i-1})^{\perp_{\tilde{A}}}$,
and $X^i := (Y^i)^{\perp_{A}}$ for $i = 1,2,\dots$.
By $\F^m = X^0 \supseteq X^1 \supseteq \cdots$ and $\kerR(\tilde{A}) = Y^1 \supseteq Y^2 \supseteq \cdots$,
the limits $X^\infty, Y^\infty$
are obtained by $\min \{ m, n \}$ iterations
and
satisfy $Y^\infty = (X^\infty)^{\perp_{\tilde{A}}}$ and
$X^\infty = (Y^\infty)^{\perp_{A}}$.
In~\cite{JCSS/IKQS15},
they showed $\rank A = \rank \tilde{A}$ if $\kerL(\tilde{A}) \subseteq X^\infty$,
which is used as an ``optimality witness'' of $\tilde{A}$ in IQS-algorithm~\cite{CC/IQS18}.

Let us consider a general partitioned matrix $A =(A_{\alpha \beta})$ of the form~\eqref{eq:A},
where $A_{\alpha \beta}$ is an $m_\alpha \times n_\beta$ matrix.
The partitioned structure of $A$ simplifies the computation of $(Y)^{\perp_A}$.
Indeed, $X = (Y)^{\perp_A}$ is represented
as
$X = \bigoplus_\alpha X_\alpha$
for
\begin{align*}
X_\alpha = \left(\bigcap_\beta \left(\proj_\beta(Y)\right)^{\perp_{A_{\alpha \beta}}}\right),
\end{align*}
where
$\proj_\beta(Y)$ denotes
the projection of $Y \subseteq \bigoplus_\beta \F^{n_\beta}$ to the $\beta$-th coordinate.
Thus,
the essential part to obtain the Wong sequence
is the computation of $\proj_\beta((\bigoplus_\alpha X_\alpha)^{\perp_{\tilde{A}}})$.

Consider the case where each $A_{\alpha \beta}$ is $1 \times 1$.
As mentioned in Introduction,
computing $\rank A$ is equivalent to computing the size of a maximum matching of the bipartite graph $G$ corresponding to the nonzero pattern of $A$.
In this case,
$X_\alpha = \{0\}$ if $\alpha$ is incident to some edge $\alpha\beta$ in $G$ with $\proj_\beta(Y) = \F$,
and
$X_\alpha = \F$ otherwise.
Suppose further that $\tilde{A}$ is the canonical substitution of $A$ with respect to a (bipartite) matching $I$.
Then
$Y_\beta := \proj_\beta((\bigoplus_\alpha X_\alpha)^{\perp_{\tilde{A}}}) = \{0\}$ if $\beta$ is incident to an edge $\alpha\beta$ in $I$ with $X_\alpha = \F$,
and $Y_\beta = \F$ otherwise.
In particular,
for
$Y^1 = \kerR(\tilde{A})$,
the $\beta$-th projection of $Y^1$ is $\F$ if and only if $\beta$ is unmatched by $I$.
The computation of the Wong sequence
is nothing but
the BFS search for an augmenting path in $G$;
the node set $\{ \alpha \mid X_\alpha = \{0 \}\} \cup \{ \beta \mid Y_\beta = \F \}$ in the $i$-th phase
coincides with the reachable node set from the unmatched nodes via the BFS search for an augmenting path in $i$ steps.
The optimality condition $\kerL(\tilde{A}) \subseteq X_\infty$
holds if and only if
there is no augmenting path for the current matching $I$.
The above argument says that
the matching structure of $I$ can allow us to compute $\proj_\beta((\bigoplus_\alpha X_\alpha)^{\perp_{\tilde{A}}})$
in a combinatorial manner.

We then consider a $(2 \times 2)$-type generic partitioned matrix $A$.
In fact,
the definition of an augmenting space-walk and our labeling procedure for $A$
are essentially the computation of the Wong sequence for $(A, \tilde{A}_I)$,
where $I$ is a matching of $A$.
The vector spaces $X_\alpha^*$ and $Y_\beta^*$ in the labeling procedure (\cref{subsubsec:labeling})
play a role of the above $X_\alpha$ and $\proj_\beta((\bigoplus_\alpha X_\alpha)^{\perp_{\tilde{A}_I}})$,
respectively.
If $\beta$ is incident to an isolated rank-2 edge $\alpha \beta$ in $I$,
then $\proj_\beta((\bigoplus_\alpha X_\alpha)^{\perp_{\tilde{A}_I}}) = (X_\alpha)^{\perp_{\alpha \beta}}$,
which corresponds to (B) in the labeling procedure.
If $\beta$ belongs to a rank-1 connected component of $I$,
then $\proj_\beta((\bigoplus_\alpha X_\alpha)^{\perp_{\tilde{A}_I}})$ can be determined by inner space-walks as follows.
\begin{lem}\label{lem:proj}
	Suppose that $\beta$ belongs to a rank-1 connected component of $I$.
	If
	there is an inner space-walk $\calQ$ for $I$ from some $\alpha$ to $\beta$ such that $Y(\calQ) = V_\beta^+$ and $X_\alpha \not\supseteq U_\alpha^+$,
	then
	$\proj_\beta((\bigoplus_\alpha X_\alpha)^{\perp_{\tilde{A}_I}}) \supseteq V_\beta^+$.
	Otherwise,
	$\proj_\beta((\bigoplus_\alpha X_\alpha)^{\perp_{\tilde{A}_I}}) \subseteq V_\beta^-$.
\end{lem}
The analogous statement interchanged $+$ with $-$ also holds.
These explain (C).
\begin{proof}
	We may assume that $I$ itself is a rank-1 connected component.
	As in the proof of \cref{thm:chara} (\cref{subsubsec:thm:chara}),
	via a basis transformation,
	we can change $\tilde{A}_I$ as a matrix of the form~\eqref{eq:EAF}.
	For notational simplicity,
	we use $1^+, 2^+, \dots, k^+, 1^-, 2^-, \dots, k^-$ as the row and column indices of $\tilde{A}_I$ after the transformation.
	Then
	$U_\alpha^+$ and $U_{\alpha}^-$ are
	the vector spaces spanned by the $\alpha^+$-th and $\alpha^-$-th unit vectors,
	respectively.
	The same holds for $V_{\beta}^+$ and $V_{\beta}^-$.
	We denote by $a_{\alpha \beta}^+$ and $a_{\alpha \beta}^-$
	the $\alpha^+ \beta^+$-th and $\alpha^- \beta^-$-th (nonzero) entries of $\tilde{A}_I$,
	respectively.

	Suppose first that the if-condition in the statement of \cref{lem:proj} holds.
	That is,
	suppose that, e.g., $X_{1} \not\supseteq U_{1}^+$
	and that $\beta$ belongs to the maximal inner space-walk $\calQ$ with $\alpha(\calQ) = 1$ and $X(\calQ) = U_1^-$.
	By $X_{1} \not\supseteq U_{1}^+$,
	there is $c \in \F$
	such that any $x \in \bigoplus_{\alpha} X_\alpha$ has of the form
	\begin{align*}
	x =
	\begin{blockarray}{cccccccc}
	1^+ & & & & 1^- & & & \\
	\begin{block}{[cccc|cccc]}
	cp & \ast & \cdots & \ast & p & \ast & \cdots & \ast \\
	\end{block}
	\end{blockarray}
	\end{align*}
	for some $p \in \F$.
	We can construct $y \in (\bigoplus_{\alpha} X_\alpha)^{\perp_{\tilde{A}_I}}$
	such that
	\begin{align}\label{eq:A_Iy}
	\tilde{A}_I y =
	\begin{blockarray}{cccccccc}
	1^+ & & & & 1^- & & & \\
	\begin{block}{[cccc|cccc]}
	1 & 0 & \cdots & 0 & -c & 0 & \cdots & 0 \\
	\end{block}
	\end{blockarray}.
	\end{align}
	Indeed,
	such $y$ is given as
	\begin{align*}
	y = 
	\begin{blockarray}{cccccccccccc}
	1^+ & & & & (k-1)^+ & k^+ & 1^- & 2^- & & & & & \\
	\begin{block}{[cccccc|cccccc]}
	0 & \cdots & 0 & \cdots & -\frac{a_{k-1k-1}^+}{a_{1k}^+ a_{k k-1}^+} & \frac{1}{a_{1k}^+} & 
	-\frac{c}{a_{11}^-} & \frac{ca_{21}^-}{a_{11}^- a_{22}^-} & \cdots & 0 & \cdots & 0\\
	\end{block}
	\end{blockarray},
	\end{align*}
	where
	the nonzero entries of $y$ can be determined along inner paths starting from $1$
	so that~\eqref{eq:A_Iy} is satisfied.
	In particular,
	the nodes in $\calQ$ correspond to the nonzero entries in the $+$ part of $y$.
	Since $\beta$ belongs to $\calQ$,
	it holds $\proj_{\beta}((\bigoplus_{\alpha} X_\alpha)^{\perp_{\tilde{A}_I}}) \supseteq V_{\beta}^+$,
	which implies the if-part.
	
	Let $\calQ'$ denote the maximal inner space-walk with $\beta(\calQ') = \beta$ and $Y(\calQ') = V_\beta^+$;
	the ``otherwise'' condition
	is equivalent to $X_\alpha \supseteq U_\alpha^+$ for any $\alpha$ in $\calQ'$.
	We show $y \notin (\bigoplus_{\alpha} X_\alpha)^{\perp_{\tilde{A}_I}}$
	for any $y$ with a nonzero $\beta^+$-th entry,
	which implies $\proj_{\beta}((\bigoplus_{\alpha} X_\alpha)^{\perp_{\tilde{A}_I}}) \subseteq V_{\beta}^-$.
	Suppose to the contrary that
	there is $y \in (\bigoplus_{\alpha} X_\alpha)^{\perp_{\tilde{A}_I}}$
	with a nonzero $\beta^+$-th entry.
	Let $\alpha \beta$ be the last edge of $\calQ'$.
	By $X_\alpha \supseteq U_\alpha^+$,
	the $\alpha^+$-th entry of $x \in \bigoplus_{\alpha} X_\alpha$
	can be arbitrary.
	Hence,
	if $\alpha(\calQ') = \alpha$,
	i.e.,
	$\alpha$ is incident only to $\alpha \beta$ in $I$, or $\alpha$ is incident to two edges $\alpha \beta$ and $\alpha \beta'$ in $I$
	and $\alpha \beta'$ is rank-1,
	then $x^\top \tilde{A}_I y \neq 0$ for $x \in \bigoplus_{\alpha} X_\alpha$ having nonzero only on the $\alpha^+$-th entry, a contradiction.
	Otherwise, $\alpha$ is incident to two edges $\alpha \beta$ and $\alpha \beta'$ in $I$
	and $\alpha \beta'$ is rank-2.
	In this case, the ${\beta'}^+$-th entry of $y$ must be nonzero for $y \in (\bigoplus_{\alpha} X_\alpha)^{\perp_{\tilde{A}_I}}$.
	By repeating the same argument,
	we obtain a contradiction at the initial edge of $\calQ'$.
\end{proof}
While the computation of the Wong sequence corresponds to the BFS of an auxiliary graph,
our algorithm does not.
An augmenting space-walk partially simulates the Wong sequence
for deriving $\kerL(\tilde{A}_I) \not\subseteq X^\infty$.

\section{Augmentation}\label{sec:augmentation}
In order to prove \cref{thm:augmentation},
we present an augmentation procedure
for a given matching $I$ and an augmenting space-walk $\calT$ with respect to $I$.
The procedure is an inductive construction that repeats to replace $(I,\calT)$ by $(I',\calT')$.
Here $I'$ is relaxed to be a {\it quasi-matching},
which is a weaker notion than matching,
with $r(I') = r(I)$ and $\calT'$ 
is an augmenting space-walk for $I'$.
In the base case $(I, \calT)$,
it holds $r(I') > r(I)$.
The irredundancy of $(I, \calT)$ always holds in the augmentation procedure,
which is easily verified.

\subsection{Preliminaries}\label{subsec:preliminaries}
We here introduce several notions for the augmentation procedure
and the proof of its validity.
For an outer or inner space-walk $\mathcal{R} = (Z_1, \gamma_1 \gamma_2, Z_2, \gamma_2 \gamma_3, \dots, \gamma_{k-1} \gamma_k, Z_k)$,
we define $\mathcal{R}[\gamma_i, \gamma_j]$ by the subsequence of $\mathcal{R}$ from $Z_i$ to $Z_j$ if $i < j$,
and by the subsequence of the inverse of $\mathcal{R}$ from $Z_i$ to $Z_j$ if $i > j$.
In particular, if $\gamma_i$ is the initial vertex $\gamma_1$ of $\mathcal{R}$,
then we denote $\mathcal{R}[\gamma_1, \gamma_j]$ by $\mathcal{R}(\gamma_j]$.
If $\gamma_j$ is the last vertex $\gamma_k$ of $\mathcal{R}$,
then we denote $\mathcal{R}[\gamma_i, \gamma_k]$ by $\mathcal{R}[\gamma_i)$.
The same notation is used for a walk $R$ in $G$.
For an outer or inner space-walk $\mathcal{R}$,
we denote the underlying walk of $\mathcal{R}$
by its italic style $R$.
Note that
an inner space-walk $\calQ$
is uniquely determined from its underlying inner walk $Q$.

\subsubsection{Quasi-matching}
In our algorithm,
we utilize a weaker notion than matching.
An edge subset $I \subseteq E$ is called a {\it quasi-matching} or {\it q-matching}
if $I$ satisfies (Deg), (VL), (Path), and the following weaker condition (q-Cycle) than (Cycle):
\begin{description}
	\item[(q-Cycle)]
	Each cycle component of $I$ has at least one rank-1 edge.
\end{description}
From a q-matching $I'$,
we easily obtain a matching $I$ with $r(I) \geq r(I')$;
see \cref{lem:elimination} in \cref{subsubsec:elimination} below.

We naturally extend an augmenting space-walk for a q-matching $I$.
Note here that, as a matching in \cref{subsec:def and chara}, if $I$ satisfies (Deg), (Path), and (q-Cycle)
then two labels on any vertex with degree 2
are uniquely determined.
(VL) requires that they are different.

\subsubsection{Elimination}\label{subsubsec:elimination}
Let $I$ be an edge subset satisfying (Deg), (q-Cycle), and (VL) (note that (Path) is not imposed).
The {\it elimination}
is an operation
of modifying $I$ to a matching $I'$ with $r(I') \geq r(I)$ as follows.
If $C$ is a cycle component of $I$ such that all $+$-edges in $C$ are rank-2,
i.e., $C$ violating (Cycle),
then we
remove all $-$-edges from $C$.
Suppose that $C$ is a non-isolated path component of $I$
such that one of the end edges is rank-2,
i.e., $C$ violates (Path).
We may assume that $C$ is of the form $(\alpha_1 \beta_1, \beta_1\alpha_2, \dots)$
and $\alpha_1 \beta_1$ is rank-2.
Then we remove $\beta_1 \alpha_2, \beta_2 \alpha_3, \dots, \beta_\ell \alpha_{\ell+1}$
for the maximum $\ell \geq 1$
such that all $\alpha_1 \beta_1, \alpha_2 \beta_2, \dots, \alpha_{\ell} \beta_{\ell}$ are rank-2.
If the other end edge is not removed and is also rank-2,
do the same procedure from this edge in the reverse way.

Observe that the elimination operation to a cycle component does not change $r$,
and to a path component does not decrease $r$.
Thus the following holds.
\begin{lem}\label{lem:elimination}
	Let $I$ be an edge subset satisfying {\rm (Deg)}, {\rm (q-Cycle)}, and {\rm (VL)}
	and
	$I'$ the edge subset obtained from $I$ by the elimination.
	Then $I'$ is a matching with $r(I') \geq r(I)$.
\end{lem}

\subsubsection{Propagation}\label{subsubsec:propagation}
Let $P = (\beta_1 \alpha_1, \alpha_1 \beta_2, \dots, \beta_k \alpha_k)$
be an outer walk.
For a vector subspace $Y_1 \subseteq V_{\beta_1}$,
the {\it front-propagation} $Y_1 \triangleright P$
is a sequence $(Y_1, \beta_1 \alpha_1, X_1, \alpha_1 \beta_2, Y_2, \dots, Y_k, \beta_k \alpha_k, X_k)$
such that $X_i = (Y_i)^{\perp_{\beta_i \alpha_i}}$ and $Y_{i+1} = (X_i)^{\perp_{\alpha_i \beta_{i+1}}}$ for each $i$.
For a vector subspace $X_k' \subseteq U_{\alpha_k}$,
the {\it back-propagation} $P \triangleleft X_k'$ for $P$ and $X_k'$
is a sequence $(Y_1', \beta_1 \alpha_1, X_1', \alpha_1 \beta_2, Y_2', \dots, Y_k', \beta_k \alpha_k, X_k')$
such that $Y_i' = (X_i')^{\perp_{\alpha_i \beta_i}}$ and $X_{i-1}' = (Y_i')^{\perp_{\beta_i \alpha_{i-1}}}$ for each $i$.
Note that an outer space-walk $\calP$
coincides with $Y(\calP) \triangleright P$;
recall that $Y(\calP)$ denotes the initial space of $\calP$.
We often use the following lemma.
\begin{lem}\label{lem:diff}
	Let
	$\calP$
	be an outer space-walk with $\alpha := \alpha(\calP)$ and $\beta := \beta(\calP)$,
	and $X \subseteq U_\alpha$ and $Y \subseteq V_\beta$ be nonzero vector subspaces.
	\begin{itemize}
		\item[{\rm (1)}]
		If $X \neq X(\calP)$,
		then
		the vector subspaces at each position in $\calP$ and in $P \triangleleft X$
		are different.
		
		\item[{\rm (2)}]
		In addition,
		if $Y \neq Y(P \triangleleft X)$,
		then
		the vector subspaces at each position in $Y \triangleright P$ and in $P \triangleleft X$ are different,
		and $Y \triangleright P$ forms an outer space-walk.
	\end{itemize}
\end{lem}
\begin{proof}
	The both immediately follow from the fact that,
	if $\alpha \beta$ is rank-2
	then $X' \not\subseteq X$ implies $(X')^{\perp_{\alpha\beta}} \not\supseteq X^{\perp_{\alpha\beta}}$,
	and
	if $\alpha \beta$ is rank-1
	then $X \not\subseteq \kerL(\A)$ implies $X^{\perp_{\alpha\beta}} = \kerR(\A)$.
\end{proof}

A {\it truncated outer walk} is a walk of the form $P[\alpha)$ for some $\alpha$ belonging to an outer walk $P$.
We will consider the front-/back-propagation for truncated outer walks
in the same way.
The same statement of \cref{lem:diff}~(2) holds for a truncated outer walk $P[\alpha)$
and a vector subspace of $X_{\alpha}$
not including the space of $P \triangleleft X$ at $\alpha$.

Let $R$ be a truncated outer walk with the initial vertex $\alpha(\calP)$.
We denote $\calP \circ \left( X(\calP) \triangleright R \right)$ by
$\calP \triangleright R$ for simplicity.

\subsubsection{Decremental quantity $\theta$}
In order to estimate the time complexity,
we introduce a quantity $\theta$ which decreases during the algorithm.
Let $I$ be a q-matching and $\calT$ an augmenting space-walk for $I$.
The {\it extended support} of $\calT$
is
the union of the underlying walk $T$ of $\calT$ and the rank-1 connected components intersecting with $T$.
Our procedure will be done within the extended support.
An edge $\alpha \beta \in I$ is said to be {\it inner-double}
if $\alpha \beta$ appears twice in the union of inner space-walks of $\calT$.

Let $N(I, \calT)$ be the number of edges in the extended support of $\calT$
and
$D_{\rm inner}(I, \calT)$ the number of inner-double edges in $\calT$.
Define $\theta(I, \calT) := N(I, \calT) + D_{\rm inner}(I, \calT)$.
Clearly $\theta(I, \calT)$ is bounded by $O(|E|)$.

\subsection{Initial stage}\label{subsec:initial stage}
We start to describe the augmentation procedure.
Let $\calT = \calP_0 \circ \calQ_1 \circ \calP_1 \circ \cdots \circ \calQ_m \circ \calP_m$ be an augmenting space-walk for $I$.
We consider the following condition ($\No$):
\begin{description}
	\item[($\No$)]
	Suppose that a vertex $\gamma$ in $\calP_m$ appears twice in $\calT$,
	where the first appearance occurs in $\calP_{\ell}$ with $\ell \leq m$.
	Then the subspace at the first appearance of $\gamma$ in $\calP_{\ell}$
	is equal to the subspace at the last appearance of $\gamma$
	at $P_m \triangleleft \ker_I(\alpha(\calP_m))$.
\end{description}
When $\calT$ violates ($\No$),
we replace $\calT$ by another augmenting space-walk as follows.
Suppose that $\gamma$ appears twice in $\calT$ as in ($\No$), 
but the subspace of $\calP_{\ell}$ at the first $\gamma$
is different from the subspace of $P_m \triangleleft \ker_I(\alpha(\calP_m))$ at the second $\gamma$.
Choose the first such $\gamma$ in $\calP_\ell$.
Note that, if $\gamma \neq \alpha$,
then $\gamma = \beta(\calP_m) = \beta(\calP_{\ell})$.
Update $\calT$ as
\begin{align}\label{eq:T' Nouter}
\calT \leftarrow
\begin{cases}
\calP_0 \circ \calQ_1 \circ \cdots \circ \calQ_{\ell} \circ \calP_{\ell} & \text{if $\gamma = \alpha(\calP_m)$},\\
\calP_0 \circ \calQ_1 \circ \cdots \circ \calQ_{\ell} \circ \left(\calP_{\ell}(\alpha] \triangleright P_m[\alpha)\right) & \text{if $\gamma = \alpha \neq \alpha(\calP_m)$},\\
\calP_0 \circ \calQ_1 \circ \cdots \circ \calQ_{\ell} \circ \left( Y(\calQ_{\ell}) \triangleright P_m\right) & \text{if $\gamma = \beta(\calP_m) = \beta(\calP_{\ell})$}.
\end{cases}
\end{align}
(In fact, the first case never occurs in the procedure, provided the initial augmenting space-walk is obtained in the algorithm in \cref{subsec:finding}.)

Clearly the resulting $\calT$ is a compatibly-concatenated space-walk and
satisfies ($\Ai$).
Consider the case of $\gamma = \alpha \neq \alpha(\calP)$. (The other cases are similar.)
Since the subspace $X$ of $\calP_{\ell}$ at the first $\alpha$ is different from the subspace of $P_m \triangleleft \ker_I(\alpha(\calP_m))$
at $\alpha$,
the last space of $X \triangleright P_m[\alpha)$ is different from $\ker_I(\alpha(\calP_m))$ by \cref{lem:diff}~(2).
This implies that $\calP_{\ell}(\alpha] \triangleright P_m[\alpha)$ is a single outer space-walks for $I$,
and that ($\Al$) holds.
Thus the new $\calT$ is an augmenting space-walk for $I$.

By this update,
$\theta$ strictly decreases.
Indeed,
clearly $D_{\rm inner}$ does not increase.
If $\ell = m$,
then, by the irredundancy and the definition of $\gamma = \alpha$,
the edge in $P_m \setminus I$ incident to the second $\gamma$
exits the extended support of the new $\calT$.
Hence $N$ decreases by at least one.
If $\ell < m$,
then
the irredundancy implies the following:
If $\gamma = \alpha$,
then the last edge of $P_{m}(\alpha]$
exits the extended support of the new $\calT$.
If $\gamma = \beta(\calP_m) = \beta(\calP_{\ell})$,
then the initial edge of $P_\ell$ exits.
Hence $N$ decreases by at least one.
Checking ($\No$) and the update~\eqref{eq:T' Nouter}
can be done in $O(|E|)$.

Now $\calT$ satisfies ($\No$).
We let $\calT = \calP_0 \circ \calQ_1 \circ \calP_1 \circ \cdots \circ \calQ_m \circ \calP_m$ again
by re-index.
An outer space-walk is said to be {\it simple}
if its underlying walk is actually a path,
i.e., does not use the same edge twice.
There are three cases:
\begin{itemize}
	\item $m = 0$ and $\calP_0$ is simple.
	\item $m \geq 1$ and $\calP_m$ is simple.
	\item $\calP_m$ is not simple.
\end{itemize}
In the first case,
we obtain a matching $I^*$ with $r(I^*) > r(I)$ in $O(|E|)$ as required
and the augmentation procedure terminates,
which is dealt with in \cref{subsec:base}.
In the second and third cases,
we basically modify $(I, \calT)$ in $O(|E|)$ time so that $\theta$ strictly decreases,
and then return to the initial stage.
In a certain situation of the second case,
we obtain a larger matching in $O(|E|)$ as required
and terminate the procedure.
They are dealt with in \cref{subsec:not simple,subsec:path},
respectively.
Since $\theta$ is bounded by $O(|E|)$ here,
the time complexity of the augmentation procedure is bounded by $O(|E|^2)$.
This implies \cref{thm:augmentation}.

\subsection{Base case: $\calT = \calP_0$ and $\calP_0$ is simple}\label{subsec:base}
Define
\begin{align*}
I^* := I \cup P_0,
\end{align*}
where $P_0$ is regarded as an edge set.
Then the following holds.
\begin{lem}\label{lem:I^*}
	$I^*$ satisfies {\rm (Deg)}, {\rm (q-Cycle)}, {\rm (VL)}, and $r(I^*) > r(I)$.
\end{lem}
\begin{proof}
	Suppose $\calP_0 = (Y_{1}, \beta_1\alpha_1, X_{1}, \alpha_1 \beta_2,\ldots, \beta_k \alpha_k,  X_k)$ ($k \geq 1$).
	One can easily check that $I^*$ satisfies (Deg), (q-Cycle), and $r(I^*) > r(I)$;
	(Deg) follows from $\deg_I(\gamma) \leq 1$ for all $\gamma$ belonging to $P_0$.
	We show that $I^*$ satisfies (VL),
	i.e., $I^*$ has a valid labeling.
	We can assume that, if $\deg_I(\alpha_k) = 1$ then $\alpha_k$ is incident 
	to a $-$-edge in $I$,
	i.e., $\beta_1 \alpha_1, \beta_2 \alpha_2, \dots, \beta_k \alpha_k$ are $+$-edges in $I'$.
	If $\deg_I(\alpha_k) = 1$,
	then define $U_{\alpha_k}^-$ as $\ker_I(\alpha_k)$.
	If $\deg_I(\alpha_k) = 0$,
	then define $U_{\alpha_k}^-$ as any 1-dimensional subspace of $U_{\alpha_k}$ different from $X_k$.
	Consider the back-propagation $P_0 \triangleleft U_{\alpha_k}^-$.
	Define $V^+_{\beta_{i}}$ as the propagated space at $\beta_i$
	and $U^-_{\alpha_{i}}$ as the propagated space at $\alpha_i$
	for all $i$.
	Note that the above propagated spaces are 1-dimensional
	and that $V_{\beta_i}^+ = \kerR(A_{\alpha_i \beta_i})$ if $\beta_i \alpha_i$ is rank-1.
	
	The other labeling $(U^+_{\alpha_i}, V^-_{\beta_i})$ 
	are defined as $(X_i,Y_i)$ if both $X_i$ and $Y_i$ are both 1-dimensional.
	In this case,
	by \cref{lem:diff}~(1),
	we have $U^+_{\alpha_i} \neq U^-_{\alpha_i}$ and $V^+_{\beta_i} \neq V^-_{\beta_i}$ for each $i$.
	If some $Y_i$ is 2-dimensional (equivalently $X_{i-1}$ is zero-dimensional),
	then 
	$\deg_I(\beta_1) = 0$ and $\beta_j \alpha_j$ is rank-2 for $j=1,2,\cdots,i-1$.
	In this case, consider the maximum index $i$ with this property
	and define $V_{\beta_i}^-$ as any 1-dimensional subspace different from $\kerR A_{\alpha_i \beta_i} = V_{\beta_i}^+$.
	Then
	define $U_{\alpha_{j-1}}^+ := (V_{\beta_j}^-)^{\perp_{\beta_j \alpha_{j-1}}}$
	and $V_{\beta_{j-1}}^- := (U_{\alpha_{j-1}}^+)^{\perp_{\alpha_{j-1} \beta_j}}$
	for each $j$.
	Note $U_{\alpha_i}^+ = \kerL(A_{\alpha_i \beta_i})$ if $\beta_i \alpha_i$ is rank-1.
	By a similar argument as \cref{lem:diff}~(1),
	we have $U^+_{\alpha_i} \neq U^-_{\alpha_i}$ and $V^+_{\beta_i} \neq V^-_{\beta_i}$.

	The resulting labeling is valid for $I^*$.
	Indeed, we have already seen $U^+_{\alpha_i} \neq U^-_{\alpha_i}$ and $V^+_{\beta_i} \neq V^-_{\beta_i}$.
	The orthogonal property~\eqref{eq:+-} is satisfied by the construction.
	As seen above,
	we have $(U_{\alpha_i}^+, V_{\beta_i}^+) = (\kerL(A_{\alpha_i \beta_i}), \kerR(A_{\alpha_i \beta_i}))$
	if $\beta_i \alpha_i$ is rank-1,
	which implies~\eqref{eq:++ --}.
	We need to consider the case where $\beta_1$ already belongs to a path component $C$.
	Say, $\beta_1$ is incident to $\beta_1 \alpha$ (with $\alpha \neq \alpha_1$) that is an end edge of $C$,
	then the validness requires that $\beta_1 \alpha$ is a $-$-edge in $I^*$.
	If $C$ contains $\alpha_k$, then $\beta_1 \alpha$ is necessarily a (rank-1) $-$-edge (with $Y_{\beta_1} = \kerR(A_{\alpha \beta_1})$).
	Otherwise we can assume by re-coloring that $\beta_1 \alpha$ is a $-$-edge.   
	Thus we obtain a valid labeling for $I^*$.
\end{proof}
We apply the elimination to $I^*$.
By \cref{lem:elimination},
the resulting set is a desired augmentation.
The augmentation procedure terminates.
This update
can be clearly done in $O(|E|)$ time.

\subsection{$\calP_m$ is simple and $m \geq 1$}\label{subsec:path}
Suppose that $m \geq 1$ and $\calP_m$
is simple.
We consider the following condition:
\begin{description}
	\item[($\Ni$)]
	For any $\ell \leq m-2$,
	there is no inner space-walk $\calQ$ such that
	it admits a compatible concatenation
	$\calP_{\ell} \circ \calQ$ with
	$\beta(\calQ) = \beta(\calP_m) = \beta(\calQ_m)$,
	and $Y(\calQ)$ is different from
	the initial space of $P_m \triangleleft \ker_I(\alpha(\calP_m))$.
\end{description}
Note that ($\Ni$) clearly holds if $m = 1$.

Suppose that ($\Ni$) is violated for some $\ell \leq m-2$.
Consider minimum such $\ell$.
Update $\calT$ as
\begin{align}\label{eq:T' Ninner}
\calT \leftarrow \calP_0 \circ \calQ_1 \circ \cdots \circ \calP_{\ell} \circ \calQ \circ (Y(\calQ) \triangleright P_m).
\end{align}
The resulting $\calT$ is a compatibly-concatenated space-walk for $I$.
($\Ai$) clearly holds.
Since $Y(\calQ)$ is different from the initial space of $P_m \triangleleft \ker_I(\alpha(\calP_m))$,
we have
($\Al$) by \cref{lem:diff}~(2).
Thus the resulting $\calT$ is an augmenting space-walk for $I$.

This update can be done in $O(|E|)$ time,
since it suffices to find $\calP_\ell$ and $\calQ$ violating ($\Ni$)
and compute the front-propagation.
Note that the last outer space-walk remains simple.

By re-index,
we let $\calT = \calP_0 \circ \calQ_1 \circ \calP_1 \circ \cdots \circ \calQ_m \circ \calP_m$ again.
(Note that the new $m$ can be smaller than the old $m$.)
Here $\theta$ can increase.
Let $\Delta$ denote the increase of $\theta$ in the modification~\eqref{eq:T' Ninner},
where $\Delta := 0$
if $\theta$ does not increase.
Then the following holds:
\begin{lem}\label{lem:Delta}
	$\Delta \leq |Q_m| - 1$.
\end{lem}
\begin{proof}
	It suffices to consider the case of $\Delta > 0$.
	Obviously, the update~\eqref{eq:T' Ninner} does not increase
	$N$.
	Let $F$ and $F'$ denote the set of inner-double edges in the old $Q_m$
	and the new $Q_m$,
	respectively.
	Then $\Delta \leq |F'| - |F|$ holds.
	Hence it suffices to prove $|F| > 0$ if $F'$ coincides with the new $Q_m$.
	We can assume that the initial edge of the new $\calQ_m$ is a $+$-edge.
	The equality $F' = Q_m$
	implies that
	there exists an inner space-walk $\calQ_{\ell}$ with $\ell < m$ in $\calT$ including $\calQ_m$.
	Then the initial edge of $\calQ_{\ell}$ must be a $-$-edge;
	otherwise $\calT$ violates ($\Ni$), a contradiction.
	Since $\calQ_{\ell}$ contains both the $+$-edge and the $-$-edge incident to $\beta(\calP_m)$,
	$\calQ_{\ell}$ intersects with the old $\calQ_m$.
	Hence $|F| > 0$,
	as required.
\end{proof}

Now $\calT$ satisfies ($\Ni$).
Let $C$ be the rank-1 connected component containing $Q_m$,
and $\beta^*$ denote the last vertex $\beta(\calQ_m)$ of $\calQ_m$.
We can assume that the last edge of $\calQ_m$ is a $+$-edge,
which implies $Y(\calQ_m) = V_{\beta(\calQ_m)}^-$.
Let $Q^+$ and $Q^-$ be the maximal inner walks in $C$ such that the last vertex is $\beta^*$ and the last edge is a $+$-edge and a $-$-edge,
respectively.
We denote by $\alpha^+$ and $\alpha^-$ the initial vertices of $Q^+$ and of $Q^-$,
respectively.
See e.g., Figure~\ref{fig:cycle}.
Let $Y^\vee$ be the initial space of $P_m \triangleleft \ker_I(\alpha(\calP_m))$.
Note that $V_{\beta^*}^-$ is different from $Y^\vee$.
Consider the following two cases:
\begin{itemize}
	\item $C$ is a cycle component.
	\item $C$ is a path component.
\end{itemize}
They are dealt with in \cref{subsubsec:cycle,subsubsec:path},
respectively.

\subsubsection{$C$ is a cycle component}\label{subsubsec:cycle}
There are additional two cases:
\begin{itemize}
	\item
	$C$ has a rank-1 $-$-edge.
	\item
	$C$ has no rank-1 $-$-edge.
\end{itemize}

\paragraph{Case 1: $C$ has a rank-1 $-$-edge.}
In this case,
it holds that $\alpha^+$ is incident to a rank-1 $-$-edge
which does not belong to $Q^+$.
Define
\begin{align}\label{eq:I' cycle case 1}
I' := (I \cup P_m) \setminus \{ \text{all $+$-edges in $Q^+$} \}.
\end{align}
\begin{lem}\label{lem:I' cycle case 1}
	$I'$ satisfies {\rm (Deg)}, {\rm (q-Cycle)}, {\rm (VL)}, and $r(I') = r(I)$.
\end{lem}
\begin{proof}
	One can easily check that $I'$ satisfies (Deg), (q-Cycle), and $r(I') = r(I)$.
	We show that $I'$ also satisfies (VL).
	We can assume that $\alpha(\calP_m)$ is incident to a $-$-edge in $I$,
	since $\alpha(\calP_m)$ does not belong to $C$.
	The labels on the vertices in $P_m$
	are determined in the same way as in the proof of \cref{lem:I^*}.
	
	Let $(\beta^* = \beta_1 \alpha_1, \alpha_1 \beta_2, \dots)$ be the maximal path in $C$ consisting of rank-2 edges
	such that the initial vertex is $\beta^*$ and the initial edge is a $-$-edge.
	Then we define $U_{\alpha_i}^- := (V_{\beta_i}^+)^{\perp_{\beta_i \alpha_i}}$
	and $V_{\beta_{i-1}}^+ := (U_{\alpha_i}^-)^{\perp_{\alpha_i \beta_{i-1}}}$
	for each $i$.
	On the other hand,
	we do not change
	$U_{\alpha_i}^+$ and $V_{\beta_i}^-$.
	Then this labeling is valid for $I'$.
	Indeed, by ($\Al$) and \cref{lem:diff}~(1),
	we have $V_{\beta}^- \neq V_{\beta}^+$.
	Hence $U_{\alpha_i}^+ \neq U_{\alpha_i}^-$ and $V_{\beta_i}^+ \neq V_{\beta_i}^-$ hold for each $i$.
\end{proof}

If $I'$ satisfies (Path),
then
$I'$ is a q-matching with $r(I') = r(I)$.
Otherwise
we apply the elimination operation to $I'$ from $\alpha(\calP_m)$
so that (Path) holds.
The resulting set, also denoted by $I'$, is a q-mathcing with $r(I') \geq r(I)$ (in fact $r(I') = r(I)$) by \cref{lem:elimination}.

We then modify $\calT$ to obtain an augmenting space-walk $\calT'$ for $I'$.
Define
\begin{align}\label{eq:T' cycle case 1}
\calT' := \calP_0 \circ \calQ_1 \circ \cdots \circ \calQ_{m-1} \circ \left(\calP_{m-1} \triangleright Q^+[\alpha(\calQ_m), \alpha^+]\right),
\end{align}
where $Q^+[\alpha(\calQ_m), \alpha^+]$ forms a truncated outer walk
since all $-$-edges in $Q^+$
are isolated rank-2 edges in $I'$.
See Figure~\ref{fig:cycle}.
\begin{figure}
	\centering
	\begin{tikzpicture}[
	node/.style={
		fill=black, circle, minimum height=5pt, inner sep=0pt,
	},
	I1/.style={
		line width = 3pt,
		decorate,
		decoration={snake, amplitude=.4mm,segment length=2.5mm,post length=0mm}
	},
	I2/.style={
		line width = 3pt
	},
	ID1/.style={
		line width = 3pt,
		blue,
		decorate,
		decoration={snake, amplitude=.3mm,segment length=2mm,post length=0mm}
	},
	ID2/.style={
		line width = 3pt,
		blue
	},
	NI/.style={
	},
	NIA2/.style={
		red,
		semithick
	},
	NIA1/.style={
		red,
		semithick,
		decorate,
		decoration={snake, amplitude=.3mm,segment length=1.5mm,post length=0mm}
	}
	]
	
	\def\mu{a1, a2, a3, a4, a5}
	\def\nu{b1, b2, b3, b4, b5, b6}
	\def\bmu{ba-2, ba-1, ba0, ba1, ba2, ba3, ba4, ba5}
	\def\bnu{bb-2, bb-1, bb0, bb1, bb2, bb3, bb4, bb5, bb6}
	\def\size{1.7cm}
	\def\hight{4cm}
	\def\side{5pt}

	\nodecounter{\bnu}
	\coordinate (pos);
	\foreach \currentnode in \bnu {
		\node[node, below=0 of pos, anchor=center] (\currentnode) {};
		\coordinate (pos) at ($(pos)+(-\size, 0)$);
	}
	\nodecounter{\bmu}
	\coordinate (pos) at ($(bb-2) + (-\size, -\size)$);
	\foreach \currentnode in \bmu {
		\node[node, below=0 of pos, anchor=center] (\currentnode) {};
		\coordinate (pos) at ($(pos)+(-\size, 0)$);
	}
	
	\nodecounter{\nu}
	\coordinate (pos) at ($(bb1) + (0, \hight)$);
	\foreach \currentnode in \nu {
		\node[node, below=0 of pos, anchor=center] (\currentnode) {};
		\coordinate (pos) at ($(pos)+(-\size, 0)$);
	}
	\nodecounter{\mu}
	\coordinate (pos) at ($(b1) + (-\size, -\size)$);
	\foreach \currentnode in \mu {
		\node[node, below=0 of pos, anchor=center] (\currentnode) {};
		\coordinate (pos) at ($(pos)+(-\size, 0)$);
	}

	\draw[I1] (bb-2) -- (ba-2);
	\draw[I1] (bb1) -- node[left = -1pt]{$-$} (ba0);
	\draw[I1] (b6) -- node[left]{$-$} (a5);
	\draw[I1] (bb6) -- node[left]{$-$} (ba5);
	
	\foreach \i / \j in {bb0/ba-1, bb2/ba1, bb3/ba2, bb5/ba4, b2/a1, b3/a2, b4/a3, b5/a4} {
		\draw[I2] (\i) -- (\j);
	}
	\draw[I2] (bb4) -- node[right = -2pt]{$-$} (ba3);
	\draw[I2] (bb-1) -- node[right = -2pt]{$-$} (ba-2);
	\draw[I2] (bb-1) -- node[above left = -3pt and -3pt]{$+$} (ba-1);
	\draw[I2] (bb0) -- node[below]{$+$} (ba0);
	
	\foreach \i / \j in {b2/a2, b3/a3} {
		\draw[NIA2] (\i) -- (\j);
	}
	\draw[NIA2] (b1) -- node[above]{$+$} (a1);
	\draw[NIA2] (b5) -- node[above]{$+$} (a5);
	
	\draw[NIA1] (b4) -- (a4);
	
	\foreach \i / \j in {bb2/ba2, bb5/ba5} {
		\draw[ID2] (\i) -- (\j);
	}
	
	\draw[ID2] (bb3) -- node[above]{$+$} (ba3);
	\draw[ID2] (bb1) -- node[below right = 0 and -5pt]{$+$} (ba1);

	\draw[ID1] (bb4) -- (ba4);

	\coordinate [label=below:{$\alpha(\calP_m)$}] () at (a5);
	\coordinate [label=below:{$\alpha^+$}] () at (ba5);
	
	\draw[dotted, semithick] (b1) -- (bb1);
	\coordinate [label=above right:{$\beta^*$}] () at (b1);
	\coordinate [label=above right:{$\beta^*$}] () at (bb1);
	
	\foreach \i in {a1, a2, a3, a4, a5, b1, b2, b3, b4, b5, b6, ba-2, ba-1, ba0, ba1, ba2, ba3, ba4, ba5, bb-2, bb-1, bb0, bb1, bb2, bb3, bb4, bb5, bb6} {
		\coordinate [left = \side] (l\i) at (\i);
		\coordinate [right = \side] (r\i) at (\i);
	}

	\draw[[-{Latex[length=3mm]}, dashed, thick] ($(rb1)!0.1!(ra1)$) -- node[above left = 0 and -4pt]{$\calP_m$} ($(rb1)!0.6!(ra1)$);
	\draw[[-{Latex[length=3mm]}, dashed, thick] ($(rb5)!0.5!(ra5)$) -- node[below right = 0 and -4pt]{$\calP_m$} (ra5);
	
	\coordinate [below = 1cm] (Pms) at (ra2);
	\coordinate [label=below:{$\calP_{\ell}$}] () at (Pms);
	\coordinate [above = 1.5cm] (Pmt) at (rb4);
	\draw[[-{Latex[length=3mm]}, thick] (Pms) -- (ra2) -- (rb3) -- (ra3) -- (rb4) -- (Pmt);
	
	\coordinate [below = 1cm] (oPm-1s) at (rba3);
	\coordinate [label=below:{$\calP_{m-1}$}] () at (oPm-1s);
	\draw[[-{Latex[length=3mm]}, dashed, thick] (oPm-1s) -- (rba3);
	\draw[[-{Latex[length=3mm]}, dashed, thick] (rba3) --node[below = 3pt]{$\calQ_m$} (rbb3) -- (rba2) -- (rbb2) -- (rba1) -- ($(rbb1)!0.1!(rba1)$);
	
	\coordinate [below = 1cm] (nPm-1s) at (lba3);
	\draw[[-{Latex[length=3mm]}, thick] (nPm-1s) -- (lba3) -- (lbb4) -- (lba4) -- (lbb5) -- ($(lba5)!0.1!(lbb5)$);
	
	\coordinate [above = 5pt] (Q+s) at ($(lba5)!0.1!(lbb5)$);
	\coordinate [above = 5pt] (Q+st) at (lbb5);
	\draw[[-{Latex[length=3mm]}, dotted, thick] (Q+s) -- node[above]{$Q^+$} ($(Q+s)!0.7!(Q+st)$);
	
	\coordinate (Q+ts) at ($(Q+s) + (4 * \size, 0)$);
	\coordinate (Q+t) at ($(Q+st) + (4 * \size, 0)$);
	\draw[[-{Latex[length=3mm]}, dotted, thick] ($(Q+ts)!0.4!(Q+t)$) -- node[above]{$Q^+$} (Q+t);
	
	\coordinate [left = 5pt] (Q-s) at (ba-2);
	\coordinate [above = 1cm] (Q-st) at (Q-s);
	\draw[[-{Latex[length=3mm]}, dotted, thick] (Q-s) -- node[left = -4pt]{$Q^-$} (Q-st);
	
	\coordinate [right = 5pt] (Q-t) at (bb1);
	\coordinate [below = 1cm] (Q-ts) at (Q-t);
	\draw[[-{Latex[length=3mm]}, dotted, thick] (Q-ts) -- node[right = -3pt]{$Q^-$} (Q-t);
	
	\coordinate [below = 1cm] (Pls) at (lba0);
	\coordinate [label=below:{$\calP_{\ell-1}$}] () at (Pls);
	\draw[[-{Latex[length=3mm]}, thick] (Pls) -- (lba0);
	\draw[[-{Latex[length=3mm]}, thick] (lba0) -- (lbb0) -- node[left = -3pt]{$\calQ_{\ell}$} (lba-1) -- (lbb-1);
	\coordinate [above = 1.5cm] (Plt) at (lbb-1);
	\draw[[-{Latex[length=3mm]}, thick] (lbb-1) -- node[right]{$\calP_\ell$} (Plt);
	
	\coordinate[label=below:{$\alpha^-$}] () at ($(ba-2) + (0, -1pt)$);
	
	\coordinate[label=left:$\cdots$] () at ($(bb6)!0.5!(ba5) + (-0.5cm, 0)$);
	\coordinate (ba-3) at ($(ba-2) + (\size, 0)$);
	\coordinate[label=right:$\cdots$] () at ($(bb-2)!0.5!(ba-3)$);
	
	\end{tikzpicture}
	\caption{
		Modification in \cref{subsubsec:cycle}.
		The thick lines and the thin lines represent edges in $I$ and in $E \setminus I$,
		respectively.
		The solid lines and the wavy lines represent rank-2 edges and rank-1 edges,
		respectively.
		The blue lines and the red lines represent deleted edges from $I$
		and added edges to $I$ by the modification~\eqref{eq:I' cycle case 1},
		respectively.
		The dashed paths and the solid paths represent outer/inner space-walks in $\calT$ and in $\calT'$,
		respectively.
	}
	\label{fig:cycle}
\end{figure}
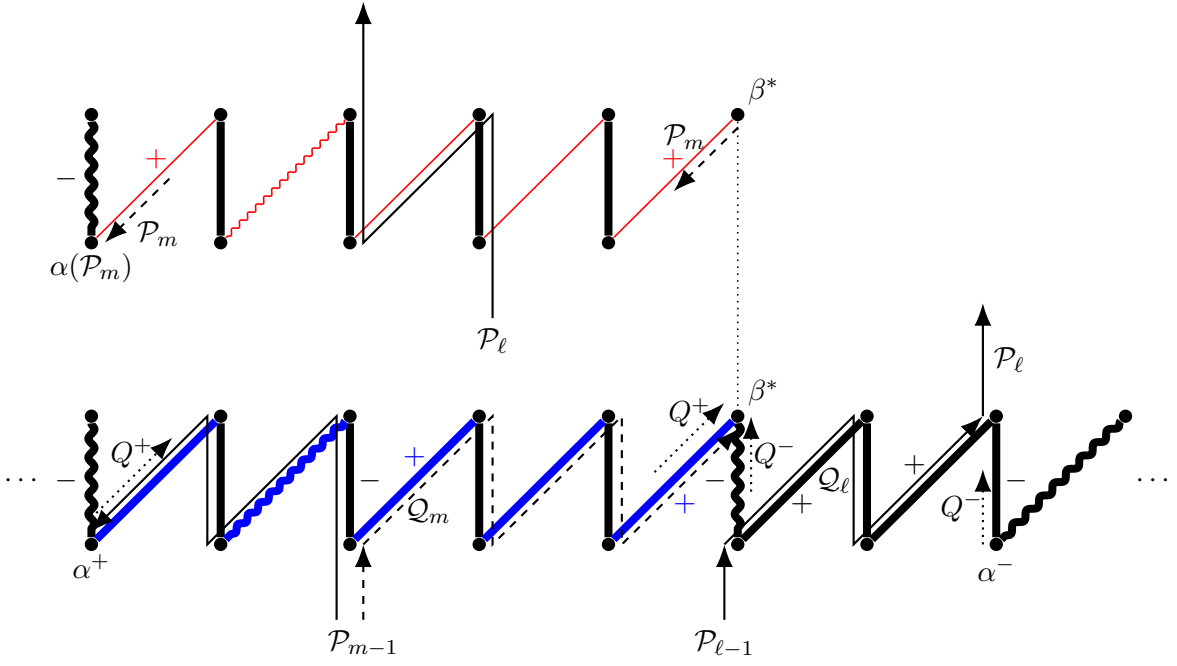
Then the following holds.
\begin{lem}\label{lem:T' cycle case 1}
	$\calT'$ is an augmenting space-walk for $I'$.
\end{lem}
\begin{proof}
	$\calT'$ clearly satisfies ($\Ai$).
	By $X(\calP_{m-1}) \neq U_{\alpha(\calP_{m-1})}^-$ and \cref{lem:diff}~(2),
	$X(\calP_{m-1}) \triangleright Q^+[\alpha(\calQ_m), \alpha^+]$ forms an outer space-walk for $I'$,
	and
	the last space of $X(\calP_{m-1}) \triangleright R$ is different from $U_{\alpha^+}^-$
	and from $\ker_{I'}(\alpha^+) (\supseteq U_{\alpha^+}^-)$.
	Thus $\calT'$ satisfies ($\Al$).
	
	The new q-matching $I'$ is obtained from $I$ by addition and deletion.
	Therefore some inner and outer space-walks for $I$ are no longer those for $I'$.  
	In particular, $\calP_{\ell}$ and $\calQ_{\ell}$ ($\ell \leq m-1$) in~\eqref{eq:T' cycle case 1}
	are not necessarily outer and inner space-walks for $I'$.
	We however show that such space-walks can admit another compatible concatenation structure, which 
	results that $\calT'$ is a compatibly-concatenated space-walk for $I'$.

	(1).
	First we examine the effect of the relabeling of $V^+_{\beta}$ and $U^-_{\alpha}$
	on 
	the subwalk of $Q^-$ of rank-2 edges ending $\beta^*$ in the case of $Y^\vee \neq V_{\beta^*}^+$.
	We may assume that the elimination cannot reach $C$ 
	(otherwise the relabeled vertices belong to isolated rank-2 edges in $I'$; see (2) below).
	In this case, by ($\Ni$), there is no outer space-walk $\calP_{\ell}$
	ending vertex $\alpha$ in $Q^-$ with the last space $X(\calP_{\ell}) \neq U_{\alpha}^+$.
	(Otherwise we could concatenate $\calP_\ell \circ \calQ^-[\alpha, \beta^*] \circ \calP_m$ compatibly,
	violating ($\Ni$).)
	Consequently, if $\calQ_{\ell}$ meets $Q^-$, 
	then it is still an inner space-walk for $I'$ with unchanged labels; 
	it enters $\alpha$ with $X(\calQ_\ell) = U_\alpha^+ = X(\calP_{\ell-1})$,
	leaves $\beta$ with $Y(\calQ_\ell) = V_\beta^-$,
	and does not meet $Q^+$ (even if $C$ is a cycle).
	
	(2).
	Second we consider the effect of addition of $P_m$ to $I'$, 
	particularly when a previous outer walk $P_\ell$ shares edges in $P_m$.
	(We see in~(4) below that such edges cannot be eliminated.)
	The intersection $P_m \cap P_\ell$
	is the disjoint union of subwalks
	(each of which must consist of rank-2 edges by the irredundancy).
	Consider a subwalk starting $\alpha$ in $P_m$,  
	which has (newly assigned) label $U_\alpha^+$ at $\alpha$ by ($\No$)
	and runs abreast with $P_m$ and labels $U^+_{\bullet}$ and $V^-_{\bullet}$.
	It leaves $P_m$ at $\beta$, or it ends at $\alpha(\calP_m)$ with label $\ker_I(\alpha(\calP_m))$.
	For the former case, $\calP_\ell[\alpha,\beta]$ is an inner space-walk for $I'$
	in which $\calP_\ell(\alpha] \circ \calP_\ell[\alpha,\beta] \circ \calP_\ell [\beta)$ is
	compatibly concatenated.
	For the latter case,
	since $C$ is a cycle,
	we have $\alpha(\calP_m) \neq \alpha^-$
	and
	$\calP_\ell[\alpha,\alpha(\calP_m)]$ is compatibly concatenated with $\calQ_{\ell+1}$
	to become a single inner space-walk for $I'$.
	The other possible subwalk starts at initial vertex $\beta^*$ of $P_m$.
	In this case, $\calQ_{\ell}$ stretches to a single inner space-walk $\calQ_\ell \circ \calP_{\ell}[\beta)$ 
	or $\calQ_\ell \circ \calP_{\ell} \circ \calQ_{\ell +1}$, compatibly concatenated with the next space-walk.

	(3).
	Third we consider the effect of deletion of the $+$-edges in $Q^+$ from $I$.
	Suppose that $Q_\ell$ shares edges in $Q^+$.
	By (2) above and re-index,
	we can regard $\calP_{\ell-1}$ and $\calP_\ell$
	as outer space-walks for $I'$.
	By ($\Ni$),
	the intersection $Q_\ell \cap Q^+$ is a subwalk 
	such that it starts at the $-$-edge $\alpha(\calQ_\ell) \beta'$ with $X(\calQ_\ell) = U_{\alpha(\calQ_\ell)}^- = X(\calP_{\ell-1})$
	or it starts at the $+$-edge $\beta^* \alpha'$ with the subspace $V_{\beta^*}^+$ at $\beta^*$.
	For the former case, 
	$\calP_{\ell-1} \circ \calQ_\ell \circ \calP_\ell$ is a single outer space-walk for $I'$ if it leaves a vertex $\beta$ in $Q^+$, and
	$\calP_\ell \circ \calQ_\ell = (\calP_{\ell} \circ \calQ_\ell(\alpha^+]) \circ \calQ_\ell[\alpha^+)$ 
	is a compatible concatenation of outer and inner paths for $I'$
	otherwise (i.e., it ends at $\alpha^+$).
	The latter case is similar: $\calQ_\ell \circ \calP_\ell = \calQ_\ell(\beta^*] \circ (\calQ_\ell [\beta^*) \circ \calP_\ell)$
	is a compable concatenation of inner and outer paths, or 
	$\calQ_\ell = \calQ_\ell(\beta^*] \circ \calQ_\ell [\beta^*, \alpha^+] \circ \calQ_\ell [\alpha^+)$ 
	is that of inner, outer, and inner space-walks.

	(4).
	Finally we consider the effect of elimination operation.
	It occurs precisely when the last vertex $\alpha(\calP_m)$ of $P_m$ is not incident to any edge in $I$ and
	the last edge of $P_m$ is rank-2.
	By ($\No$),
	removed edges in $P_m$ cannot be shared by any previous outer walk $P_\ell$.
	Furthermore, if all edges in $P_m$ are rank-2, then the elimination enters $C$ 
	and removes all $-$-edges in $Q^-$. 
	Hence we have to consider the situation where some inner space-walk $\calQ_\ell$ meets $Q^-$. 
	By the same argument as~(1),
	($\Ni$) implies that $Q_\ell$ must start at $+$-edge in $Q^-$ 
	with $X(\calQ_\ell) = U_{\alpha(\calQ_\ell)}^+ = X(\calP_{\ell-1})$,
	leaves $\beta(\calQ_\ell)$ with $Y(\calQ_\ell) = V_{\beta(\calQ_\ell)}^-$
	(and does not meet $Q^+$ even if $C$ is a cycle).
	Now, if $Q_\ell$ is included in $Q^-$,
	then
	$\calP_{\ell-1} \circ \calQ_\ell \circ \calP_{\ell}$ is a single outer space-walk for $I'$.
	Otherwise $Q_\ell$ passes through the initial vertex $\alpha^-$ of $Q^-$.
	Then
	$\calP_{\ell-1} \circ \calQ_\ell(\alpha]$
	is a single outer space-walk compatible with the inner space-walk $\calQ_\ell[\alpha)$.

	Summarizing,  $\calT'$ is viewed as a compatibly-concatenated space-walk for $I'$. 
\end{proof}

Let $(I, \calT) \leftarrow (I', \calT')$.
This update can be done in $O(|E|)$ time,
since
we can compute the front-propagation
in $O(|E|)$ time.
The following holds on $\theta$:
\begin{lem}\label{lem:theta cycle case 1}
	$\theta$ decreases by at least $|Q_m|$.
\end{lem}
\begin{proof}
	An edge $\alpha \beta \in Q_m$ contributes only to $N$ if $\alpha \beta$ was used only once in $\calT$,
	and to $N$ and $D_{\rm inner}$
	if $\alpha \beta$ was used twice in $\calT$.
	In the former case,
	by the transformation~\eqref{eq:T' 1 case 1},
	$\alpha \beta$ exits the extended support of $\calT$,
	decreasing $N$ by one.
	In the latter case,
	$\alpha \beta$ exits an inner space-walk,
	decreasing $D_{\rm inner}$ by one.
\end{proof}

By \cref{lem:theta cycle case 1},
the quantity of $\theta$ decrease by at least $|Q_m|$.
Hence
the difference of $\theta$ between the one before~\eqref{eq:T' Ninner} and the one after~\eqref{eq:T' cycle case 1}
is bounded by $\Delta - |Q_m|$;
$\theta$ strictly decreases by \cref{lem:Delta}.
Return to the initial stage (\cref{subsec:initial stage}).

\paragraph{Case~2: $C$ has no rank-1 $-$-edge.}
Define
\begin{align*}
I' := I \setminus \{ \text{all $+$-edges in $C$} \}.
\end{align*}
It is clear that $I'$ is a q-matching with $r(I') = r(I)$.
Suppose
that
$C$ consists of the disjoint union of $Q_m$ and $Q$.
Note that $Q \circ P_m$ is a truncated outer walk for $I'$.

Suppose that $\calP_{m-1} \triangleright (Q \circ P_m)$
forms a single outer space-walk for $I'$
and its last space is different from $\ker_I(\alpha(\calP_m)) = \ker_{I'}(\alpha(\calP_m))$.
Then define
\begin{align}\label{eq:T' 1 case 1}
\calT' := \calP_0 \circ \calQ_1 \circ \cdots \circ \calQ_{m-1} \circ \left(\calP_{m-1} \triangleright (Q \circ P_m)\right).
\end{align}
This is an augmenting space-walk for $I'$.
Indeed, ($\Ai$) is obvious.
By ($\Ni$),
every inner space-walk $\calQ_\ell$ ($\ell \leq m-1$)
sharing edges in $C$
starts at a $-$-edge with $X(\calQ_\ell) = U_{\alpha(\calQ_\ell)}^- = X(\calP_{\ell -1})$.
By the same argument as (3) in \cref{lem:T' cycle case 1},
$\calP_{\ell - 1} \circ \calQ_\ell \circ \calP_\ell$
forms a single outer space-walk for $I'$.
($\Al$) immediately follows from the assumption on $\calP_{m-1} \triangleright (Q \circ P_m)$.

Let $(I, \calT) \leftarrow (I', \calT')$.
This update can be done in $O(|E|)$ time,
since
we can compute the front-propagation
in $O(|E|)$ time.
By the same argument as in \cref{lem:theta cycle case 1},
$\theta$ strictly decreases.
Return to the initial stage (\cref{subsec:initial stage}).

Suppose not.
In this case,
we obtain an augmenting space-walk $\calT'$ for $I$ ({\it not for $I'$})
such that the situation reduces to Case~1.
By the assumption,
the propagated space $Y$ of $\calP_{m-1} \triangleright (Q \circ P_m)$ at $\beta^* = \beta(\calQ_m) = \beta(\calQ)$ coincides with $Y^\vee$.
Note $Y^\vee \neq V_{\beta^*}^-$.
We also have $Y^\vee \neq V_{\beta^*}^+$.
Indeed,
by the compatibility of $(\calP_{m-1}, \calQ_m)$,
we have $X(\calP_{m-1}) \neq U_{\alpha(\calQ_m)}^-$.
Hence $Y$ is different from $V_{\beta^*}^+$.
It also holds that $X(\calP_{m-1}) \neq U_{\alpha(\calQ_m)}^+$
and no rank-1 edge exists in $Q$.
Indeed,
otherwise
the propagated space of $\calP_{m-1} \triangleright (Q \circ P_m)$ at $\beta^*$ is $V_{\beta^*}^- \neq Y^\vee$.
Thus define
\begin{align}\label{eq:T' 1 case 2}
\calT' := \calP_0 \circ \calQ_1 \circ \cdots \circ \calP_{m-1} \circ \calQ \circ (V_{\beta^*}^+ \triangleright P_m),
\end{align}
where $\calQ = (U_{\alpha_1}^-, \alpha_1 \beta_1, V_{\beta_1}^+, \beta_1 \alpha_2, \dots, \alpha_k \beta_k, V_{\beta_k}^+ = V_{\beta^*}^+)$
is the inner space-walk corresponding to $Q$.
Then $\calT'$ is an augmenting space-walk for $I$.
Indeed, ($\Ai$) clearly holds.
The compatibility of $(\calP_{m-1}, \calQ)$ follows from $X(\calP_{m-1}) \neq U_{\alpha(\calQ_m)}^+$.
By $Y^\vee \neq V_{\beta^*}^+$ and \cref{lem:diff},
we have ($\Al$).

Let $(I, \calT) \leftarrow (I', \calT')$.
This update can be done in $O(|E|)$ time,
since
we can construct the inner space-walk $\calQ$
and compute the front-propagation $V_{\beta^*}^+ \triangleright P_m$
in $O(|E|)$ time.
The following holds on $\theta$:
\begin{lem}\label{lem:dec 1 case 2}
	$\theta$ decreases by at least $\Delta$.
\end{lem}
\begin{proof}
	By the proof of \cref{lem:Delta},
	$\Delta$ is at most the number $\tilde{\Delta}$ of inner-double edges in $\calQ_m$.
	By ($\Ni$) and $V_{\beta^*}^+ \neq Y^\vee \neq V_{\beta^*}^-$,
	there is no outer space-walk $\calP_{\ell}$ with $\ell < m-1$
	that intersects $Q$.
	Hence,
	in the transformation~\eqref{eq:T' 1 case 2},
	the last inner walk $\calQ$ of $\calT'$ has no inner-double edge by ($\Ni$).
	Thus $\theta$ decreases by at least $\tilde{\Delta}$.
\end{proof}
By \cref{lem:Delta,lem:dec 1 case 2},
$\theta$ is not greater than that in the beginning of \cref{subsec:path}.
Now $\calT$ satisfies ($\Ni$).
Since there was at least one $+$-edge in $C$ by (q-Cycle)
and the last edge of the last inner space-walk $\calQ$ of $\calT$ is a $-$-edge,
the resulting augmenting space-walk $\calT$ also satisfies the assumption of Case~1.
Go to Case~1 (with $\Delta = 0$).

\subsubsection{$C$ is a path component}\label{subsubsec:path}
In this case, as in the previous case (\cref{subsubsec:cycle}),
let $I'$ be defined by~\eqref{eq:I' cycle case 1}.
\begin{lem}\label{lem:I' path}
	$I'$ satisfies {\rm (Deg)}, {\rm (q-Cycle)}, {\rm (VL)}, and $r(I') = r(I)$.
\end{lem}
\begin{proof}
	In addition to the proof of \cref{lem:I' cycle case 1},
	we need to consider the case where
	$\alpha(\calP_m)$ belongs to $C$
	and $\alpha(\calP_m)$ is incident to a $+$-edge $\alpha(\calP_m) \beta'$ of $C$ in $I$.
	In this case, $\alpha(\calP_m)$ and $\alpha^+$ belong to the same connected component $C'$ of $I \setminus \{ \text{all $+$-edges in $Q^+$} \}$;
	note that $\beta^*$ does not belong to $C'$.
	Thus, after deleting the $+$-edges in $Q^+$,
	we swap the signs $+$ and $-$ for all edges in $C'$ and the corresponding valid labelings
	so that $\alpha(\calP_m) \beta'$ is a $-$-edge.
	This coloring matches that of the resulting component containing $\beta^*$.
	Then the labels on the vertices in $P_m$
	are determined in the same way as in the proof of \cref{lem:I' cycle case 1}.
\end{proof}
If $I'$ satisfies (Path),
then
$I'$ is a q-matching with $r(I') = r(I)$.
Otherwise
we apply the elimination operation to $I'$ from $\alpha(\calP_m)$.
The resulting set, also denoted by $I'$, is a q-mathcing with $r(I') \geq r(I)$ by \cref{lem:elimination}.

We consider additional two patterns:
\begin{itemize}
	\item[(A)]
	$\alpha(\calP_m) = \alpha^+$.
	\item[(B)]
	$\deg_I(\beta^*) = 1$.
\end{itemize}
If neither (A) nor (B) holds,
then define $\calT'$ as~\eqref{eq:T' cycle case 1}.
If $\alpha(\calP_m) \neq \alpha^-$,
then, by the same argument as in \cref{lem:T' cycle case 1} (Case~1 of \cref{subsubsec:cycle}),
$\calT'$ is an augmenting space-walk for $I'$.
If $\alpha(\calP_m) = \alpha^-$,
then we need to consider the effect of addition of $P_m$ to $I'$
in addition to \cref{lem:T' cycle case 1}~(2).
Suppose that a previous outer walk $P_\ell$ shares edges in $P_m$.
Then it can happen that $P_m$ ends at $\alpha^-$ with label $\ker_I(\alpha^-)$
since $C$ is a path component.
In this case, the maximal common suffix of $P_m$ and $P_\ell$ is compatibly concatenated with $\calQ_{\ell+1} \cap \calQ^-$
to become a single inner space-walk for $I'$.
By the same argument as in \cref{lem:theta cycle case 1},
the difference of $\theta$ between the one before~\eqref{eq:T' Ninner} and the one after this update
is bounded by $\Delta - |Q_m|$;
$\theta$ strictly decreases by \cref{lem:Delta}.
Return to the initial stage (\cref{subsec:initial stage}).

In the following,
we consider the case where (A) or (B) holds.

\paragraph{Case~1: Only (A) holds.}
Suppose first that
the last edge of $\calP_m$ is rank-1.
In this case, the elimination does not occur.
Define $\calT'$ as~\eqref{eq:T' cycle case 1}.
Then $\calT'$ is an augmenting space-walk for $I'$.
Indeed, $\calT'$ satisfies ($\Ai$) and is a compatibly-concatenated space-walk for $I'$
by the same argument as in \cref{lem:T' cycle case 1}.
(A) implies that $\deg_I(\alpha^+) = 1$ and $\alpha^+$ is incident to a rank-1 in $I$.
Hence the last space of $\calP_{m-1} \triangleright Q^+[\alpha(\calQ_m), \alpha^+]$ is $\ker_I(\alpha^+) \neq \{0\}$.
On the other hand,
since $\calT$ satisfies ($\Al$) for $I$,
$X(\calP_m)$ is different from $\ker_I(\alpha^+)$.
By $\ker_{I'}(\alpha^+) = \ker_{I'}(\alpha(\calP_m)) = X(\calP_m)$,
$\calT'$ satisfies ($\Al$) for $I'$.

Suppose next that the last edge of $\calP_m$ is rank-2.
Then the elimination from $\alpha(\calP_m)$ occurs;
let $D$ denote the set of removed edges by the elimination.
Consider the path $P_m \cup Q^-$.
Denote by $\alpha^*$ the farthest vertex from $\alpha(\calP_m)$
in $P_m \cup Q^-$ such that it is incident to a deleted edge in $D$.
Let $R$ be the subpath in $P_m \cup Q^-$ from $\alpha^+ = \alpha(\calP_m)$ to $\alpha^*$.
Note here that its direction is the reserve of $P_m$,
and $\ker_{I'}(\alpha^*) \neq \{0\}$.
See Figure~\ref{fig:path case 1}.
\begin{figure}
	\centering
	\begin{tikzpicture}[
	node/.style={
		fill=black, circle, minimum height=5pt, inner sep=0pt,
	},
	I1/.style={
		line width = 3pt,
		decorate,
		decoration={snake, amplitude=.4mm,segment length=2.5mm,post length=0mm}
	},
	I2/.style={
		line width = 3pt
	},
	ID1/.style={
		line width = 3pt,
		blue,
		decorate,
		decoration={snake, amplitude=.3mm,segment length=2mm,post length=0mm}
	},
	ID2/.style={
		line width = 3pt,
		blue
	},
	NI/.style={
	},
	NIA2/.style={
		red,
		semithick
	},
	NIA1/.style={
		red,
		semithick,
		decorate,
		decoration={snake, amplitude=.3mm,segment length=1.5mm,post length=0mm}
	}
	]
	
	\def\mu{a1, a2, a3, a4, a5}
	\def\nu{b1, b2, b3, b4, b5}
	\def\bmu{ba-2, ba-1, ba0, ba1, ba2, ba3, ba4, ba5}
	\def\bnu{bb-2, bb-1, bb0, bb1, bb2, bb3, bb4, bb5}
	\def\size{1.7cm}
	\def\hight{4cm}
	\def\side{5pt}

	\nodecounter{\bnu}
	\coordinate (pos);
	\foreach \currentnode in \bnu {
		\node[node, below=0 of pos, anchor=center] (\currentnode) {};
		\coordinate (pos) at ($(pos)+(-\size, 0)$);
	}
	\nodecounter{\bmu}
	\coordinate (pos) at ($(bb-2) + (-\size, -\size)$);
	\foreach \currentnode in \bmu {
		\node[node, below=0 of pos, anchor=center] (\currentnode) {};
		\coordinate (pos) at ($(pos)+(-\size, 0)$);
	}
	
	\nodecounter{\nu}
	\coordinate (pos) at ($(bb1) + (0, \hight)$);
	\foreach \currentnode in \nu {
		\node[node, below=0 of pos, anchor=center] (\currentnode) {};
		\coordinate (pos) at ($(pos)+(-\size, 0)$);
	}
	\nodecounter{\mu}
	\coordinate (pos) at ($(b1) + (-\size, -\size)$);
	\foreach \currentnode in \mu {
		\node[node, below=0 of pos, anchor=center] (\currentnode) {};
		\coordinate (pos) at ($(pos)+(-\size, 0)$);
	}

	\draw[I1] (bb-2) -- (ba-2);
	\draw[ID2] (bb1) -- (ba0);
	\draw[ID2] (bb0) -- (ba-1);
	
	\foreach \i / \j in {bb2/ba1, bb3/ba2, bb5/ba4} {
		\draw[I2] (\i) -- (\j);
	}
	\draw[I2] (bb4) -- node[right = -2pt]{$-$} (ba3);
	\draw[ID2] (bb-1) -- node[left]{$-$} (ba-2);
	\draw[I2] (bb-1) -- (ba-1);
	\draw[I2] (bb0) -- node[below = 2pt]{$+$} (ba0);
	
	\draw[NIA2] (b2) -- (a2);
	\draw[NIA2] (b1) -- (a1);
	\draw[NIA2] (b5) -- (a5);
	\draw[NIA2] (b4) -- (a4);
	\draw[ID2] (b2) -- (a1);
	\draw[ID2] (b3) -- (a2);
	\draw[ID2] (b5) -- (a4);
	\draw[ID2] (b4) -- (a3);
	
	\draw[NIA2] (b3) -- (a3);
	
	\draw[ID2] (bb2) -- (ba2);
	\draw[ID1] (bb5) -- (ba5);
	
	\draw[ID2] (bb3) -- node[above]{$+$} (ba3);
	\draw[ID2] (bb1) -- node[below right = 0 and -5pt]{$+$} (ba1);

	\draw[ID1] (bb4) -- (ba4);
	
	\coordinate [label=below:{$\alpha$}] () at (a1);
	\coordinate [label=left:{$\alpha(\calP_m)$}] () at (a5);
	\coordinate [label=left:{$\alpha^+$}] () at (ba5);
	
	\draw[dotted, semithick] (b1) -- (bb1);
	\draw[dotted, semithick] (a5) -- (ba5);
	\coordinate [label=above right:{$\beta^*$}] () at (b1);
	\coordinate [label=above right:{$\beta^*$}] () at (bb1);
	
	\foreach \i in {a1, a2, a3, a4, a5, b1, b2, b3, b4, b5, ba-2, ba-1, ba0, ba1, ba2, ba3, ba4, ba5, bb-2, bb-1, bb0, bb1, bb2, bb3, bb4, bb5} {
		\coordinate [left = \side] (l\i) at (\i);
		\coordinate [right = \side] (r\i) at (\i);
	}

	\draw[[-{Latex[length=3mm]}, dashed, thick] ($(rb1)!0.1!(ra1)$) -- node[below right = -2pt and -4pt]{$\calP_m$} ($(rb1)!0.6!(ra1)$);
	\draw[[-{Latex[length=3mm]}, dashed, thick] ($(rb5)!0.5!(ra5)$) -- node[below right = 0 and -4pt]{$\calP_m$} (ra5);
	
	\coordinate [below = 1cm] (Pms) at (ra1);
	\coordinate [above = 1.5cm] (Pmt) at (rb3);
	\coordinate[label=below:{$\calP_\ell$}] () at (Pms);
	\draw[[-{Latex[length=3mm]}, dashed, thick] (Pms) -- (ra1) -- (rb2) -- (ra2) -- (rb3) -- node[right]{(p1)} (Pmt);
	
	\coordinate [below = 1cm] (oPm-1s) at (rba3);
	\coordinate[label=below:{$\calP_{m-1}$}] () at (oPm-1s);
	\draw[[-{Latex[length=3mm]}, dashed, thick] (oPm-1s) -- (rba3);
	\draw[[-{Latex[length=3mm]}, dashed, thick] (rba3) --node[below = 3pt]{$\calQ_m$} (rbb3) -- (rba2) -- (rbb2) -- (rba1) -- ($(rbb1)!0.1!(rba1)$);
	
	\coordinate [below = 1cm] (nPm-1s) at (lba3);
	\draw[[-{Latex[length=3mm]}, thick] (nPm-1s) -- (lba3) -- (lbb4) -- (lba4) -- (lbb5) -- ($(lba5)!0.1!(lbb5)$);
	\draw[[-{Latex[length=3mm]}, thick] ($(la5)!0.1!(lb5)$) -- (lb5) -- (la4) -- (lb4) -- (la3) -- (lb3) -- (la2) -- (lb2) -- (la1) -- (lb1);
	\draw[[-{Latex[length=3mm]}, thick] (lbb1) -- (lba0) -- (lbb0) -- (lba-1) -- (lbb-1) -- (lba-2);
	
	\coordinate [above = 5pt] (Q+s) at ($(lba5)!0.1!(lbb5)$);
	\coordinate [above = 5pt] (Q+st) at (lbb5);
	\draw[[-{Latex[length=3mm]}, dotted, thick] (Q+s) -- node[above]{$Q^+$} ($(Q+s)!0.7!(Q+st)$);
	
	\coordinate (Q+ts) at ($(Q+s) + (4 * \size, 0)$);
	\coordinate (Q+t) at ($(Q+st) + (4 * \size, 0)$);
	\draw[[-{Latex[length=3mm]}, dotted, thick] ($(Q+ts)!0.4!(Q+t)$) -- node[above]{$Q^+$} (Q+t);
	
	\coordinate [right = 5pt] (Q-s) at (ba-2);
	\coordinate [above = 1.3cm] (Q-st) at (Q-s);
	\draw[[-{Latex[length=3mm]}, dotted, thick] ($(Q-s)!0.2!(Q-st)$) -- node[right = -3pt]{$Q^-$} (Q-st);
	
	\coordinate [right = 5pt] (Q-t) at (bb1);
	\coordinate [below = 1cm] (Q-ts) at (Q-t);
	\draw[[-{Latex[length=3mm]}, dotted, thick] (Q-ts) -- node[right = -3pt]{$Q^-$} (Q-t);
	
	\coordinate [below = 1cm] (Pls) at (rba0);
	\coordinate[label=below:{$\calP_{k-1}$}] () at (Pls);
	\draw[[-{Latex[length=3mm]}, thick] (Pls) -- (rba0);
	\draw[[-{Latex[length=3mm]}, thick] (rba0) -- node[below right = -3pt and -3pt]{$\calQ_{k}$} (rbb0) -- (rba-1) -- (rbb-1);
	\coordinate [above = 1.5cm] (Plt) at (rbb-1);
	\coordinate[label=above:{$\calP_k$}] () at (Plt);
	\draw[[-{Latex[length=3mm]}, thick] (rbb-1) -- node[left]{(q1-1)} node[right]{(q1-2)} (Plt);
	
	\coordinate[label=right:{$\alpha^- = \alpha^*$}] () at (ba-2);
	
	\end{tikzpicture}
	\caption{
		Modification in Case~1 of \cref{subsubsec:path};
		the definitions of all lines and paths are the same as in \cref{fig:cycle}.
	}
	\label{fig:path case 1}
\end{figure}

Consider outer space-walks $\calP_\ell$ $(\ell \leq m)$ having the following property.
\begin{itemize}
	\item[(p1)] $P_\ell$ meets $R \cap P_m$.  
\end{itemize}
Note that $\calP_m$ always satisfies (p1).

We next classify previous inner space-walks $Q_k$ $(k \leq m-1)$
meeting $R$ (when the elimination enters $Q^-$).
In this case, it holds $\alpha^* = \alpha^-$.
\begin{itemize}
	\item[(q1-1)] $\calQ_k$ meets $R \cap Q^-$
	and
	$\calP_{k-1} \triangleright Q^-[\alpha(\calQ_k), \alpha^-]$ is an outer space-walk for $I'$
	such that its last space is different from $\ker_{I'}(\alpha^-) \neq \{0\}$.
	\item[(q1-2)] $\calQ_k$ meets $R \cap Q^-$ such that it is not the case of (q1-1). 
\end{itemize}

Choose $\calP_\ell$ with minimum $\ell$ satisfying (p1).
Also choose $Q_k$ with minimum $k$ satisfying (q1-1).
(If such an index does not exist, we let it to be $+ \infty$ below.)
Suppose $\ell < k$.
If $\ell < m$,
then
there is $\alpha$ in $P_\ell$
such that $\alpha$ is incident to an edge in $D$
and $\calP_\ell(\alpha]$ does not meet $D$ except $\alpha$.
Indeed,
otherwise $\calQ_\ell$ satisfies (q1-1),
a contradiction to $\ell < k$.
Define
\begin{align}\label{eq:T' path case 1}
\calT' :=
\begin{cases}
\calP_0 \circ \calQ_1 \circ \cdots \circ \calQ_\ell \circ \left(\calP_\ell(\alpha] \triangleright R[\alpha, \alpha^*]\right) & \text{if $\ell < m$},\\
\calP_0 \circ \calQ_1 \circ \cdots \circ \calQ_{m-1} \circ \left(\calP_{m-1} \triangleright \left(Q^+[\alpha(\calQ_m), \alpha^+] \circ R\right)\right) & \text{if $\ell = m$},
\end{cases}
\end{align}
where $R[\alpha, \alpha^*]$ and $Q^+[\alpha(\calQ_m), \alpha^+] \circ R$ form truncated outer-walk
since all edges in $I' \cap (Q^+[\alpha(\calQ_m), \alpha^+] \circ R)$ are isolated rank-2 edges.
If $k \leq \ell$, then
define
\begin{align}\label{eq:T' path case 2}
\calT' :=
\calP_0 \circ \calQ_1 \circ \cdots \circ \calQ_{k-1} \circ (\calP_{k-1} \triangleright Q^-[\alpha(\calP_{k-1}), \alpha^-]),
\end{align}
where $Q^-[\alpha(\calP_k), \alpha^-]$ forms a truncated outer-walk.
\begin{lem}\label{lem:T' path case12}
	If $\ell < k$,
	then $\calT'$ of the form~\eqref{eq:T' path case 1} is an augmenting space-walk for $I'$.
	If $k \geq \ell$,
	then $\calT'$ of the form~\eqref{eq:T' path case 2} is an augmenting space-walk for $I'$.
\end{lem}
\begin{proof}
	($\Ai$) is obvious.
	If the elimination from $\alpha(\calP_m)$ stops within $P_m$,
	then $\ker_{I'}(\alpha^*) (\neq \{0\})$ is the space of $\calP_m$ at $\alpha^*$.
	Otherwise we have $\alpha^*= \alpha^-$ and hence $\ker_{I'}(\alpha^*) \supseteq U_{\alpha^-}^+$.

	We first consider the case of $\ell < k$, particularly, $\ell < \min(m, k)$;
	the case of $\ell = m < k$ is similar.
	By ($\No$),
	the space $X$ of $\calP_\ell$ at $\alpha$ is equal to the propagated space of $P_m \triangleleft \ker_I(\alpha(\calP_m))$ at $\alpha$.
	If the elimination from $\alpha(\calP_m)$ stops within $P_m$,
	then
	the last space $X^*$ of $X \triangleright R[\alpha, \alpha^*]$ is
	the propagated space of $P_m \triangleleft \ker_I(\alpha(\calP_m))$ at $\alpha^*$.
	By \cref{lem:diff}~(1),
	$\calP_\ell(\alpha] \triangleright R[\alpha, \alpha^*]$ forms an outer space-walk for $I'$.
	Also we have $X^* \neq \ker_{I'}(\alpha^*)$.
	Hence $\calT'$ of the form~\eqref{eq:T' path case 1} satisfies ($\Al$).
	If the elimination enters $Q^-$,
	then
	the propagated space of $X \triangleright R[\alpha, \alpha^*]$ at $\beta^*$
	is $Y^\vee$.
	By $Y^\vee \neq V_{\beta^*}$,
	the last space of $X \triangleright R[\alpha, \alpha^*]$
	is different from $U_{\alpha^-}^+$.
	Thus $\calT'$ of the form~\eqref{eq:T' path case 1} satisfies ($\Al$).

	For the case of $k \leq \ell$,
	the assumption of (q1-1) immediately implies
	that $\calT'$ of the form~\eqref{eq:T' path case 1} satisfies ($\Al$).

	Finally we see previous inner space-walks $Q_{k'}$ satisfying (q1-2),
	where $k' \leq \ell$ if $\ell < k$, and $k' < k$ if $k \leq \ell$.
	If the last space $X(\calP_{k'-1})$ is different from $U_{\alpha(\calP_{k'-1})}^+$,
	then $\calQ_{k'}$ satisfies (q1-1),
	which contradicts the minimality of $k$.
	Hence we have $X(\calP_{k'-1}) = U_{\alpha(\calP_{k'-1})}^+$.
	By the same argument as (4) in \cref{lem:T' cycle case 1},
	if $Q_{k'}$ is included in $Q^-$,
	then
	$\calP_{k'-1} \circ \calQ_{k'} \circ \calP_{k'}$ is a single outer space-walk for $I'$.
	If $Q_{k'}$ passes through the initial vertex $\alpha^-$ of $Q^-$,
	then
	$\calP_{k'-1} \circ \calQ_{k'}(\alpha^-]$
	is a single outer space-walk compatible with the inner space-walk $\calQ_{k'}[\alpha^-)$.
	Hence $\calT'$ is a compatibly-concatenated space-walk.
	
	Summarizing, $\calT'$ is an augmenting space-walk.
\end{proof}

Let $(I, \calT) \leftarrow (I', \calT')$.
This update can be done in $O(|E|)$ time,
since
we can find $\calP_{\ell}$ satisfying (p1) or $\calQ_{k}$ satisfying (q1-1),
and compute several front-propagation in $O(|E|)$ time.
By the same argument as in \cref{lem:Delta},
the quantity of $\theta$ decrease by at least $|Q_m|$.
Recall that we execute the transformations in Case~2 of \cref{subsubsec:path} after the update~\eqref{eq:T' Ninner}.
Hence
the difference of $\theta$ between the one before~\eqref{eq:T' Ninner} and the one after the transformations in Case~2 of \cref{subsubsec:path}
is bounded by $\Delta - |Q_m|$;
$\theta$ strictly decreases by \cref{lem:Delta}.
Return to the initial stage (\cref{subsec:initial stage}).

\paragraph{Case~2: Only (B) holds.}
Let $\calT'$ be defined as~\eqref{eq:T' cycle case 1}.
By $\deg_{I'}(\beta^*) = 1$,
the elimination operation 
is applied from $\beta^*$ 
if the edge incident to $\beta^*$ in $I'$ is rank-2.
If such elimination does not occur, 
then the procedure is precisely the same as Case~1 of \cref{subsubsec:path}.
Suppose that such elimination occurs.
There are two cases:
the elimination from $\beta^*$ stops within $P_m$
and the elimination goes through $P_m$ to 
enter the rank-1 component $C'$ (for $I$) 
containing $\alpha(\calP_m)$.
The latter occurs precisely when all edges of $P_m$ are rank-2 and $\alpha(\calP_m)$ 
has degree one in $I$.
We will deal with the both cases simultanously.

Let $D$ denote the set of removed edges by the elimination from $\beta^*$.
Our main analysis concerns the situations 
where previous inner and outer space-walks meet edges in $D$.
Consider the path $P_m \cup C'$,
where
$C' = \emptyset$ if $\deg_I(\alpha(\calP_m)) = 0$.
Let $\beta^0$ be the farthest vertex from $\beta^*$ in $P_m \cup C'$
such that it is incident to a deleted edge in $D$.
Let $L$ be the subpath in $P_m \cup C'$ from $\beta^0$ to $\beta^*$.
Note here that its direction is the reverse of $P_m$,
and $\ker_I(\beta^0) \neq \{0\}$.
We may assume that all edges in $D \cap L$ are $+$-edges.

Consider previous outer space-walks $\calP_\ell$ $(\ell \leq m-1)$ having the following property. 
\begin{itemize}
	\item[(p2)] $P_\ell$ meets $L \cap P_m$ and does not end at $\alpha(\calP_m)$, 
	i.e., there is a vertex $\beta$ in $P_m$ such that $P_\ell[\beta)$ meets $P_m$ only at $\beta$.  
\end{itemize}
See Figure~\ref{fig:path case 2}.
\begin{figure}
	\centering
	\begin{tikzpicture}[
	node/.style={
		fill=black, circle, minimum height=5pt, inner sep=0pt,
	},
	I1/.style={
		line width = 3pt,
		decorate,
		decoration={snake, amplitude=.4mm,segment length=2.5mm,post length=0mm}
	},
	I2/.style={
		line width = 3pt
	},
	ID1/.style={
		line width = 3pt,
		blue,
		decorate,
		decoration={snake, amplitude=.4mm,segment length=2.5mm,post length=0mm}
	},
	ID2/.style={
		line width = 3pt,
		blue
	},
	NI/.style={
	},
	NIA2/.style={
		red,
		semithick
	},
	NIA1/.style={
		red,
		semithick,
		decorate,
		decoration={snake, amplitude=.3mm,segment length=1.5mm,post length=0mm}
	}
	]
	
	\def\mu{a1, a2, a3, a4, a5, a6, a7, a8, a9}
	\def\nu{b1, b2, b3, b4, b5, b6, b7, b8, b9}
	\def\bmu{ba1, ba2}
	\def\bnu{bb1, bb2}
	\def\size{1.7cm}
	\def\hight{5cm}
	\def\side{5pt}

	\nodecounter{\nu}
	\coordinate (pos);
	\foreach \currentnode in \nu {
		\node[node, below=0 of pos, anchor=center] (\currentnode) {};
		\coordinate (pos) at ($(pos)+(-\size, 0)$);
	}
	\nodecounter{\mu}
	\coordinate (pos) at ($(b1) + (-\size, -\size)$);
	\foreach \currentnode in \mu {
		\node[node, below=0 of pos, anchor=center] (\currentnode) {};
		\coordinate (pos) at ($(pos)+(-\size, 0)$);
	}

	\draw[I1] (b9) -- node[above left = 0-2pt and -5pt]{$\calQ_k$} (a9);
	\draw[ID2] (b9) -- node[left = -2pt]{$+$} (a8);
	\draw[ID1] (b4) -- (a3);
	\draw[ID1] (b8) -- (a7);
	\draw[ID2] (b7) -- (a6);
	
	\foreach \i / \j in {b5/a5, b6/a6, b7/a7, b8/a8} {
		\draw[I2] (\i) -- (\j);
	}
	\draw[I2] (b4) -- node[above left = -3pt and -3pt]{$-$} (a4);
	
	\foreach \i / \j in {b1/a1, b2/a2, b3/a3} {
		\draw[NIA2] (\i) -- (\j);
	}
	
	\foreach \i / \j in {b2/a1, b3/a2, b5/a4} {
		\draw[ID2] (\i) -- (\j);
	}
	
	\draw[ID2] (b6) -- node[right = -2pt]{$+$} (a5);

	\coordinate [label=right:{$\beta^*$}] () at (b1);
	
	\coordinate [label=below:{$\alpha(\calP_m)$}] () at (a3);
	\coordinate [label=above:{$\beta^0$}] () at (b9);
	\coordinate [label=above right:{$\beta$}] () at ($(b3) + (3pt, 0)$);
	\coordinate [label=above right:{$\beta$}] () at ($(b4) + (3pt, 0)$);
	\coordinate [label=above right:{$\beta$}] () at ($(b7) + (3pt, 0)$);
	\coordinate [label=above right:{$\beta$}] () at ($(b8) + (3pt, 0)$);
	
	\foreach \i in {a1, a2, a3, a4, a5, a6, a7, a8, a9, b1, b2, b3, b4, b5, b6, b7, b8, b9} {
		\coordinate [left = \side] (l\i) at (\i);
		\coordinate [right = \side] (r\i) at (\i);
	}
	
	\draw[[-{Latex[length=3mm]}, dashed, thick] ($(rb1)!0.1!(ra1)$) -- node[below right = -5pt and -2pt]{$\calP_m$} ($(rb1)!0.6!(ra1)$);
	\draw[[-{Latex[length=3mm]}, dashed, thick] ($(rb3)!0.5!(ra3)$) -- node[below right = 0 and -4pt]{$\calP_m$} (ra3);
	
	\coordinate [label=below:{$\calP_\ell$}] (Pms) at ($(ra1) + (0, -1.5cm)$);
	\coordinate [above = 1.5cm] (Pmt) at (rb3);
	\draw[[-{Latex[length=3mm]}, dashed, thick] (Pms) -- (ra1) -- (rb2) -- (ra2) -- (rb3) -- node[right]{(p2)} (Pmt);

	\coordinate [above = 1.5cm] (p0P0t) at (lb3);
	\draw[[-{Latex[length=3mm]}, thick] (lb9) -- (la8) -- (lb8) -- (la7) -- (lb7) -- (la6) -- (lb6) -- (la5) -- (lb5) -- (la4) -- (lb4) -- (la3) -- (lb3) -- (p0P0t);

	\coordinate [above = 1.5cm] (q0Qkt) at (rb8);
	\coordinate [label=above:$\calP_k$] () at (q0Qkt);
	\draw[[-{Latex[length=3mm]}, thick] (ra9) -- (rb9) -- (ra8) -- (rb8);
	\draw[[-{Latex[length=3mm]}, thick] (rb8) -- node[left]{(q2-0)} (q0Qkt);

	\coordinate [label=below:$\calP_{k-1}$] (q1Pk-1s) at ($(ra5) + (0, -1.5cm)$);
	\draw[[-{Latex[length=3mm]}, dashed, thick] (q1Pk-1s) -- (ra5);
	\draw[[-{Latex[length=3mm]}, dashed, thick] (ra5) -- (rb6) -- node[below = 2pt]{$\calQ_k$} (ra6) -- (rb7);
	\coordinate [label=above:$\calP_k$] (q1Pkt) at ($(rb7) + (0, 1.5cm)$);
	\draw[[-{Latex[length=3mm]}, dashed, thick] (rb7) -- node[right]{(q2-1)} (q1Pkt);
	\coordinate [above = 1.5cm] (q1Pkt) at (lb7);
	\draw[[-{Latex[length=3mm]}, thick] (lb7) -- (q1Pkt);

	\coordinate [label=below:$\calP_{k-1}$] (q23Pk-1s) at ($(ra4) + (0, -1.5cm)$);
	\draw[[-{Latex[length=3mm]}, thick] (q23Pk-1s) -- (ra4);
	\draw[[-{Latex[length=3mm]}, thick] (ra4) -- node[below = 2pt]{$\calQ_k$} (rb4);
	\coordinate [label=above:$\calP_k$] (q23Pkt) at ($(rb4) + (0, 1.5cm)$);
	\draw[[-{Latex[length=3mm]}, thick] (rb4) -- node[right]{(q2-2)} (q23Pkt);
	\coordinate [above = 1.5cm] (q23Pkt) at (lb4);
	\draw[[-{Latex[length=3mm]}, thick] (lb4) -- node[left]{(q2-3)} (q23Pkt);

	\coordinate [right = 8pt] (Rs) at (b9);
	\coordinate [below = 1.5cm] (Rst) at (Rs);
	\draw[[-{Latex[length=3mm]}, dotted, thick] (Rs) -- node[right = -3pt]{$L$} ($(Rs)!0.7!(Rst)$);
	
	\coordinate [left = 5pt] (Rt) at (b1);
	\coordinate [left = 5pt] (Rts) at (a1);
	\draw[[-{Latex[length=3mm]}, dotted, thick] ($(Rts)!0.4!(Rt)$) -- node[above = 0pt]{$L$} (Rt);

	\end{tikzpicture}
	\caption{
		Modification in Case~2 of \cref{subsubsec:path};
		the definitions of all lines and paths are the same as in Figure~\ref{fig:cycle}.
	}
	\label{fig:path case 2}
\end{figure}
In the case where the elimination stops within $P_m$, 
the outer space-walk $P_\ell$ meeting $D$ always satisfies (p).  

We next classify previous inner space-walks $\calQ_k$ $(k \leq m-1)$
meeting $L$ (when the elimination enters the rank-1 component $C'$ of $I$) as follows:
\begin{itemize}
	\item[(q2-0)] $\calQ_k$ passes through $\beta^0$ to enter $L$ 
	and leaves $L$ at $\beta$ with label $V_\beta^+$.
	\item[(q2-1)] $\calQ_k$ starts with $+$-edge in $L$ and leaves $L$ at $\beta$ with label $V_\beta^-$.
	\item[(q2-2)] $\calQ_k$ starts with $-$-edge in $L$ and leaves $L$ at $\beta$ with label $V_\beta^+$
	such that $\calP_{k-1} \triangleright (Q_k \circ P_k)$ is an outer space-walk for $I'$
	and compatibly concatenated with $\calQ_{k+1}$. 
	\item[(q2-3)] $\calQ_k$ starts with $-$-edge in $L$ and leaves $L$ at $\beta$ with label $V_\beta^+$
	such that it is not the case of (q2-2). 
\end{itemize}
Note that if $P_\ell$ ends at $\alpha(\calP_m)$, 
then $\ell < m-1$ and the next $Q_{\ell+1}$ is in the case of (q2-1).

Choose $P_\ell$ with maximum $\ell$ satisfying (p2).
Also choose $Q_k$ with maximum $k$ satisfying (q2-1) or (q2-3).
(If such an index does not exist, we let it to be $- \infty$ below.)
If $\ell \geq k$ and $\ell > -\infty$, then
consider the vertex $\beta$ in (p2), and
replace the prefix $\calP_0 \circ \calQ_1 \circ \cdots \circ \calQ_\ell \circ \calP_\ell$ of $\calT'$
by 
\begin{align}\label{eq:after1}
\left(\ker_{I'} (\beta^0) \triangleright L(\beta]\right) \circ \calP_\ell[\beta).
\end{align}
If $\ell < k$, then consider the vertex $\beta$ in (q2-1) or (q2-3), 
replace the prefix $\calP_0 \circ \calQ_1 \circ \cdots \circ \calQ_k \circ \calP_k$ of $\calT'$
by 
\begin{align}\label{eq:after2}
\ker_{I'} (\beta^0) \triangleright (L(\beta] \circ P_k).
\end{align}

For each inner space-walk $\calQ_{k'}$ satisfying (q2-2) with $k' \geq \max(\ell, k)$, 
replace the subsequence $\calP_{k'-1} \circ \calQ_{k'} \circ \calP_{k'}$
by a single outer space-walk
$\calP_{k'-1} \triangleright (Q_{k'} \circ P_{k'})$.
\begin{lem}\label{lem:T' path case 2}
	${\cal T}'$ is an augmenting space-walk for $I'$.
\end{lem}
\begin{proof}
	Assume that the elimination operation is applied from $\beta^*$.
	By the same argument as in \cref{lem:T' cycle case 1},
	($\Al$) holds.
	Consider ($\Ai$).
	If we do not replace the prefix $\calP_0 \circ \calQ_1 \circ \cdots \circ \calQ_\ell \circ \calP_\ell$ of $\calT'$
	in the update, i.e., $\ell = k = -\infty$,
	obviously ($\Ai$) holds.
	Otherwise we replace it
	by either~\eqref{eq:after1} or~\eqref{eq:after2}.

	Suppose the former case ($\ell \geq k$ and $\ell > -\infty$).
	By ($\No$),
	the subspace of $\calP_{\ell}$ at $\beta$ is equal to the propagated space of $P_m \triangleleft \ker_I(\alpha(\calP_m))$ at $\beta$.
	Furthermore one can see that
	the last space of $\ker_{I'}(\beta^0) \triangleright L(\beta]$ coincides with the propagated space of $P_m \triangleleft \ker_I(\alpha(\calP_m))$ at $\beta$ as follows.
	If $L$ is included in $P_m$,
	$\ker_{I'}(\beta^0)$ is the propagated space of $P_m \triangleleft \ker_I((\calP_m))$ at $\beta^0$,
	which implies the desired identity.
	Suppose that $L$ properly includes $P_m$.
	Since $L$ is a subpath in a rank-1 connected component,
	$\ker_{I'}(\beta^0) = V_{\beta^0}^-$ if $\deg_{I'}(\beta^0) = 1$,
	and $\ker_{I'}(\beta^0) = V_{\beta^0}$ if $\deg_{I'}(\beta^0) = 0$.
	The propagated space of $\ker_{I'}(\beta^0) \triangleright L$ at $\gamma \in L$
	is equal to $V_{\beta'}^-$ if $\gamma = \beta'$ (except for $\gamma = \beta^0$ in case of $\deg_{I'}(\beta^0) = 0$)
	and to $U_\alpha^+$ if $\gamma = \alpha'$.
	Hence the propagated space of $\ker_{I'}(\beta^0) \triangleright L$ at $\alpha(\calP_m)$ is equal to $U_{\alpha(\calP_m)}^+$,
	implying the desired identity.
	Thus $\left(\ker_{I'}(\beta^0) \triangleright L(\beta']\right) \circ P_{\ell}[\beta)$ forms an outer space-walk for $I'$,
	implying ($\Ai$).
	In addition, $\left(\ker_{I'}(\beta^0) \triangleright L(\beta']\right) \circ P_{\ell}[\beta)$
	is clearly compatible with $\calQ_\ell$.
	
	Suppose the latter case ($\ell < k$).
	If $\calQ_k$ satisfies (q2-1),
	then the propagated space of $\ker_{I'}(\beta^0) \triangleright L$ at $\beta$
	is $V_{\beta'}^-$ by the above argument,
	which coincides with the subspace of $\calQ_k$ at $\beta$.
	Hence $\ker_{I'} (\beta^0) \triangleright (L(\beta] \circ P_k)$ forms a single outer space-walk for $I'$.
	In addition, since the last space of $\ker_{I'} (\beta^0) \triangleright (L(\beta] \circ P_k)$
	coincides with that of $\calP_k$,
	$\ker_{I'} (\beta^0) \triangleright (L(\beta] \circ P_k)$ is compatible with $\calQ_{k+1}$.
	If $\calQ_k$ satisfies (q2-3),
	then we have $X(\calP_{k-1}) \neq U_{\alpha(\calP_{k-1})}^-$.
	Indeed, otherwise the last space of $\calP_{k-1}$ coincides with $Y(\calQ_k)$,
	which contradicts the compatibility of $(\calQ_k, \calP_k)$.
	Thus, by \cref{lem:diff},
	$\ker_{I'} (\beta^0) \triangleright (L(\beta] \circ P_k) = \left(\ker_{I'} (\beta^0) \triangleright L(\beta]\right) \circ (V_\beta^- \triangleright P_k)$
	forms a single outer space-walk for $I'$.
	In addition,
	the last space of $\ker_{I'} (\beta^0) \triangleright (L(\beta] \circ P_k)$
	is different from that of $\calP_{k-1} \triangleright (Q_k \circ P_k)$.
	implying that $\ker_{I'} (\beta^0) \triangleright (L(\beta] \circ P_k)$ is compatible with $\calQ_{k+1}$.
	
	We next show that $\calT'$ is a compatibly-concatenated space-walk for $I'$.
	We have already seen the compatibility between the first outer space-walk and the first inner space-walk in $\calT'$ above.
	We
	have already dealt in (3) and (4) in \cref{lem:T' cycle case 1} with the cases where some $Q_{k'}$ with $\max(\ell, k) < k' < m$ meets deleted edges in $Q^+$
	and where some $\calP_{k'}$ with $\max(\ell, k) < k' < m$ meets undeleted edges in $P_m$,
	respectively:
	$\calP_{k'-1} \circ \calQ_{k'} \circ \calP_{k'}$
	forms a concatenation of outer and inner space-walks for $I'$ in the former case,
	and
	$\calP_{k'} \circ \calQ_{k'+1}$
	forms a concatenation of outer and inner space-walks for $I'$ in the latter case.
	
	Thus
	it suffices to see previous inner space-walks $\calQ_{k'}$ with $\max(\ell, k) < k' < m$
	satisfying (q2-0) or (q2-2).
	Suppose that $\calQ_{k'}$ satisfies (q2-0).
	Then, by the same argument as (3) in \cref{lem:T' cycle case 1},
	$\calQ_{k'}(\beta^0]$ remains an inner space-walk
	and $\calQ_{k'}[\beta^0) \circ \calP_{k'}$ forms an outer space-walk for $I'$.
	The compatibility of $\calQ_{k'}(\beta^0]$ and $\calQ_{k'}[\beta^0) \circ \calP_{k'}$
	clearly holds.
	Suppose that $\calQ_{k'}$ satisfies (q2-2).
	Obviously $\calP_{k'-1} \triangleright (Q_{k'} \circ P_{k'})$ forms an outer space-walk for $I'$.
	The assumption says that
	$\calP_{k'-1} \triangleright (Q_{k'} \circ P_{k'})$ is compatible with $\calQ_{k'+1}$.
\end{proof}

Let $(I, \calT) \leftarrow (I', \calT')$.
By the same argument as in Case~1,
this update can be done in $O(|E|)$ time
and
$\theta$ strictly decreases.
Return to the initial stage (\cref{subsec:initial stage}).

\paragraph{Case~3: Both (A) and (B) hold.}
If $P_m$ has a rank-1 edge,
then we can simply combine the arguments in Cases~1 and~2 above.
That is,
we first define $\calT'$ as~\eqref{eq:T' cycle case 1} if the last edge of $\calP_m$ is rank-1,
and as~\eqref{eq:T' path case 1} if the last edge of $\calP_m$ is rank-2.
Note here that, by $Q^- = \emptyset$,
there is no inner space-walks in $\calT'$ satisfying (q1-1) or (q1-2) in Case~1.
Let $\calR$ be the maximum prefix of $\calT'$
that coincides with a prefix of $\calT$.
Namely,
if the last edge of $\calP_m$ is rank-1, or the last edge of $\calP_m$ is rank-2 and $\ell = m$ in~\eqref{eq:T' path case 1},
then
$\calR := \calP_0 \circ \calQ_1 \circ \cdots \circ \calQ_{m-1}\circ \calP_{m-1}$,
and if the last edge of $\calP_m$ is rank-2 and $\ell < m$ in~\eqref{eq:T' path case 1},
then $\calR := \calP_0 \circ \calQ_1 \circ \cdots \circ \calQ_\ell \circ \calP_{\ell}(\alpha]$,
where the definitions of $\ell$ and $\alpha$ are given in Case~1.
Let $\ell$ denote the last index of $\calR$ again.

In addition,
suppose that there is $\calP_{\ell'}$ with $\ell' \leq \ell$
such that it satisfies (p2) with $\beta$ belonging to $\calR$ in Case~2.
Then, for such maximum $\ell'$, we replace the prefix $\calP_0 \circ \calQ_1 \circ \cdots \circ \calP_{\ell'}$ of $\calT'$
by $\left(\ker_{I'} (\beta^0) \triangleright L(\beta]\right) \circ \calP_{\ell'}[\beta)$ as~\eqref{eq:after1}.
Here the definitions of $\beta$ and $\beta^0$ are given in Case~2.
Note that, since $P_m$ has a rank-1 edge,
there is no inner space-walk in $\calT'$ satisfying (q2-0), (q2-1), (q2-2), or (q2-3) in Case~2.
By the same arguments as in the proofs of \cref{lem:T' path case 2,lem:T' path case12},
the resulting $\calT'$ is an augmenting space-walk for $I'$.

Let $(I, \calT) \leftarrow (I', \calT')$.
By the same arguments as in Cases~1 and~2,
this update can be done in $O(|E|)$ time
and
$\theta$ strictly decreases.
Return to the initial stage (\cref{subsec:initial stage}).

Suppose that $P_m$ has no rank-1 edge.
Then one can see $r(I') > r(I)$.
Thus the resulting set is a desired augmentation.
The augmentation procedure terminates.
This update can be clearly done in $O(|E|)$ time.

\subsection{$\calP_m$ is not simple}\label{subsec:not simple}
Suppose that the last outer space-walk $\calP_m$ of $\calT$
is not simple.
Let $\beta_0$ be the vertex such that
$\calP_m[\beta_0)$ contains $\beta_0$ twice
and the other vertices once;
such $\beta_0$ exists
by the assumption.
Then $\calP_m$ is of the form
\begin{align*}
\calP_m = (\dots, X_0, \alpha_0 \beta_0, Y_0, \beta_0 \alpha_1, X_1, \alpha_1 \beta_1, \dots, \alpha_k \beta_k = \alpha_0 \beta_0, Y_k, \beta_k \alpha_{k+1}, \dots),
\end{align*}
where $\beta_0 \alpha_1, \alpha_1 \beta_1, \dots$ are distinct.
Let us denote by $\calP_m^\vee$
the back-propagation $P_m \triangleleft \ker_I(\alpha(\calP_m))$,
which has been already computed in the initial stage (\cref{subsec:initial stage}).
For $i = 0,1,2,\dots,k$,
let $X_i^\vee$ and $Y_i^\vee$ be the subspaces of $\calP_m^\vee$ at $\alpha_i$ and at $\beta_i$,
respectively.
By ($\No$),
we have $X_0 = X_k^\vee$.
Consider the two cases:
all $\beta_0 \alpha_1, \beta_1 \alpha_2, \dots, \beta_{k-1} \alpha_k$ are rank-2,
and at least one of $\beta_0 \alpha_1, \beta_1 \alpha_2, \dots, \beta_{k-1} \alpha_k$ is rank-1.

\subsubsection{All $\beta_0 \alpha_1, \beta_1 \alpha_2, \dots, \beta_{k-1} \alpha_k$ are rank-2}\label{subsubsec:not simple case 1}
Define
\begin{align*}
I' := \left(I \cup \{\beta_0 \alpha_1, \beta_1 \alpha_2, \dots, \beta_{k-1} \alpha_k\}\right) \setminus \{ \alpha_k \beta_k = \alpha_0 \beta_0, \alpha_1 \beta_1, \dots, \alpha_{k-1} \beta_{k-1} \}.
\end{align*}
It is clear that $I'$ is a q-matching with $r(I') = r(I)$.

We then modify $\calT$ to obtain an augmenting space-walk $\calT'$ for $I'$
as follows.
Let $\calP_{\ell}$
be the first outer space-walk in $\calT$ that intersects with $\alpha_1, \alpha_2, \dots, \alpha_k = \alpha_0$;
note that such $\ell$ exists since $\calP_m$ contains them.
Let $\alpha$ be the first appearance of such a vertex in $\calP_{\ell}$,
and suppose $\alpha = \alpha_s$.
By ($\No$),
the subspace at $\alpha$ in $\calP_{\ell}$
coincides with $X_s^\vee$.
Then define
\begin{align}\label{eq:T' not simple case 1}
\calT' := \calP_0 \circ \calQ_1 \circ \cdots \circ \calQ_{\ell} \circ \calP_{\ell}(\alpha] \circ \calP_m^\vee[\alpha_s, \beta_0] \circ \left(Y_0^\vee \triangleright P_m[\beta_k)\right).
\end{align}
See Figure~\ref{fig:non simple 1}.
\begin{figure}
	\centering
	\begin{tikzpicture}[
	node/.style={
		fill=black, circle, minimum height=5pt, inner sep=0pt,
	},
	I/.style={
		line width = 3pt
	},
	ID/.style={
		line width = 3pt, blue
	},
	NI/.style={
		semithick
	},
	NIA2/.style={
		red,
		semithick
	},
	NIA1/.style={
		red,
		semithick,
		decorate,
		decoration={snake, amplitude=.3mm,segment length=1.5mm,post length=0mm}
	}
	]
	
	\def\mu{a1, a2, a3, a4, a5, a6, a7, a8, a9}
	\def\nu{b1, b2, b3, b4, b5, b6, b7, b8, b9}
	\def\size{1.7cm}
	\def\side{5pt}
	
	\nodecounter{\nu}
	\coordinate (pos);
	\foreach \currentnode in \nu {
		\node[node, below=0 of pos, anchor=center] (\currentnode) {};
		\coordinate (pos) at ($(pos)+(-\size, 0)$);
	}
	\nodecounter{\mu}
	\coordinate (pos) at ($(b1) + (-\size, -\size)$);
	\foreach \currentnode in \mu {
		\node[node, below=0 of pos, anchor=center] (\currentnode) {};
		\coordinate (pos) at ($(pos)+(-\size, 0)$);
	}
	
	\foreach \i / \j in {b1/a1, b2/a2, b8/a8, b9/a9}{
		\draw[NI] (\i) -- (\j);
	}
	
	\foreach \i / \j in {b3/a3, b4/a4, b5/a5, b6/a6, b7/a7}{
		\draw[NIA2] (\i) -- (\j);
	}
	
	\foreach \i / \j in {b2/a1, b9/a8}{
		\draw[I] (\i) -- (\j);
	}
	
	\foreach \i / \j in {b3/a2, b4/a3, b5/a4, b6/a5, b7/a6, b8/a7}{
		\draw[ID] (\i) -- (\j);
	}
	
	\coordinate [above right = 5pt and 5pt] (arb3) at (b3);
	\coordinate [above left = 5pt and 5pt] (alb3) at (b3);
	\coordinate [below right = 5pt and 5pt] (bra2) at (a2);
	\coordinate [below left = 5pt and 5pt] (bla2) at (a2);
	\draw[dotted, very thick] (arb3) -- (alb3) -- (bla2) -- (bra2) -- (arb3);
	
	\coordinate [above right = 5pt and 5pt] (arb8) at (b8);
	\coordinate [above left = 5pt and 5pt] (alb8) at (b8);
	\coordinate [below right = 5pt and 5pt] (bra7) at (a7);
	\coordinate [below left = 5pt and 5pt] (bla7) at (a7);
	\draw[dotted, very thick] (arb8) -- (alb8) -- (bla7) -- (bra7) -- (arb8);
	
	\coordinate [label=above:{$\beta_0$}] () at ($(arb3)!0.5!(alb3)$);
	\coordinate [label=below:{$\alpha_0$}] () at ($(bra2)!0.5!(bla2)$);
	\coordinate [label=above:{$\beta_k$}] () at ($(arb8)!0.5!(alb8)$);
	\coordinate [label=below:{$\alpha_k$}] () at ($(bra7)!0.5!(bla7)$);
	
	\coordinate [label=left:{$\alpha = \alpha_s$}] () at ($(a5) + (-2pt, 0)$);
	
	\foreach \i in {a1, a2, a3, a4, a5, a6, a7, a8, a9, b1, b2, b3, b4, b5, b6, b7, b8, b9} {
		\coordinate [left = \side] (l\i) at (\i);
		\coordinate [right = \side] (r\i) at (\i);
	}
	
	\coordinate [below = 1.5cm] (oPms) at (la5);

	\coordinate [above = 1.5cm] (oPmt) at (lb7);
	\draw[[-{Latex[length=3mm]}, dashed, thick] (oPms) -- node[left]{$\calP_{\ell}$} (la5) -- (lb6) -- (la6) -- (lb7) -- (oPmt);
	
	\coordinate [below = 1.5cm] (nPms) at (ra5);
	\draw[[-{Latex[length=3mm]}, thick] (nPms) -- (ra5) -- (rb5) -- (ra4) -- (rb4) -- (ra3) -- ($(ra3)!0.9!(rb3)$);
	\draw[[-{Latex[length=3mm]}, thick] ($(rb8)!0.1!(ra8)$) -- (ra8) -- (rb9) -- (ra9);

	\draw[[-{Latex[length=3mm]}, dashed, thick] (lb1) -- node[above left = 0 and -4pt]{$\calP_m$} ($(lb1)!0.5!(la1)$);
	\coordinate [label=above:{$\beta(\calP_m)$}] () at (b1);
	\draw[[-{Latex[length=3mm]}, dashed, thick] ($(lb9)!0.4!(la9)$) -- node[above left = 0 and -4pt]{$\calP_m$} ($(lb9)!0.9!(la9)$);
	\coordinate [label=below:{$\alpha(\calP_m)$}] () at (a9);

	\end{tikzpicture}
	\caption{
		Modification in \cref{subsubsec:not simple case 1};
		the definitions of all lines and paths are the same as in Figure~\ref{fig:cycle}.
	}
	\label{fig:non simple 1}
\end{figure}

We show that $\calT'$ is an augmenting space-walk for $I'$.
($\Ai$) is obvious.
Since $X_k^\vee \neq X_k$
and all $\beta_0 \alpha_1, \beta_1 \alpha_2, \dots, \beta_{k-1} \alpha_k$ are rank-2,
we have $Y_0^\vee \neq Y_0$.
Also, by $X_0 = X_k^\vee$,
we have $Y_0 = Y_k^\vee$.
Thus the last space $Y_0^\vee$ of $\calP_m^\vee[\alpha_s, \beta_0]$
is different from $Y_k^\vee$.
By \cref{lem:diff}~(2),
the last space of $Y_0^\vee \triangleright P_m[\beta_k)$ does not include $\ker_I(\alpha(\calP_m))$.
This implies ($\Al$).
By the definitions of $\calP_{\ell}$ and $\alpha$,
the prefix $\calP_0 \circ \calQ_1 \circ \cdots \circ \calQ_{\ell} \circ \calP_{\ell}(\alpha]$ does not intersect with $\alpha_1, \alpha_2, \dots, \alpha_k$ except for $\alpha$.
Hence $\calP_0 \circ \calQ_1 \circ \cdots \circ \calQ_{\ell} \circ \calP_{\ell}(\alpha]$
remains a compatibly-concatenated space-walk for $I'$.
In addition,
since $\alpha_s \beta_{s-1}, \beta_{s-1} \alpha_{s-1}, \dots, \alpha_1 \beta_0$ are distinct,
$\calP_{\ell}(\alpha] \circ \calP_m^\vee[\alpha_s, \beta_0] \circ \left(Y_0^\vee \triangleright P_m[\beta_k)\right)$ forms a single outer space-walk for $I'$,
which is clearly compatible with $\calQ_{\ell}$.
These imply that $\calT'$ is an augmenting space-walk for $I'$.

Let $(I, \calT) \leftarrow (I', \calT')$.
This update can be done in $O(|E|)$ time,
since
we can find such $\calP_{\ell}$ in~\eqref{eq:T' not simple case 1}
and compute the front-propagation in $O(|E|)$ time.
Furthermore
$D_{\rm inner}$ does not increase
and $N$ strictly decreases;
the edge $\alpha_0 \beta_0 = \alpha_k \beta_k$ exits the extended support of the new $\calT$.
Hence $\theta$ strictly decreases.
Return to the initial stage (\cref{subsec:initial stage}).

\subsubsection{At least one of $\beta_0 \alpha_1, \beta_1 \alpha_2, \dots, \beta_{k-1} \alpha_k$ is rank-1}\label{subsubsec:not simple case 2}
Let $r$ be the maximum index with $1 \leq r \leq k$ such that $\beta_{r-1} \alpha_r$ is rank-1.
Then define
\begin{align*}
I' := \left(I \cup P_m[\beta_0)\right) \setminus \{ \alpha_r \beta_r, \alpha_{r+1} \beta_{r+1}, \dots, \alpha_k \beta_k \}.
\end{align*}
\begin{lem}\label{lem:I' 1 case 2}
	$I'$ satisfies {\rm (Deg)}, {\rm (q-Cycle)}, {\rm (VL)}, and $r(I') = r(I)$.
\end{lem}
\begin{proof}
	We can easily see that $I'$ satisfies (Deg), (q-Cycle), and $r(I') = r(I)$.
	We show that $I'$ satisfies (VL).
	We may assume that $\beta_k \alpha_{k+1}$ is a $-$-edge in $I'$,
	namely, $\beta_k \alpha_{k+1}, \beta_{k+1} \alpha_{k+2}, \dots$ are $-$-edges
	and $\beta_0 \alpha_1, \beta_1 \alpha_2, \dots, \beta_{r-1} \alpha_r$ are $+$-edges.
	For each $\alpha, \beta$ belonging to $P_m[\beta_k)$,
	define $U_\alpha^+$ and $V_\beta^-$ as in \cref{lem:I^*}.
	Note $V_{\beta_i}^- = \kerR(A_{\alpha_{i+1}\beta_i})$ if $\beta_i \alpha_{i+1}$ with $i \geq k$ is rank-1.
	For $\alpha, \beta$ belonging to $P_m[\alpha_r, \beta_0]$,
	define $U_\alpha^+$ and $V_\beta^-$ as the propagated spaces of $P_m[\alpha_r, \beta_0] \triangleleft V_{\beta_k}^-$
	at $\alpha$ and at $\beta$,
	respectively.
	Here $V_{\beta_k}^- = Y_k^\vee$ holds by ($\No$).
	Note $U_{\alpha_{i+1}}^+ = \kerL(A_{\alpha_{i+1}\beta_i})$ if $\beta_i \alpha_{i+1}$ with $0 \leq i \leq r-1$ is rank-1.
	
	Next define $U_{\alpha_r}^-$ as any 1-dimensional subspace of $U_{\alpha_r}$
	different from $U_{\alpha_r}^+ = \kerL(A_{\alpha_r \beta_{r-1}})$,
	where $\kerL(A_{\alpha_r \beta_{r-1}})$ is equal to the space of $\calP_m$ at $\alpha_r$.
	Then, for each $\alpha, \beta$ belonging to $P_m[\alpha_r, \beta_0] \circ P_m[\beta_k)$,
	define $U_\alpha^-$ and $V_\beta^+$ as the propated spaces of $U_{\alpha_r}^- \triangleright \left(P_m[\alpha_r, \beta_0] \circ P_m[\beta_k)\right)$
	at $\alpha$ and at $\beta$,
	respectively.
	Here $V_{\beta_0}^+ = Y_0^\vee$ holds,
	which is different from $Y_k^\vee = V_{\beta_k}^-$.
	Hence, by \cref{lem:diff},
	we have $U_\alpha^+ \neq U_\alpha^-$ and $V_\beta^+ \neq V_\beta^-$ for each $\alpha, \beta$.
	One can see that, if $\beta_i \alpha_{i+1}$ with $0 \leq i \leq r-1$ is rank-1 then $V_{\beta_i}^+ = \kerR(A_{\alpha_{i+1}\beta_i})$,
	and if $\beta_i \alpha_{i+1}$ with $i \geq k$ is rank-1 then $U_{\alpha_{i+1}}^- = \kerL(A_{\alpha_{i+1}\beta_i})$.
	This implies~\eqref{eq:++ --}.
	The orthogonal property~\eqref{eq:+-} is satisfied by the construction.
\end{proof}

If $\deg_I(\alpha(\calP_m)) = 1$ or the last edge of $\calP_m$ is rank-1,
then $I'$ satisfies (Path),
which implies that $I'$ is a q-mathcing.
Otherwise the end edge of the path component of $I'$
is rank-2.
We apply the elimination operation to $I'$ from $\alpha(\calP_m)$ so that (Path) holds.
The resulting set, also denoted by $I'$, is a q-mathcing with $r(I') \geq r(I)$ (in fact $r(I') = r(I)$) by \cref{lem:elimination}.

We next modify $\calT$ to obtain an augmenting space-walk $\calT'$ for $I'$
as follows.
Let $\calP_{\ell}$
be the first outer space-walk in $\calT$ that intersects with $\alpha_r, \alpha_{r+1}, \dots, \alpha_k$.
Let $\alpha$ be the first appearance of such a vertex in $\calP_{\ell}$,
and suppose $\alpha = \alpha_s$.
By~($\No$),
the subspace of $U_{\alpha}$ in $\calP_{\ell}$ at $\alpha$
coincides with $X_s^\vee$.
Then define
\begin{align}\label{eq:T' not simple case 2}
\calT' := \calP_0 \circ \calQ_1 \circ \cdots \calQ_{\ell} \circ \calP_{\ell}(\alpha] \circ \calP_m^\vee[\alpha_s, \alpha_r];
\end{align}
see Figure~\ref{fig:non simple 2}.
\begin{figure}
	\centering
	\begin{tikzpicture}[
	node/.style={
		fill=black, circle, minimum height=5pt, inner sep=0pt,
	},
	I/.style={
		line width = 3pt
	},
	ID/.style={
		line width = 3pt, blue
	},
	NI/.style={
		semithick
	},
	NIA2/.style={
		red,
		semithick
	},
	NIA1/.style={
		red,
		semithick,
		decorate,
		decoration={snake, amplitude=.3mm,segment length=1.5mm,post length=0mm}
	}
	]
	
	\def\mu{a1, a2, a3, a4, a5, a6, a7, a8, a9}
	\def\nu{b1, b2, b3, b4, b5, b6, b7, b8, b9}
	\def\size{1.7cm}
	\def\side{5pt}
	
	\nodecounter{\nu}
	\coordinate (pos);
	\foreach \currentnode in \nu {
		\node[node, below=0 of pos, anchor=center] (\currentnode) {};
		\coordinate (pos) at ($(pos)+(-\size, 0)$);
	}
	\nodecounter{\mu}
	\coordinate (pos) at ($(b1) + (-\size, -\size)$);
	\foreach \currentnode in \mu {
		\node[node, below=0 of pos, anchor=center] (\currentnode) {};
		\coordinate (pos) at ($(pos)+(-\size, 0)$);
	}
	
	\foreach \i / \j in {b1/a1, b2/a2}{
		\draw[NI] (\i) -- (\j);
	}
	
	\foreach \i / \j in {b5/a5, b6/a6, b7/a7}{
		\draw[NIA2] (\i) -- (\j);
	}

	\draw[NIA2] (b8) -- node[above left = 0 and -7pt]{$-$} (a8);
	\draw[NIA1] (b9) -- node[above left = 0 and -7pt]{$-$} (a9);
	
	\foreach \i / \j in {b3/a3, b4/a4}{
		\draw[NIA1] (\i) --node[above left = 0 and -7pt]{$+$} (\j);
	}
	
	\foreach \i / \j in {b2/a1, b4/a3, b9/a8}{
		\draw[I] (\i) -- (\j);
	}
	
	\foreach \i / \j in {b3/a2,b5/a4, b6/a5, b7/a6, b8/a7}{
		\draw[ID] (\i) -- (\j);
	}
	
	\coordinate [above right = 5pt and 5pt] (arb3) at (b3);
	\coordinate [above left = 5pt and 5pt] (alb3) at (b3);
	\coordinate [below right = 5pt and 5pt] (bra2) at (a2);
	\coordinate [below left = 5pt and 5pt] (bla2) at (a2);
	\draw[dotted, very thick] (arb3) -- (alb3) -- (bla2) -- (bra2) -- (arb3);
	
	\coordinate [above right = 5pt and 5pt] (arb8) at (b8);
	\coordinate [above left = 5pt and 5pt] (alb8) at (b8);
	\coordinate [below right = 5pt and 5pt] (bra7) at (a7);
	\coordinate [below left = 5pt and 5pt] (bla7) at (a7);
	\draw[dotted, very thick] (arb8) -- (alb8) -- (bla7) -- (bra7) -- (arb8);
	
	\coordinate [label=above:{$\beta_0$}] () at ($(arb3)!0.5!(alb3)$);
	\coordinate [label=below:{$\alpha_0$}] () at ($(bra2)!0.5!(bla2)$);
	\coordinate [label=above:{$\beta_{r-1}$}] () at (b4);
	\coordinate [label=below:{$\alpha_r$}] () at (a4);
	\coordinate [label=above:{$\beta_k$}] () at ($(arb8)!0.5!(alb8)$);
	\coordinate [label=below:{$\alpha_k$}] () at ($(bra7)!0.5!(bla7)$);
	
	\coordinate [label=left:{$\alpha =\alpha_s$}] () at ($(a5) + (-2pt, 0)$);
	
	\foreach \i in {a1, a2, a3, a4, a5, a6, a7, a8, a9, b1, b2, b3, b4, b5, b6, b7, b8, b9} {
		\coordinate [left = \side] (l\i) at (\i);
		\coordinate [right = \side] (r\i) at (\i);
	}
	
	\coordinate [below = 1.5cm] (oPms) at (la5);
	
	\coordinate [above = 1.5cm] (oPmt) at (lb7);
	\draw[[-{Latex[length=3mm]}, dashed, thick] (oPms) -- node[left]{$\calP_{\ell}$}  (la5) -- (lb6) -- (la6) -- (lb7) --(oPmt);
	
	\coordinate [below = 1.5cm] (nPms) at (ra5);
	\draw[[-{Latex[length=3mm]}, thick] (nPms) -- (ra5) -- (rb5) -- (ra4);
	
	\draw[[-{Latex[length=3mm]}, dashed, thick] (lb1) -- node[above left = 0 and -4pt]{$\calP_m$} ($(lb1)!0.5!(la1)$);
	\coordinate [label=above:{$\beta(\calP_m)$}] () at (b1);
	\draw[[-{Latex[length=3mm]}, dashed, thick] ($(lb9)!0.4!(la9)$) -- node[above left = 0 and -4pt]{$\calP_m$} ($(lb9)!0.9!(la9)$);
	\coordinate [label=below:{$\alpha(\calP_m)$}] () at (a9);

	\end{tikzpicture}
	\caption{
		Modification in \cref{subsubsec:not simple case 1};
		the definitions of all lines and paths are the same as in Figure~\ref{fig:cycle}.
	}
	\label{fig:non simple 2}
\end{figure}

\begin{lem}\label{lem:augmenting space-walk not simple case 2}
	$\calT'$ is an augmenting space-walk for $I'$.
\end{lem}
\begin{proof}
	($\Ai$) is obvious.
	By $X_r^\vee \not\subseteq X_r = \kerL(A_{\beta_{r-1} \alpha_r})$,
	we have ($\Al$).
	The suffix $\calP_{\ell}(\alpha] \circ \calP_m^\vee[\alpha_s, \alpha_r]$ forms a single outer space-walk for $I'$,
	since
	$\alpha_s \beta_{s-1}, \beta_{s-1} \alpha_{s-1}, \dots, \beta_r \alpha_r$ are distinct.
	We verify that the prefix $\calP_0 \circ \calQ_1 \circ \cdots \circ \calQ_{\ell} \circ \calP_{\ell}(\alpha]$
	is viewed as a compatibly-concatenated space-walk for $I'$.
	This implies that $\calT'$ is an augmenting space-walk for $I'$.

	By the definitions of $P_{\ell}$ and $\alpha$,
	the prefix $\calP_0 \circ \calQ_1 \circ \cdots \circ \calQ_{\ell} \circ \calP_{\ell}(\alpha]$ does not intersect with $\alpha_r, \alpha_{r+1}, \dots, \alpha_k$ except for $\alpha$.
	By the irredundancy,
	$\calP_0 \circ \calQ_1 \circ \cdots \circ \calQ_{\ell} \circ \calP_{\ell}(\alpha]$ does not meet $\alpha_0 = \alpha_k$ and $\beta_0 = \beta_k$.
	Hence it suffices to consider the case where there is $\ell' < \ell$ such that
	$P_{\ell'}$ intersects with $\alpha_1, \alpha_2, \dots, \alpha_{r-1}, \alpha_{k+1}, \dots, \alpha(\calP_m)$.
	Note that $P_{\ell'} \cap P_m$ forms the disjoint union of several subwalks of $P_{\ell'}$.
	By the same argument as (4) in \cref{lem:T' cycle case 1},
	no edge in $P_{\ell'} \cap P_m[\beta_k)$ is deleted from $I$ via the elimination operation.
	By (2) in \cref{lem:T' cycle case 1},
	each subwalk in $P_{\ell'} \cap P_m$ must consist of rank-2 edges.
	Hence it holds $P_{\ell'} \cap P_m = (P_{\ell'} \cap P_m[\beta_0, \beta_{r-1}]) \cup (P_{\ell'} \cap P_m[\beta_k))$.

	Take a subpath $P$ belonging to $P_{\ell'} \cap P_m$.
	If no edge in $P$ is deleted from $I$ via the elimination operation,
	then,
	by the same argument as in~(2) in \cref{lem:T' cycle case 1},
	the corresponding sequence with spaces to $P$
	forms an inner space-walk for $I'$,
	or constitutes a part of an inner space-walk for $I'$.
	If an edge in $P$ is deleted from $I$ via the elimination operation,
	then
	$P$ belongs to $P_{\ell'} \cap P_m[\beta_0, \beta_{r-1}]$;
	we can assume $P = (\alpha_i \beta_i, \beta_i \alpha_{i+1}, \dots, \alpha_j \beta_j)$ for $1 \leq i < j \leq r-1$.
	Since there is no rank-1 edge in $P$,
	all $\beta_i \alpha_{i+1}, \beta_{i+1} \alpha_{i+2}, \dots, \beta_{j-1} \alpha_j$ are deleted from 
	$I$ via the elimination operation.
	Hence $P$ remains a part of an outer space-walk for $I'$.
\end{proof}

Let $(I, \calT) \leftarrow (I', \calT')$.
This update can be done in $O(|E|)$ time,
since we can find the maximum index $r$ with $1 \leq r \leq k$ such that $\beta_{r-1} \alpha_r$ is rank-1 and such $\calP_{\ell}$ in~\eqref{eq:T' not simple case 2}
in $O(|E|)$ time.
Also
$\theta$ strictly decreases
by the same argument as in~\cref{subsubsec:not simple case 1}.
Return to the initial stage (\cref{subsec:initial stage}).

\section{Concluding remarks}\label{sec:discussion}
This article provides the first combinatorial blow-up-free algorithm for Edmonds' problem
for a $(2 \times 2)$-type generic partitioned matrix.
We end this paper with the following remarks.

\paragraph{Bit-complexity.}
We verify that,
in the case of $\F = \Q$,
the required bit-size during the algorithm
is polynomially bounded.
Without loss of generality,
we assume that each entry of $\A$ is an integer.

Consider the algorithm for finding an augmenting space-walk.
During the algorithm,
a 1-dimensional vector subspace $Z\subseteq \F^2$
is represented as a nonzero vector $z \in Z$.
In the initial phase,
for each $\alpha \beta \in I$ such that $\alpha \beta$ is rank-1 and $\deg_I(\beta) = 1$,
we can take an integer nonzero vector $y_\beta \in \kerR(\A)$ with the bit-length bounded in a polynomial of the bit-size of $\A$.
In the update phase,
we compute $X^{\perp_{\alpha \beta}}$ and $Y^{\perp_{\alpha \beta}}$ for $X \subseteq U_\alpha$ and $Y \subseteq V_\beta$,
respectively.
This can be simulated as follows.
Here we only consider the case of computing $X^{\perp_{\alpha \beta}}$.
Suppose $\rank \A = 1$ and that we have an integer nonzero vector $x \in X$ at hand.
Then $X^{\perp_{\alpha \beta}} = V_\beta$ if $x \in \kerL(\A)$,
and $X^{\perp_{\alpha \beta}} = \kerR(\A)$ if $x \not\in \kerL(\A)$.
Suppose $\rank \A = 2$
and
$\A = \left[
\begin{array}{cc}
a & b\\
c & d
\end{array}
\right]$ and $x =
\left[\begin{array}{c}
s\\
t
\end{array}
\right]$.
Then a nonzero vector $y =
\left[\begin{array}{c}
-(cs + dt)\\
as + bt
\end{array}
\right]$
belongs to $X^{\perp_{\alpha \beta}}$.
By
$\log(|cs + dt|) = \log |c| + \log |s|$ if $dt = 0$,
$\log(|cs + dt|) = \log |d| + \log |t|$ if $cs = 0$,
and $\log(|cs + dt|) \leq \log\left(|csdt|\left(1/|cs| + 1/|dt|\right)\right) \leq \log 2 + \log |c| + \log |d| + \log |s| + \log |t|$,
we have $\bit(y) = \bit(\A) + \bit(x) + O(1)$,
where $\bit(\cdot)$ is the bit-length of the argument.
Hence the bit-length is polynomially bounded in finding an augmenting space-walk.

The case of the augmentation procedure is similar.
Thus, in the whole process,
the bit-size is polynomially bounded.

\paragraph{On general generic partitioned matrices.}
A {\it generic partitioned matrix}~\cite{SIMAA/IIM94} is a matrix $A$ of the form~\eqref{eq:A}
with an $m_\alpha \times n_\beta$ matrix $\A$ over $\F$ for $\alpha$ and $\beta$.
Iwata and Murota~\cite{SIMAA/IM95} observed $\rank A \neq \ncrank A$ for a generic partitioned matrix $A$ in general;
in particular,
they gave such a matrix consisting only of $2 \times 2$ and $3 \times 2$ blocks.
It is known~\cite{SIAGA/H19}
that Edmonds' problem is equivalent to the problem of computing the rank of a generic partitioned matrix.

Our matching concept can be easily generalized to one for a generic partitioned matrix.
We hope that this matching notion can lead to a blow-up-free algorithm for a generic partitioned matrix
such that its rank and nc-rank coincide,
and hence to a polynomial-time algorithm for Edmonds' problem.

\paragraph{A simpler and faster algorithm.}
There exists a combinatorial $O(\mu\nu \min\{ \mu, \nu \})$-time
algorithm for Edmonds' problem for a $(2 \times 2)$-type generic partitioned matrix of the form~\eqref{eq:A},
which is faster than the algorithm proposed in this paper.
An outline of the faster algorithm is as follows:
\begin{enumerate}
	\item
	Compute an augmenting walk $T := (\beta_1\alpha_1, \alpha_1\beta_2, \dots, \beta_k\alpha_k)$ (without spaces) by the breath-first search version of the algorithm in \cref{subsec:finding}.
	\item
	For $i = k, k-1,\dots, 2$,
	update $I$ as
	\begin{align*}
	I \leftarrow
	\begin{cases}
	I \cup \{\beta_i \alpha_i\} & \text{if $\alpha_i$ is incident only to $\alpha_i\beta_{i-1}$ in $I$ and $\alpha_i \beta_{i-1}$ is rank-2},\\
	I \setminus \{\alpha_i\beta_{i-1}\} & \text{if $\alpha_i$ is incident to $\beta_i\alpha_i$ and $\alpha_i \beta_{i-1}$ in $I$},\\
	(I \cup \{ \alpha_i \beta_i \}) \setminus \{\alpha_i \beta_{i-1}\}& \text{otherwise}.
	\end{cases}
	\end{align*}
	Then add $\beta_1 \alpha_1$ to $I$.
	\item
	Apply the elimination operation to $I$ in \cref{subsubsec:elimination},
	and output the resulting $I$.
\end{enumerate}
In fact, our algorithm in \cref{sec:augmentation}
coincides with this simpler algorithm
if
($\No$) and ($\Ni$) always hold
and redundant eliminations are omitted.
The proof of the validity of this algorithm is considerably more complicated.
Hence we presented the current version of the algorithm.

\section*{Acknowledgments}
We thank the anonymous reviewers of for their helpful comments.
The authors were supported by JSPS KAKENHI Grant Number JP17K00029.
The first author was supported by JST PRESTO Grant Number JPMJPR192A, Japan.
The second author was supported by
JSPS KAKENHI Grant Numbers JP19J01302, 20K23323, 20H05795, Japan.
This is a post-peer-review, pre-copyedit version of an article published in Mathematical Programming.
The final authenticated version is available online at: \url{https://doi.org/10.1007/s10107-021-01676-5}.

\end{document}